\documentclass[reqno,a4paper]{amsart}

\usepackage{amsthm}
\usepackage{amsmath}
\usepackage[all]{xy} \SelectTips{eu}{}
\usepackage{latexsym,amsfonts,amssymb}
\usepackage{hyperref}
\hypersetup{colorlinks=true, linkcolor=teal}
\usepackage{cite}
\hypersetup{citecolor=brown}
\usepackage{leftidx}
\usepackage[dvipsnames,svgnames,table]{xcolor}
\usepackage{mathptmx}
\usepackage{nicefrac}
\usepackage{array}
\usepackage{mathrsfs}
\usepackage{rotating}
\usepackage{tikz-cd}

\usepackage{vmargin}
\setpapersize{A4}
\setmargrb{25mm}{20mm}{25mm}{20mm}

\renewcommand{\theequation}{$\smash{\sharp}\mspace{0.5mu}$\arabic{equation}}

\newcommand{\numberseries}{\mdseries}   %Fontseries used for numbering theorem

\newlength{\thmtopspace}                %Space above theorem
\newlength{\thmbotspace}                %Space below theorem
\newlength{\thmheadspace}               %Space between theorem caption and text
\newlength{\thmindent}                  %For indenting

\newtheoremstyle{bfupright head,slanted body}
                {\thmtopspace}{\thmbotspace}
                {\slshape}{\thmindent}{\bfseries}{.}{\thmheadspace}
                {{\numberseries \thmnumber{{\bf #2} }}\thmnote{#3}}

\newtheoremstyle{bfupright head,upright body}
                {\thmtopspace}{\thmbotspace}
                {\upshape}{\thmindent}{\bfseries}{.}{\thmheadspace}
                {{\numberseries \thmnumber{{\bf #2} }}\thmnote{#3}}

\newtheoremstyle{bfit head,upright body}
                {\thmtopspace}{\thmbotspace}
                {\upshape}{\thmindent}{\upshape}{.}{\thmheadspace}
                {{\numberseries\thmnumber{{\bf #2} }}
                {\bfseries\itshape\thmnote{\negthickspace#3}}}

\newtheoremstyle{it head,upright body}
                {\thmtopspace}{\thmbotspace}
                {\upshape}{\thmindent}{\upshape}{.}{\thmheadspace}
                {{\numberseries\thmnumber{{\bf #2} }}
                {\itshape\thmnote{\negthickspace#3}}}

\newtheoremstyle{fixed bf head,slanted body}
                {\thmtopspace}{\thmbotspace}{\slshape}
                {\thmindent}{\bfseries}{\!.}{\thmheadspace}
                {{\thmname{#1}\numberseries\thmnumber{ {\bf #2}}}\thmnote{ (#3)} }

\newtheoremstyle{fixed bf head,upright body}
                {\thmtopspace}{\thmbotspace}{\upshape}
                {\thmindent}{\bfseries}{\!.}{\thmheadspace}
                {{\thmname{#1}\numberseries\thmnumber{ {\bf #2}}}\thmnote{ (#3)} }

\newtheoremstyle{independent paragraph}
                {\thmtopspace}{\thmbotspace}
                {\upshape}{\thmindent}{\upshape}{}{0pt}
                {\thmnote{#3 }}

\newtheoremstyle{subparagraph}
                {\thmbotspace}{\thmbotspace}
                {\upshape}{\thmindent}{\upshape}{}{0pt}
                {\thmnote{#3 }}

\newtheoremstyle{notes}
                {\thmtopspace}{\thmbotspace}
                {\ttfamily}{\thmindent}{\ttfamily\small }{}{0pt}
                {\thmnote{#3 }}
                
\newtheoremstyle{numbered paragraph}
                {\thmtopspace}{\thmbotspace}{\upshape}
                {\thmindent}{\upshape}{}{\thmheadspace}
                {{\numberseries \thmnumber{\bf #2.}}}

\theoremstyle{bfupright head,slanted body}
\newtheorem{res}{}[section]             \newtheorem*{res*}{}

\theoremstyle{bfit head,upright body}
                 \newtheorem*{com*}{}

\theoremstyle{bfupright head,upright body}
               \newtheorem*{bfhpg*}{}

\theoremstyle{it head,upright body}
               \newtheorem*{ithpg*}{}

\theoremstyle{fixed bf head,slanted body}
\newtheorem{thm}[res]{Theorem}          \newtheorem*{thm*}{Theorem}
\newtheorem{prp}[res]{Proposition}      \newtheorem*{prp*}{Proposition}
\newtheorem{cor}[res]{Corollary}        \newtheorem*{cor*}{Corollary}
\newtheorem{lem}[res]{Lemma}            \newtheorem*{lem*}{Lemma}

\theoremstyle{fixed bf head,upright body}
\newtheorem{dfn}[res]{Definition}       \newtheorem*{dfn*}{Definition}
      \newtheorem*{obs*}{Observation}
\newtheorem{rmk}[res]{Remark}           \newtheorem*{rmk*}{Remark}
\newtheorem{exa}[res]{Example}          \newtheorem*{exa*}{Example}
         \newtheorem*{exe*}{Exercise}
\newtheorem{stp}[res]{Setup}            \newtheorem{stp*}{Setup}
         \newtheorem{ntn*}{Notation}
\newtheorem{con}[res]{Construction}     \newtheorem{con*}{Construction}
      \newtheorem*{conj*}{Conjecture}

\theoremstyle{numbered paragraph}

\theoremstyle{subparagraph}

\theoremstyle{notes}

\newlength{\thmlistleft}        %leftmargin
\newlength{\thmlistright}       %rightmargin
\newlength{\thmlistpartopsep}   %partopsep
\newlength{\thmlisttopsep}      %topsep
\newlength{\thmlistparsep}      %parsep
\newlength{\thmlistitemsep}     %itemsep

\newcounter{eqc} 
\newenvironment{eqc}{\begin{list}{\upshape (\textit{\roman{eqc}})}%
    {\usecounter{eqc}%
      \setlength{\leftmargin}{\thmlistleft}%
      \setlength{\labelwidth}{\thmlistleft}%
      \setlength{\rightmargin}{\thmlistright}%
      \setlength{\partopsep}{\thmlistpartopsep}%
      \setlength{\topsep}{\thmlisttopsep}%
      \setlength{\parsep}{\thmlistparsep}%
      \setlength{\itemsep}{\thmlistitemsep}}}%
  {\end{list}}%

\newcommand{\eqclbl}[1]{{\upshape(\textit{#1})}}

\newcounter{prt}
\newenvironment{prt}{\begin{list}{\upshape (\alph{prt})}%
    {\usecounter{prt}%
      \setlength{\leftmargin}{\thmlistleft}%
      \setlength{\labelwidth}{\thmlistleft}%
      \setlength{\rightmargin}{\thmlistright}%
      \setlength{\partopsep}{\thmlistpartopsep}%
      \setlength{\topsep}{\thmlisttopsep}%
      \setlength{\parsep}{\thmlistparsep}%
      \setlength{\itemsep}{\thmlistitemsep}}}%
  {\end{list}}%

\newcommand{\prtlbl}[1]{{\upshape(#1)}}

\newcounter{rqm}
\newenvironment{rqm}{\begin{list}{\upshape (\arabic{rqm})}%
    {\usecounter{rqm}%
      \setlength{\leftmargin}{\thmlistleft}%
      \setlength{\labelwidth}{\thmlistleft}%
      \setlength{\rightmargin}{\thmlistright}%
      \setlength{\partopsep}{\thmlistpartopsep}%
      \setlength{\topsep}{\thmlisttopsep}%
      \setlength{\parsep}{\thmlistparsep}%
      \setlength{\itemsep}{\thmlistitemsep}}}%
  {\end{list}}%

\newcounter{rqmm}
  {\end{list}}%

  {\end{list}}%

\newenvironment{itemlist}{\nopagebreak \begin{list}{{\small $\bullet$}}%
    {\setlength{\leftmargin}{\thmlistleft}%
      \setlength{\labelwidth}{\thmlistleft}%
      \setlength{\rightmargin}{\thmlistright}%
      \setlength{\partopsep}{\thmlistpartopsep}%
      \setlength{\topsep}{\thmlisttopsep}%
      \setlength{\parsep}{\thmlistparsep}%
      \setlength{\itemsep}{\thmlistitemsep}}}%
  {\end{list}}%

  \newcommand{\proofoftag}[2][:]{(#2)#1}

  \newcommand{\proofofimp}[3][:]{\mbox{\eqclbl{#2}$\!\implies\!$\eqclbl{#3}#1}}  

\newcommand{\pgref}[1]{\ref{#1}}

\renewcommand{\eqref}[1]{(\pgref{eq:#1})}

\newcommand{\corref}[2][Corollary ]{#1\pgref{cor:#2}}
\newcommand{\dfnref}[2][Definition~]{#1\pgref{dfn:#2}}
\newcommand{\exaref}[2][Example ]{#1\pgref{exa:#2}}
\newcommand{\lemref}[2][Lemma ]{#1\pgref{lem:#2}}

\newcommand{\prpref}[2][Proposition ]{#1\pgref{prp:#2}}

\newcommand{\rmkref}[2][Remark ]{#1\pgref{rmk:#2}}
\newcommand{\thmref}[2][Theorem ]{#1\pgref{thm:#2}}
\newcommand{\stpref}[2][Setup~]{#1\pgref{stp:#2}}

\newcommand{\conref}[2][Construction~]{#1\ref{con:#2}}

\makeatletter
\def\@nobreak@#1{\mathchoice%
  {\nobreakdef@\displaystyle\f@size{#1}}%
  {\nobreakdef@\nobreakstyle\tf@size{\firstchoice@false #1}}%
  {\nobreakdef@\nobreakstyle\sf@size{\firstchoice@false #1}}%
  {\nobreakdef@\nobreakstyle\ssf@size{\firstchoice@false #1}}%
  \check@mathfonts}%
\def\nobreakdef@#1#2#3{\hbox{{%
                    \everymath{#1}%
                    \let\f@size#2\selectfont%
                    #3}}}%
\makeatother

\DeclareFontFamily{T1}{cmex}{}
\DeclareFontShape{T1}{cmex}{m}{n}{<-> s * [0.89] cmex10}{}
\DeclareSymbolFont{cmlargesymbols}{T1}{cmex}{m}{n}
\DeclareMathSymbol{\mycoprod}{\mathop}{cmlargesymbols}{"60} 
\DeclareMathSymbol{\myprod}{\mathop}{cmlargesymbols}{"51} \let\prod\myprod

\DeclareSymbolFont{usualmathcal}{OMS}{cmsy}{m}{n}
\DeclareSymbolFontAlphabet{\mathcal}{usualmathcal}

\DeclareSymbolFont{letters}{OML}{txmi}{m}{it}

\DeclareMathSymbol{\alpha}{\mathord}{letters}{"0B}
\DeclareMathSymbol{\beta}{\mathord}{letters}{"0C}
\DeclareMathSymbol{\gamma}{\mathord}{letters}{"0D}
\DeclareMathSymbol{\sigma}{\mathord}{letters}{"0E}
\DeclareMathSymbol{\epsilon}{\mathord}{letters}{"0F}
\DeclareMathSymbol{\zeta}{\mathord}{letters}{"10}
\DeclareMathSymbol{\eta}{\mathord}{letters}{"11}
\DeclareMathSymbol{\theta}{\mathord}{letters}{"12}
\DeclareMathSymbol{\iota}{\mathord}{letters}{"13}
\DeclareMathSymbol{\kappa}{\mathord}{letters}{"14}
\DeclareMathSymbol{\lambda}{\mathord}{letters}{"15}
\DeclareMathSymbol{\mu}{\mathord}{letters}{"16}
\DeclareMathSymbol{\nu}{\mathord}{letters}{"17}
\DeclareMathSymbol{\xi}{\mathord}{letters}{"18}
\DeclareMathSymbol{\pi}{\mathord}{letters}{"19}
\DeclareMathSymbol{\rho}{\mathord}{letters}{"1A}
\DeclareMathSymbol{\sigma}{\mathord}{letters}{"1B}
\DeclareMathSymbol{\tau}{\mathord}{letters}{"1C}
\DeclareMathSymbol{\upsilon}{\mathord}{letters}{"1D}
\DeclareMathSymbol{\phi}{\mathord}{letters}{"1E}
\DeclareMathSymbol{\chi}{\mathord}{letters}{"1F}
\DeclareMathSymbol{\psi}{\mathord}{letters}{"20}
\DeclareMathSymbol{\omega}{\mathord}{letters}{"21}
\DeclareMathSymbol{\varepsilon}{\mathord}{letters}{"22}
\DeclareMathSymbol{\vartheta}{\mathord}{letters}{"23}
\DeclareMathSymbol{\varpi}{\mathord}{letters}{"24}
\DeclareMathSymbol{\varrho}{\mathord}{letters}{"25}
\DeclareMathSymbol{\varsigma}{\mathord}{letters}{"26}
\DeclareMathSymbol{\varphi}{\mathord}{letters}{"27}

\DeclareMathSymbol{\Gamma}{\mathord}{letters}{"00}
\DeclareMathSymbol{\Delta}{\mathord}{letters}{"01}
\DeclareMathSymbol{\Theta}{\mathord}{letters}{"02}
\DeclareMathSymbol{\Lambda}{\mathord}{letters}{"03}
\DeclareMathSymbol{\Xi}{\mathord}{letters}{"04}
\DeclareMathSymbol{\Pi}{\mathord}{letters}{"05}
\DeclareMathSymbol{\Upsilon}{\mathord}{letters}{"07}
\DeclareMathSymbol{\Phi}{\mathord}{letters}{"08}
\DeclareMathSymbol{\Psi}{\mathord}{letters}{"09}
\DeclareMathSymbol{\upGamma}{\mathalpha}{operators}{"00}
\DeclareMathSymbol{\upDelta}{\mathalpha}{operators}{"01}
\DeclareMathSymbol{\upTheta}{\mathalpha}{operators}{"02}
\DeclareMathSymbol{\upLambda}{\mathalpha}{operators}{"03}
\DeclareMathSymbol{\upXi}{\mathalpha}{operators}{"04}
\DeclareMathSymbol{\upPi}{\mathalpha}{operators}{"05}
\DeclareMathSymbol{\upSigma}{\mathalpha}{operators}{"06}
\DeclareMathSymbol{\upUpsilon}{\mathalpha}{operators}{"07}
\DeclareMathSymbol{\upPhi}{\mathalpha}{operators}{"08}
\DeclareMathSymbol{\upPsi}{\mathalpha}{operators}{"09}
\DeclareMathSymbol{\upOmega}{\mathalpha}{operators}{"0A}

\DeclareMathAlphabet\PazoBB{U}{fplmbb}{m}{n}%

\newcommand{\Hom}[3]{\operatorname{Hom}_{#1}(#2,#3)}
\newcommand{\Ext}[4]{\operatorname{Ext}_{#1}^{#2}(#3,#4)}
\newcommand{\Tor}[4]{\operatorname{Tor}^{#1}_{#2}(#3,#4)}

\newcommand{\Ker}[1]{\operatorname{Ker}\mspace{1mu}#1}
\newcommand{\Coker}[1]{\operatorname{Cok}\mspace{1mu}#1}

\newcommand{\Prj}[1]{\operatorname{Prj}\mspace{2mu}#1}
\newcommand{\Prjp}[1]{\operatorname{Prj}(#1)}
\newcommand{\GPrj}[1]{\operatorname{GPrj}\mspace{2mu}#1}

\newcommand{\Cy}[2]{\operatorname{Z}^{#1}\mspace{-2.5mu}#2}

\newcommand{\Hletterone}{\mathcal{H}}
\newcommand{\HH}[3]{\Hletterone^{#1}_{#2,\,#3}}
\newcommand{\HHop}[3]{\check{\Hletterone}^{#1}_{#2,\,#3}}

\newcommand{\Hlettertwo}{\mathbb{H}}

\newcommand{\sHH}[3]{%
\Hlettertwo^{#1}_{#2,\,#3}}
\newcommand{\sHHt}[3]{%
\Hlettertwo_{#1}^{#2,\,#3}}

\newcommand{\cHhj}[2]{\Hlettertwo^{#1}_{\textnormal{[}#2\textnormal{]}}}
\newcommand{\hHhj}[2]{\Hlettertwo_{#1}^{\textnormal{[}#2\textnormal{]}}}

\newcommand{\stalkco}[1]{S\mspace{-2mu}\langle{#1}\rangle}
\newcommand{\stalkcn}[1]{S\mspace{-2mu}\{#1\}}

\newcommand{\Cq}[1]{C_{\mspace{-1mu}\smash{#1}}}
\newcommand{\Sq}[1]{S_{\mspace{-4mu}\smash{#1}}}
\newcommand{\Kq}[1]{K_{\smash{#1}}}
\newcommand{\Fq}[1]{F_{\mspace{-2mu}\smash{#1}}}
\newcommand{\Eq}[1]{E_{\smash{#1}}}
\newcommand{\Gq}[1]{G_{\smash{#1}}}

\newcommand{\smasharrow}[2][\to]{\stackrel{\text{\raisebox{3pt}{\smash{$#2$}}}}{\smash{#1}}}

\newcommand{\lMod}[1]{{}_{#1}\mspace{-1mu}\operatorname{Mod}}
\newcommand{\rMod}[1]{\operatorname{Mod}_{#1}}
\newcommand{\lMode}[2]{{}_{#1}\mspace{-1mu}\operatorname{Mod}^{\,#2}}

\newcommand{\lPrj}[1]{{}_{#1}\mspace{-1mu}\operatorname{Prj}}
\newcommand{\lprj}[1]{{}_{#1}\mspace{-1mu}\operatorname{prj}}
\newcommand{\lInj}[1]
{{}_{#1}\mspace{-1mu}\operatorname{Inj}}

\newcommand{\lsmod}[1]{{}_{#1}\mspace{-1mu}\operatorname{mod}}

\DeclareMathOperator{\Inj}{Inj}
\newcommand{\exact}[1]{{}_{#1}\mathscr{E}}

\newcommand{\supp}[1]{\operatorname{supp}\mspace{1mu}#1}

\newcommand{\QSD}[2]{\mathbf{D}_{#1}(#2)}
\newcommand{\QSH}[2]{\mathbf{K}_{#1}(#2)}

\usepackage{microtype}
\usepackage{indentfirst}

\begin{document}

\title{Relative $Q$-shaped homological algebra}

%\author{Henrik Holm, Anastasios Slaftsos, Jorge Vit\'oria}
%

\author{Anastasios Slaftsos}

\address{Anastasios Slaftsos, Dipartimento di Matematica ``Tullio Levi-Civita'', Universit\'a degli Studi di Padova, via Trieste 63, 35131 Padova, Italy}
\email{slaftsos@math.unipd.it}

\author{Jorge Vit\'oria}

\address{Jorge Vit\'oria, Dipartimento di Matematica ``Tullio Levi-Civita'', Universit\'a degli Studi di Padova, via Trieste 63, 35131 Padova, Italy
} 
\email{jorge.vitoria@unipd.it}

%
%\urladdr{http://www.math.ku.dk/\~{}holm/}



%16G20 Representations of quivers and partially ordered sets
%16E35 Derived categories and associative algebras
%18G25 Relative homological algebra, projective classes 
%18G80 Derived categories, triangulated categories
%18N40 Homotopical algebra, Quillen model categories, derivators

\thanks{\textbf{Acknowledgements:} The authors would like to thank Henrik Holm for his significant contribution on this work. In addition, the authors thank Henrik Holm and Peter J\o rgensen for kindly allowing them to access their unpublished ideas concerning the characterisation of Gorenstein projective objects in $\lMod{Q,A}$ with respect to the abelian exact structure, which inspired \thmref[Theorems~]{GPrj} and \thmref[]{GInj}. 
A.S.~was supported by NextGenerationEU, Mission 4 Component 1 Investment 4.1~, CUP C96E23000600001.
J.V.~was supported by NextGenerationEU under NRRP, Call PRIN 2022  No.~104 of February 2, 2022 of the Italian Ministry of Universities and Research; Project 2022S97PMY \textit{Structures for Quivers, Algebras and Representations (SQUARE)}. }

\begin{abstract}
Exact categories are a natural generalisation of abelian categories and provide a fertile ground to develop relative homological algebra. In this paper, starting from a class of relative Gorenstein projective objects in an exact category $(\mathcal{A},\exact{})$, we define exact model structures on $\mathcal{A}$ and cohomology functors that detect trivial objects and weak equivalences.  Moreover, we show that varying the exact structure on $\mathcal{A}$ induces Bousfield (co)localisation sequences between the corresponding homotopy categories. We use these techniques to study the category $\lMod{Q,A}$ of $\lMod{A}$-valued representations, for a ring $A$, of a suitable $\Bbbk$-linear small category $Q$, where we apply our results to a range of objectwise exact structures, ranging from the split exact structure to the abelian one. In particular, we recover the $Q$-shaped derived category of Holm and J{\o}rgensen and construct an intermediate $Q$-shaped homotopy category, analogous to the homotopy category of complexes. Finally, we show that the $Q$-shaped derived category is a Verdier quotient of the $Q$-shaped homotopy category, and that this quotient functor is part of recollement - generalising results of Verdier, Krause, and Iyama-Kato-Miyachi for complexes and $N$-complexes, respectively.
\end{abstract}

\maketitle

\section{Introduction}
\label{sec:introduction}

A complex of left $A$-modules, for an algebra $A$ over a commutative ring $\Bbbk$ can be regarded as an $\lMod{A}$-valued representation of the repetitive quiver $\mathbb{Z}A_2$ of the Dynkin quiver, modulo the mesh relations. In other words, is a representation of the quiver $Q^\text{cpx}$ below, subject to the relation that every two consecutive compose to zero. 
\[
Q^\text{cpx}\colon\cdots\longrightarrow\bullet\longrightarrow\bullet\longrightarrow\bullet\longrightarrow\cdots
\]
Formally speaking, this observation witnesses an equivalence between the category of complexes $\mathbf{Ch}(A)$ and the functor category $\lMod{Q^{\text{cpx}},A}$ of $\Bbbk$-linear functors $Q^{\text{cpx}}\longrightarrow\lMod{A}$, where $Q^{\text{cpx}}$ is the $\Bbbk$-linear path category of $Q^\text{cpx}$ modulo the ideal generated by the composition of any two consecutive arrows.

This point of view allows one to consider various ``shapes'' $Q$ that admit the same properties as $Q^\text{cpx}$, and to study the categories $\lMod{Q,A}$ as variants of the category of chain complexes. The allowed shapes include (but are not restricted to) some cyclic quivers and some repetitive quivers (more generally stable translation quivers). A key tool in our approach to the study of the $\lMod{A}$-valued representation theory of $Q$ is the existence of a Serre functor in $Q$, which excludes finite quivers without oriented cycles. Inspired by the model structure on $\mathbf{Ch}(A)$ that gives rise to the derived category $\mathbf{D}(A)$ as its homotopy category, in \cite{HJ-JLMS, HJ-TAMS} the following statements are proved.
\begin{itemlist}
    \item There are classes $\mathbb{E},{}^\perp\mathbb{E}$ and $\mathbb{E}^\perp$ in $\lMod{Q,A}$ that generalise the exact, DG projective and DG injective complexes and give rise to two abelian model structures on $\lMod{Q,A}$ which in turn generalise the projective and injective model structures on complexes.
    \item These model structures share the same homotopy category, denoted by $\QSD{Q}{A}$ and called the \textbf{$Q$-shaped derived category}. In the special case $Q=Q^\text{cpx}$, we recover the usual derived category $\mathbf{D}(A)$ of $A$.
    \item The $Q$-shaped derived category is compactly generated.
\end{itemlist}

There is, however, an intermediate step between the category of complexes of $A$-modules and the derived category $\mathbf{D}(A)$: the homotopy category of complexes of $A$-modules $\mathbf{K}(A)$. In fact, the derived category $\mathbf{D}(A)$ can be expressed as a Verdier quotient of the homotopy category of complexes (cf. \cite[Chap.~III,~Def.~1.2.2]{JLV77}) and, as shown in \cite{Krause_2010}, this quotient functor $j^*$ fits in a recollement as follows
\[
\begin{tikzcd}[column sep=6em]
	{\mathbf{K}^{\text{ac}}(R)} & {\mathbf{K}(R)} & {\mathbf{D}(R).}
	\arrow["i_*"{description}, from=1-1, to=1-2]
	\arrow["{i^*}"', shift right=3, from=1-2, to=1-1]
	\arrow["{i^!}", shift left=3, from=1-2, to=1-1]
	\arrow["j^*"{description}, from=1-2, to=1-3]
	\arrow["{j_*}", shift left=3, from=1-3, to=1-2]
	\arrow["{j_!}"', shift right=3, from=1-3, to=1-2]
\end{tikzcd}
\]
The kernel of $j^*$ is the full subcategory $\mathbf{K}^{\text{ac}}(A)$ of acyclic complexes of $A$-modules. In this paper, we strive for the construction of an analogous intermediate step between the categories $\lMod{Q,A}$ and $\QSD{Q}{A}$, constructing a category which we will call the \emph{Q-shaped homotopy category} of $A$. To achieve this goal, we develop a machinery that produces projective (or injective) model structures on (efficient) exact categories (see for example \cite{SaorinStovicek}), recovering in the process the construction of the $Q$-shaped derived category as well. 

We begin this paper working with abstract exact categories, later specialising to the functor category $\lMod{Q,A}$. To state our main theorem (\thmref{main}), we consider a set (or, in some cases, a class) of Gorenstein projective objects in a weakly idempotent complete exact category $\mathcal{E}$ and we consider the cohomology functors defined as follows
\begin{equation*}
  \HH{i}{G}{n}(-) \;=\; \Ext{\mathcal{E}}{i}{\sigma^n G}{-}
  \colon \mathcal{E} \longrightarrow \mathsf{Ab}.
\end{equation*}
Here, $\sigma^nG$ denotes the $n$-th syzygy of a totally acyclic complex of projective objects in $\mathcal{E}$ which has as its zero-th syzygy the Gorenstein projective object $G$. It turns out that the functor $\HH{i}{G}{n}$ is independent of the choice of the totally acyclic complex, see \lemref{proj-equiv}. In this context, we show that $\mathcal{E}$ is equipped with a model structure for which we give a cohomological description of trivial objects and weak equivalences. In Theorem \ref{thm:main-op} we show the dual results by working with a class of Gorenstein injectives.

\begin{res*}[Theorem~A]
     
  \label{thm:A}
     Let $\mathscr{G}$ be a class of Gorenstein projective objects in a weakly idempotent complete exact category $\mathcal{E}$ with enough projectives and consider the subcategories: 
   \[
  \mathcal{W}_{\mathscr{G}}=\{W \in \mathcal{E} \ | \ \HH{1}{G}{n}(W)=0 \text{ for every } G \in \mathscr{G} \text{ and } n \in \mathbb{Z} \} \quad \text{and} \quad
  \mathcal{C}_{\mathscr{G}}={}^\perp \mathcal{W}_\mathscr{G}
  \]
If the cotorsion pair $(\mathcal{C}_\mathscr{G},\mathcal{W}_\mathscr{G})$ is complete (for example, when $\mathcal{E}$ is efficient and $\mathscr{G}$ is a set), then there is an hereditary exact model structure on $\mathcal{E}$ such that $\mathcal{C}_{\mathscr{G}}$ is the class of cofibrant objects, $\mathcal{W}_{\mathscr{G}}$ is the class of trivial objects and every object in $\mathcal{E}$ is fibrant. Furthermore, in that case, we have
\begin{enumerate}
\item An object $X$ is trivial if and only if $\HH{i}{G}{n}(X)=0$ for every $G$ in $\mathscr{G}$, $n$ and $i$ integers, with $i>0$.;
\item Given a morphism $\varphi$ in $\mathcal{E}$ the following conditions are equivalent:
\begin{eqc}
\item[(i)] $\varphi$ is a weak equivalence;
\item[(ii)] $\HH{i}{G}{n}(\varphi)$ is an isomorphism for every $G$ in $\mathscr{G}$, $n$ and $i$ integers, with $i>0$;
\item[(iii)] \smash{$\HH{1}{G}{n}(\varphi)$} and \smash{$\HH{2}{G}{n}(\varphi)$} are isomorphisms for every $G$ in $\mathscr{G}$ and $n$ integer.
\end{eqc}
\item $\mathcal{C}_{\mathscr{G}}$ is a Frobenius exact category whose projective-injective objects are precisely $\Prj{\mathcal{E}}$, and the homotopy category of the this model structure $\operatorname{Ho}(\mathcal{E})$ is equivalent to the stable category $\underline{\mathcal{C}}_\mathscr{G}$.
\end{enumerate}
\end{res*}

The functor category $\lMod{Q,A}$ is a Grothendieck abelian category, and thus an efficient exact category. In this category, we will consider not only the abelian exact structure but also \textbf{objectwise exact structures}. Indeed, if we endow $\lMod{A}$ with an exact structure $\exact{}$, there is an exact structure $\exact{Q}$ on $\lMod{Q,A}$ whose conflations are those exact sequences that are sent to conflations in $\lMod{A}$ when evaluated on any object $q$ in $Q$. Notable examples of such exact structures are the abelian exact structure, the objectwise pure exact structure and the objectwise split exact structure (\exaref{Prj-PPrj-Mod}). In this exact $Q$-shaped world, the role of the class of Gorenstein projective objects from \thmref{main} is played by stalk representations $\Sq{q}(T)$ where $T$ is a relative projective object with respect to the exact structure $\exact{}$ on $\lMod{A}$. In this case \thmref{main} can be expressed as follows (see Theorem \ref{thm:Q-shaped-relative-Prj} and Theorem \ref{thm:Q-shaped-relative-Inj} for the dual).

\begin{res*}[Theorem~B]
Let $\exact{}$ be an exact structure in $\lMod{A}$ such that $(\lMod{A},\exact{})$ has enough projectives and exact coproducts and let $\mathscr{T}$ be a class of projective objects in $(\lMod{A},\exact{})$. Consider in the exact category $(\lMod{Q,A},\exact{Q})$ the cotorsion pair $(\mathcal{C}(\mathscr{T}),\mathcal{W}(\mathscr{T}))$ generated by the class
\begin{equation*}
  \mathscr{G} \,=\, \mathscr{G}(\mathscr{T}) 
  \,=\, \{ \Sq{q}(T) \;|\; T \in \mathscr{T},\; q \in Q \} \;.
\end{equation*}
If this cotorsion pair is complete, then there is a hereditary exact model structure on $(\lMod{Q,A},\exact{Q})$ where the classes $\mathcal{W}(\mathscr{T})$ of trivial objects and of weak equivalences are described in parts (1) and (2) below, the class of cofibrant objects is $\mathcal{C}(\mathscr{T}) = {}^\perp(\mathcal{W}(\mathscr{T}))$, and every object in $(\lMod{Q,A},\exact{Q})$ is fibrant. 
\begin{rqm}
\item An object $X$ in $\lMod{Q,A}$ is trivial if and only if it satisfies the equivalent conditions:
\begin{eqc}
\setlength{\itemsep}{0pt}
\item $\HH{i}{\Sq{q}(T)}{n}(X)=0$ for all $T$ in $\mathscr{T}$, $q$ in $Q$, $n$ integer and $i>0$;

\item[\eqclbl{i$\,'$}] $\sHH{i}{\stalkco{q}}{n}(\Hom{A}{T}{X})=0$ for all $T$ in $\mathscr{T}$, $q$ in $Q$, $n$ integer and $i>0$;

\item $\HH{1}{\Sq{q}(T)}{0}(X)=0$ for all $T$ in $\mathscr{T}$ and $q$ in $Q$;

\item[\eqclbl{ii$\,'$}] $\sHH{1}{\stalkco{q}}{0}(\Hom{A}{T}{X})=0$ for all $T$ in $\mathscr{T}$ and $q$ in $Q$.
\end{eqc}

\item  A morphism $\varphi$ in $\lMod{Q,A}$ is a weak equivalence if and only if it satisfies the equivalent conditions:

\begin{eqc}
\setlength{\itemsep}{0pt}
\item $\HH{i}{\Sq{q}(T)}{n}(\varphi)$ is an isomorphism for all $T$ in $\mathscr{T}$, $q$ in $Q$, $n$ integer and $i>0$;.

\item[\eqclbl{i$\,'$}] $\sHH{i}{\stalkco{q}}{n}(\Hom{A}{T}{\varphi})$ is an isomorphism for all $T$ in $\mathscr{T}$, $q$ in $Q$, $n$ integer and $i>0$;

\item $\HH{1}{\Sq{q}(T)}{0}(\varphi)$ and $\HH{2}{\Sq{q}(T)}{0}(\varphi)$ are isomorphisms for all $T$ in $\mathscr{T}$ and $q$ in $Q$;

\item[\eqclbl{ii$\,'$}] $\sHH{1}{\stalkco{q}}{0}(\Hom{A}{T}{\varphi})$ and $\sHH{2}{\stalkco{q}}{0}(\Hom{A}{T}{\varphi})$ are isomorphisms for all $T$ in $\mathscr{T}$ and $q$ in $Q$.
\end{eqc}

\item $\mathcal{C}(\mathscr{T})$ is a Frobenius exact category whose class of projective-injective objects is precisely the class 
\begin{equation*}
  \Prj{(\lMod{Q,A},\exact{Q})} \;=\; 
  \operatorname{Add}\,\{\,\Fq{q}(T) \;|\; T \in \Prj{(\lMod{A},\exact{})},\;  q \in Q \,\} \;,
\end{equation*}
and the homotopy category $\operatorname{Ho}(\lMod{Q,A},\exact{Q})$ is equivalent to the stable category $\underline{\mathcal{C}}(\mathscr{T})$.
\end{rqm}
The functors $\Hletterone$ appearing in conditions of type \eqclbl{i} and  \eqclbl{ii} are cohomology functors in  $(\lMod{Q,A},\exact{Q})$, while the functors $\mathbb{H}$ appearing in conditions of type \eqclbl{i$\,'$} and \eqclbl{ii$\,'$} are cohomology functors in the abelian category $\lMod{Q}$.
\end{res*}

If we consider the abelian exact structure in $\lMod{Q,A}$, we recover in Proposition \ref{prp:assumptions-ok-Prj} and 
Example \ref{exa:Q-shaped derived category revisited} the $Q$-shaped derived category and the results of \cite[Thm.~7.1]{HJ-JLMS}. On the other hand, if we consider the objectwise split exact structure in $\lMod{Q,A}$, then the homotopy category obtained by the Theorem \ref{thm:Q-shaped-relative-Prj}, is what we call the \textbf{$Q$-shaped homotopy category} of $A$, which we will denote by $\mathbf{K}_Q(A)$ (see Definition \ref{dfn:q-shaped-homotopy}).

Note that exact structures in $\lMod{A}$ can be ordered by the inclusion of the classes of conflations. The minimal element for this partial order is clearly the split exact structure, while the maximal element is the abelian structure of $\lMod{A}$.  These exact structures give rise to a collection of partially ordered exact structures in $\lMod{Q,A}$. They range from the objectwise split exact structure in $\lMod{Q,A}$ to the abelian exact structure in $\lMod{Q,A}$. Inclusions of exact structures induce Bousfield (co)localisations between the associated homotopy categories. The following is a combination of Proposition \ref{prp:Ho-verdier-abstract} and \ref{prp:Ho-verdier-abstract-op}.

\begin{res*}[Theorem~C]
  Let $\exact{}_2 \subseteq \exact{}_1$ be two exact structures with enough projectives and enough injectives on an additive category $\mathcal{A}$.  Given a hereditary projective Hovey triple $(\mathcal{C}_\ell,\mathcal{W}_\ell,\mathcal{A})$  and a hereditary injective Hovey triple $(\mathcal{A},\mathcal{W}_\ell,\mathcal{F}_\ell)$ in the exact category $(\mathcal{A},\exact{}_\ell)$ with $\mathcal{W}_1 \supseteq \mathcal{W}_2$, then the following statements hold.
  \begin{prt}
  \item  $\mathcal{W}_1$, viewed as a full subcategory of $\operatorname{Ho}(\mathcal{A},\exact{}_2)$, is a thick triangulated subcategory; 
  \item The functor
  \(
      I_* \colon \operatorname{Ho}(\mathcal{A},\exact{}_2) \longrightarrow \operatorname{Ho}(\mathcal{A},\exact{}_1)
  \)
  induced by the homotopical functor $I \colon (\mathcal{A},\exact{}_2) \longrightarrow (\mathcal{A},\exact{}_1)$, can be identified with the Verdier localisation functor $V \colon \operatorname{Ho}(\mathcal{A},\exact{}_2) \longrightarrow \operatorname{Ho}(\mathcal{A},\exact{}_2)/\mathcal{W}_1$;
  \item   If we have $\mathcal{C}_1\cap\mathcal{W}_1\subseteq\mathcal{C}_2\cap\mathcal{W}_2$ and $\mathcal{F}_1\cap\mathcal{W}_1\subseteq\mathcal{F}_2\cap\mathcal{W}_2$, then there is a recollement of triangulated categories
  \[\begin{tikzcd}
	{\mathcal{W}_1} && {\operatorname{Ho}(\mathcal{A},\exact{}_2)} && {\operatorname{Ho}(\mathcal{A},\exact{}_1)}
	\arrow["J"{description}, from=1-1, to=1-3]
	\arrow["{J_r}", shift left=3, from=1-3, to=1-1]
	\arrow["{J_\ell}"', shift right=3, from=1-3, to=1-1]
	\arrow["{I_*}"{description}, from=1-3, to=1-5]
	\arrow["{I^r}", shift left=3, from=1-5, to=1-3]
	\arrow["{I^\ell}"', shift right=3, from=1-5, to=1-3]
\end{tikzcd}.\]
  \end{prt}
\end{res*}

In the special case where one considers the inclusion of the objectwise split exact structure of $\lMod{Q,A}$ into the abelian exact structure,  Theorem \ref{thm:Q-shaped-verider-loc} generalises results by Verdier \cite[Chap.~III,~Def.~1.2.2]{JLV77} and Krause \cite[Ex.~4.14]{Krause_2010} for complexes, and by Iyama, Kato and Miyachi in \cite[Thm.~3.17]{MR3742439} for $N$-complexes.

\begin{res*}[Theorem~D]
    The following statements hold.
    \begin{prt}
        \item The induced functor $I_*\colon \QSH{Q}{A}\longrightarrow\QSD{Q}{A}$ is a Verdier localisation;
        \item The functor $I_*\colon \QSH{Q}{A}\longrightarrow\QSD{Q}{A}$ admits a fully faithful left adjoint functor $\mathbf{c}\colon\QSD{Q}{A}\longrightarrow\QSH{Q}{A}$ and a fully faithful right adjoint functor $\mathbf{f}\colon\QSD{Q}{A}\longrightarrow\QSH{Q}{A}$, which map a $Q$-shaped module to its cofibrant replacement and fibrant replacement, respectively;
        \item Denote by $\mathbf{K}_Q^{\text{ac}}(A)=\mathcal{W}({\theta_1})$. Then we have a recollement of triangulated categories of the form
    \[
    \begin{tikzcd}[column sep=6em]
	{\mathbf{K}^{\text{ac}}_Q(A)} & {\QSH{Q}{A}} & {\QSD{Q}{A}.}
	\arrow["j"{description}, from=1-1, to=1-2]
	\arrow["{\ell}"', shift right=3, from=1-2, to=1-1]
 \arrow["{r}", shift left=3, from=1-2, to=1-1]
	\arrow["I_*"{description}, from=1-2, to=1-3]
	\arrow["{\mathbf{c}}"', shift right=3, from=1-3, to=1-2]
 \arrow["{\mathbf{f}}", shift left=3, from=1-3, to=1-2]
\end{tikzcd}
    \]
    \end{prt}
\end{res*}

\subsection*{Notation and conventions}
\label{notation}

Unless otherwise stated, all subcategories are strict and full. If $\mathcal{A}$ is an additive category admitting direct limits, we say that an object $X$ of $\mathcal{A}$ is \textbf{finitely presented} if $\Hom{\mathcal{A}}{X}{-}$ commutes with direct limits. For a set of objects $\mathcal{X}$ of an additive category $\mathcal{A}$, we denote by $\operatorname{add}\mathcal{X}$ (respectively, $\operatorname{Add}\mathcal{X}$) the subcategory of $\mathcal{A}$ whose objects are summands of finite (respectively, set-indexed) coproducts of objects in $\mathcal{X}$.  For any class $\mathcal{L}$ of objects in an exact category $\mathcal{E}$, and denoting by $\Ext{\mathcal{E}}{1}{-}{-}$ the pairing associating to two objects the Yoneda extensions (with respect to the exact structure) between them, we consider the following subcategories
\begin{align*}
  {}^{\perp}\mathcal{L}
  &\;=\;
  \{X \in \mathcal{E} \ |\, \Ext{\mathcal{E}}{1}{X}{L}=0 \text{ for every } L \in \mathcal{L} \}
  \\
  \mathcal{L}^{\perp}
  &\;=\;
  \{Y \in \mathcal{E} \ |\, \Ext{\mathcal{E}}{1}{L}{Y}=0 \text{ for every } L \in \mathcal{L} \}\;
\end{align*}

The word \textit{module} will stand for \textit{left module}. For a ring $A$, we denote by $\lMod{A}$ the category of left $A$-modules. For a left $A$-module $M$ (or more generally, for an object $M$ in an abelian category $\mathcal{A}$) we denote by $\operatorname{pd}M$ its projective dimension and by  $\operatorname{id}M$ its injective dimension. Unless otherwise stated, $\Bbbk$ will denote a commutative ring.

\subsection*{Structure of the paper}
 We begin in Section \ref{sec:preliminaries} preliminaries concerning (efficient) exact categories and cotorsion pairs therein. We discuss exact model categories and their connection to complete cotorsion pairs in exact categories as established by Gillespie \cite{MR2811572}, generalising the results of Hovey \cite{Hovey02} in the abelian case. Finally, we review the notions of Gorenstein projectivity (and injectivity) in these categories. 

In Section \ref{sec:projective_model_str} starting from a class of relative Gorenstein projective (or injective) objects in an exact category, we develop the main machinery of this paper. In particular, assuming that this class forms a set and the exact category is nice enough (for example it is efficient), we define a projective (or injective) model structure and cohomology functors that detect the trivial objects and weak equivalences. %Moreover, we show that these model structures induce a Frobenius model in the corresponding subcategory the cofibrant (or fibrant) objects whose stable category is the homotopy category arising from the exact model structure.

In Section \ref{sec:GProj-Proj-in-Q,A-Mod} we shift our attention to the category $\lMod{Q,A}$ of $\Bbbk$-linear functors from $Q$ to $\lMod{A}$, endowed with an exact structure $\exact{Q}$. The main goal is to ensure that $(\lMod{Q,A},\exact{Q})$ fits into the framework of Section \ref{sec:projective_model_str} and particularly of Theorem \ref{thm:main}. For this purpose, we generalise known results for the abelian case from \cite{HJ-JLMS} and \cite{HJ-TAMS} to the exact setting. In particular, we describe exactness of functors between the exact categories $(\lMod{Q,A},\exact{Q})$ and $(\lMod{A},\exact{})$ and we characterise the relative (Gorenstein) projective and relative (Gorenstein) injective objects in $(\lMod{Q,A},\exact{Q})$ with respect to the relative (Gorenstein) projective and relative (Gorenstein) injective modules in $(\lMod{A},\exact{})$. Finally, we provide examples of exact structures in $\lMod{Q,A}$ induced by classes of modules in $\lMod{A}$. These examples range from the objectwise split exact structure to the abelian exact structure.

In Section \ref{sec:Q-shaped derived category}, we interpret Theorems \ref{thm:main} and \ref{thm:main-op} in the special case of the exact category $(\lMod{Q,A},\exact{Q})$. In this case, the class of Gorenstein projective (or Gorenstein injective) objects turns out to be stalk representations $S_q(T)$, where $T$ is a relative projective (or injective) module in $(\lMod{A},\exact{})$. If $\exact{}$ is the abelian exact structure, we recover some results from \cite{HJ-JLMS} and the construction of the $Q$-shaped derived category of $A$.

Finally, in Section \ref{sec:q-shaped homotopy}, we discuss how the inclusion of two exact structures in our setup induces a comparison between their associated homotopy categories. Concretely, we see how the inclusion of exact structures in an additive category can give rise to a well-defined functor between exact model categories respecting weak equivalences and, consequently, to a functor (which turns out to be a Bousfield (co)localisation) between the associated homotopy categories. We apply this machinery to the inclusion of the objectwise split exact structure into the abelian exact structure in $\lMod{Q,A}$, showing that the $Q$-shaped derived category $\QSD{Q}{A}$ is a (co)localisation of the $Q$-shaped homotopy category $\QSH{Q}{A}$, with kernel consisting of the trivial objects in $\QSH{Q}{A}$ (with respect to the abelian exact structure). This generalises results by \cite{JLV77}, \cite{Krause_2010} and \cite{PJr05c} for complexes, and \cite{MR3742439} for $N$-complexes.

\section{Preliminaries}\label{sec:preliminaries}

\subsection{Efficient exact categories} 
\label{efficient}
\phantom{.} \vspace*{1ex}

We develop our main machinery in the framework of exact categories. We refer to the survey of B{\"u}hler \cite{Buhler} for notation and classic definitions regarding exact categories. Recall that an additive category $\mathcal{E}$ is called
\begin{itemlist}
\item \textbf{weakly idempotent complete}, if every retraction has a kernel; equivalently, if every coretraction has a cokernel (see \cite[Lem.~7.1]{Buhler}). If $\mathcal{E}$ is an exact category (see below), then weakly idempotent completeness of the underlying additive category can also be expressed in terms of the exact structure on $\mathcal{E}$ (see \cite[Cor.~7.5]{Buhler}).
\item \textbf{idempotent complete}, if every idempotent endomorphism has a kernel (see \cite[Def.~6.1 and Rem.~6.2]{Buhler}). 
\end{itemlist}
It is easy to observe that idempotent complete additive categories are also weakly idempotent complete. The converse holds if $\mathcal{E}$ admits countable coproducts (see \cite{Freyd66} and \cite[Rem.~7.3]{Buhler}).

An \textbf{exact category} $(\mathcal{E},\mathcal{S})$ is an additive category $\mathcal{E}$ together with a distinguished class $\mathcal{S}$ of kernel-cokernel pairs
\[
0\longrightarrow X\stackrel{i}{\longrightarrow}Y\stackrel{p}{\longrightarrow} Z\longrightarrow 0
\qquad (\,\textnormal{also written }
  \xymatrix@C=1.2pc{
  X \ar@{>->}[r]^-{i} & 
  Y \ar@{->>}[r]^-{p} &
  Z
  }\!)
\] 
called \textbf{conflations}. They satisfy certain axioms which make conflations behave similarly to short exact sequences in an abelian category. In this notation, the morphism $i$ is called an \textbf{inflation}, while the morphism $p$ is called a \textbf{deflation}. We denote an exact category by just $\mathcal{E}$ if $\mathcal{S}$ is implied. In such an exact category we say that an object $X$ is \textbf{projective} (respectively, \textbf{injective}) in $\mathcal{E}$ if $\Ext{\mathcal{E}}{1}{X}{-}=0$ (respectively, if $\Ext{\mathcal{E}}{1}{-}{X}=0$) and we denote the subcategory of projective objects by $\Prj{\mathcal{E}}$ (respectively, the subcategory of injective objects by $\Inj{\mathcal{E}}$). A class $\mathcal{L}$ of objects in an exact category $\mathcal{E}$ is \textbf{generating} if for every $X$ in $\mathcal{E}$ there exists a deflation $L \twoheadrightarrow X$ with $L$ in $\mathcal{L}$. Dually, $\mathcal{L}$ is \textbf{cogenerating} if for every $X$ in $\mathcal{E}$ there exists an inflation $X \rightarrowtail L$ with $L$ in $\mathcal{L}$. Let $\Prj{\mathcal{E}}$ be the class of projective objects and $\Inj{\mathcal{E}}$ the class of injective objects in $\mathcal{E}$, see \cite[Def.~11.1 and Rmk.~11.2]{Buhler}. One says that $\mathcal{E}$ has \textbf{enough projectives} (respectively, \textbf{enough injectives}) if $\Prj{\mathcal{E}}$ is generating (respectively, if $\Inj{\mathcal{E}}$ is cogenerating).
%and further
%\begin{align*}
%  {}^{\perp_\infty}\mathcal{L}
%  &\;=\;
%  \{X \in \mathcal{E} \ |\, \Ext{\mathcal{E}}{i}{X}{L}=0 \text{ for every } L \in \mathcal{L} \text{ and } i>0 \}
%  \\
%  \mathcal{L}^{\perp_\infty}
%  &\;=\;
%  \{Y \in \mathcal{E} \ |\, \Ext{\mathcal{E}}{i}{L}{Y}=0 \text{ for every } L \in \mathcal{L} \text{ and } i>0 \}\;.
%\end{align*}
Note that, for any class $\mathcal{L}$ of objects in $\mathcal{E}$, there are always inclusions $\Prj{\mathcal{E}} \subseteq {}^{\perp}\mathcal{L}$  and $\operatorname{Inj}{\mathcal{E}} \subseteq \mathcal{L}^{\perp}$.

Throughout, we will work with exact categories which are  \textbf{efficient}. This concept was introduced in \cite[Def.~2.6]{SaorinStovicek} to describe a class of exact categories with desirable properties. We avoid the rather technical definition of efficient exact categories, but we note some of the good properties of such categories that we will use throughout and highlight sufficient conditions for an exact category to be efficient.
\begin{itemlist}

\item As shown in \cite[Lem.~1.4]{SaorinStovicek}, an efficient exact category $\mathcal{E}$ has \emph{exact coproducts}, i.e.~ arbitrary (set-indexed) coproducts exist in $\mathcal{E}$ and the coproduct of any family of conflations in $\mathcal{E}$ is a conflation. In particular, efficient exact categories which are weakly idempotent complete are, in fact, idempotent complete (as noted above).

\item If $\mathcal{E}$ is an efficient exact category, then $\Ext{\mathcal{E}}{1}{X}{Y}$ is a set, for all $X$ and $Y$ in $\mathcal{E}$ (see \cite[Rem.~2.4]{SaorinStovicek}). 

\item If $\mathcal{E}$ is an idempotent complete locally finitely presented exact category (i.e., $\mathcal{E}$ admits all direct limits, its subcategory of finitely presented objects, $\mathsf{fp}(\mathcal{E})$, is skeletally small, and every object of $\mathcal{E}$ is the direct limit of a directed system in $\mathsf{fp}(\mathcal{E})$) having enough injectives, then $\mathcal{E}$ is efficient (\cite[Rem.~1.7, Proposition~2.7]{SaorinStovicek}). 
%In particular, every idempotent complete locally finitely presented category with enough injectives is efficient.

\item Further examples of efficient exact categories can be found in \cite[Ex.~2.8]{SaorinStovicek}. They include the crucial examples that matter to us: Grothendieck abelian categories and the category of complexes $\mathbf{Ch}(\mathcal{A})$ of a Grothencieck abelian category $\mathcal{A}$ equipped with the degreewise split exact structure. 
\end{itemlist}

In an exact category one can define acyclic complexes, see \cite[Def.~10.1]{Buhler} or \cite[\S1]{MR1080854}. A complex, $(X^n,\partial^{n}_X)_{n\in\mathbb{Z}}$ in $\mathcal{E}$ is said to be \textbf{acyclic} if each differential $\partial_X^{n}$ admits a factorisation,
\begin{equation*}
  \xymatrix@!=0.8pc{
    X^{n} 
    \ar@{.>>}[dr]_-{p_n}
    \ar[rr]^-{\partial_X^{n}} 
    & & 
    X^{n+1}
    \\
    & 
    \Cy{n+1}{X} \;,\mspace{-8mu}
    \ar@{>.>}[ur]_-{i_{n+1}}
    &
  }
\end{equation*}
such that, for all $n$ in $\mathbb{Z}$,
\begin{equation*}
 \xymatrix@C=1.2pc{
  \Cy{n}{X} \ar@{>->}[r]^-{i_n} & 
  X^n \ar@{->>}[r]^-{p_n} &
  \Cy{n+1}{X}
  }\
\end{equation*}
is a conflation. Note that by \cite[Rmk.~10.2]{Buhler} the object $\Cy{n}{X}$ is a kernel of $\partial_X^n \colon X^n \to X^{n+1}$, an image of $\partial_X^{n-1} \colon X^{n-1} \to X^n$, and a cokernel of $\partial_X^{n-2} \colon X^{n-2} \to X^{n-1}$. 

%%%%%%%%%%%%%%%%%%%%%%%%%%%%%%%%%%%%%%%%%%%%%%%%%%%%%%%%%%

\subsection{Cotorsion pairs}
\label{cotorsion-pairs}
\phantom{.} \vspace*{1ex}

A \textbf{cotorsion pair} in $\mathcal{E}$ is a pair $(\mathcal{A},\mathcal{B})$ of classes of objects in $\mathcal{E}$ satisfying $\mathcal{A}^\perp = \mathcal{B}$ and $\mathcal{A} = {}^\perp\mathcal{B}$. To every class $\mathcal{L}$ of objects in $\mathcal{E}$ there are two associated cotorsion pairs,
\begin{equation*}
  \mathfrak{G}_\mathcal{L}
  \;=\;
  ({}^\perp(\mathcal{L}^{\perp}),\mathcal{L}^{\perp})
  \qquad \text{and} \qquad
  \mathfrak{C}_\mathcal{L}
  \;=\;
  ({}^\perp\mathcal{L},({}^\perp\mathcal{L})^\perp)\;.
\end{equation*}
We call $\mathfrak{G}_\mathcal{L}$ the cotorsion pair \textbf{generated} by $\mathcal{L}$ and $\mathfrak{C}_\mathcal{L}$ the cotorsion pair \textbf{cogenerated}\footnote{\,This terminology is the one used in\cite[Def.~2.2.1]{GobelTrlifaj} and \cite[\S1.3]{SaorinStovicek}. Beware that some authors use the ``opposite'' terminology, see for example \cite[Def.~7.1.2]{rha}.}~by~$\mathcal{L}$.
A cotorsion pair $(\mathcal{A},\mathcal{B})$ in $\mathcal{E}$ is said to be \textbf{complete} if the following conditions hold:
\begin{rqm}
\item The class $\mathcal{A}$ is \textbf{special precovering} in $\mathcal{E}$, i.e.~for every $X$ in $\mathcal{E}$ there exists a conflation $B \rightarrowtail A \twoheadrightarrow X$ in $\mathcal{E}$ with $A$ in $\mathcal{A}$ and $B$ in $\mathcal{B}$. 
\item The class $\mathcal{B}$ is \textbf{special preenveloping} in $\mathcal{E}$, i.e.~for every $X$ in $\mathcal{E}$ there exists a conflation $X \rightarrowtail B' \twoheadrightarrow A'$ in $\mathcal{E}$ with $A'$ in $\mathcal{A}$ and $B'$ in $\mathcal{B}$. 
\end{rqm}
Sometimes, condition (1) is expressed by saying that the cotorsion pair $(\mathcal{A},\mathcal{B})$ \textbf{has enough projectives} and condition (2) by saying that $(\mathcal{A},\mathcal{B})$ \textbf{has enough injectives}, see \cite[Def.~7.1.5]{rha} and \cite[\S2.1]{MR2811572}.

\begin{rmk}
  \label{rmk:complete}
Salce's trick (\cite{Salce}) shows that for a cotorsion pair $(\mathcal{A},\mathcal{B})$ in an exact category $\mathcal{E}$, if if $\mathcal{A}$ is generating in $\mathcal{E}$, then (2) implies (1), while if $\mathcal{B}$ is cogenerating in $\mathcal{E}$, then (1) implies (2) (see the proofs of \cite[Prop.~7.1.7]{rha} or \cite[Lem.~2.2.6]{GobelTrlifaj} for module categories and \cite[Thm.~2.13(5)]{SaorinStovicek} for exact categories). 
\end{rmk}

%\begin{rmk}
%  \label{rmk:complete2}
%Let $(\mathcal{A},\mathcal{B})$ be a cotorsion pair in $\mathcal{E}$.
%\begin{prt}
%\item If $(\mathcal{A},\mathcal{B})$ is complete, then evidently $\mathcal{A}$ is generating and $\mathcal{B}$ is cogenerating in $\mathcal{E}$.
%
%\item If $\mathcal{E}$ has enough projectives, then $\mathcal{A}$ is generating (as $\Prj{\mathcal{E}}$ is contained in $\mathcal{A}$). Dually, if $\mathcal{E}$ has enough injectives, then $\mathcal{B}$ is cogenerating (as $\Inj{\mathcal{E}}$ is contained in $\mathcal{B}$).
%
%\item A trick due to Salce \cite{Salce} shows that if $\mathcal{A}$ is generating and $\mathcal{B}$ is cogenerating, then conditions \rqmlbl{1} and \rqmlbl{2} above are, in fact, equivalent. More precisely:
%\begin{itemlist}
%\item If $\mathcal{A}$ is generating, then (2) implies (1).
%\item If $\mathcal{B}$ is cogenerating, then (1) implies (2).
%\end{itemlist}
%It follows by inspecting the proofs of \cite[Proposition~7.1.7]{rha} or \cite[Lemma~2.2.6]{GobelTrlifaj} (for module categories) and \cite[Thm.~2.13(5)]{SaorinStovicek} (for exact categories). 
%\end{prt}
%\end{rmk}

Note that the assumption in the next lemma is satisfied if $(\mathcal{A},\mathcal{B})$ is complete (in this paper, we are only interested in complete cotorsion pairs). In the generality stated below, the lemma is due to {\v{S}}{\v{t}}ov{\'{\i}}{\v{c}}ek, but a version of it for abelian categories can be found e.g.~in  \cite[Lems.~2.3 and 2.4]{MR3459032} and \cite[Lem.~4.25]{SaorinStovicek}.

\begin{lem}[{{\v{S}}{\v{t}}ov{\'{\i}}{\v{c}}ek \cite[Lem.~6.17]{Stovicek2013}}]
\label{lem:Stovicek}
Let $(\mathcal{A},\mathcal{B})$ be a cotorsion pair in a weakly idempotent complete exact category $\mathcal{E}$, with $\mathcal{A}$ being generating and $\mathcal{B}$ cogenerating in $\mathcal{E}$. The following conditions are equivalent:
\begin{eqc}
\item $\Ext{\mathcal{E}}{i}{A}{B}=0$ for every $A$ in $\mathcal{A}$, $B$ in $\mathcal{B}$, and $i>0$;
\item $\Ext{\mathcal{E}}{2}{A}{B}=0$ for every $A$ in $\mathcal{A}$ and $B$ in $\mathcal{B}$;
\item For every conflation $X' \rightarrowtail X \twoheadrightarrow X''$ in $\mathcal{E}$ with $X$ and $X''$ in $\mathcal{A}$ one has $X'$ in $\mathcal{A}$;
\item For every conflation $X' \rightarrowtail X \twoheadrightarrow X''$ in $\mathcal{E}$ with $X', X$ in $\mathcal{B}$ one has that $X''$ lies in $\mathcal{B}$.
\end{eqc}
\end{lem}

Following {\v{S}}{\v{t}}ov{\'{\i}}{\v{c}}ek \cite[Def.~6.18]{Stovicek2013} a cotorsion pair $(\mathcal{A},\mathcal{B})$ in $\mathcal{E}$ is called \textbf{hereditary} if $\mathcal{A}$ is generating, $\mathcal{B}$ is cogenerating, and the equivalent conditions in \lemref{Stovicek} hold. 

%%%%%%%%%%%%%%%%%%%%%%%%%%%%%%%%%%%%%%%%%%%%%%%%%%%%%%%%%%

\subsection{Exact model structures}
\label{exact-model-structures}
\phantom{.} \vspace*{1ex}

Model structures on categories were introduced by Quillen \cite[I\S1 Def.~1]{Qui67} as a way to formally invert certain classes of morphisms. A model structure consists of three classes of morphisms: cofibrations, weak equivalences and fibrations subject to certain axioms. We refer to Hovey's book \cite{Hovey02} for details. Given a model structure on a category with a zero object, one defines the following three classes of objects:
\begin{itemlist}
\item the class $\mathcal{C}$ of \textbf{cofibrant objects}, made of the objects $X$ for which $0\rightarrow X$ is a cofibration;
\item the class $\mathcal{W}$ of \textbf{trivial objects}, made of the objects $X$ for which $0\rightarrow X$ is a weak equivalence (equivalently, the objects $X$ for which $X\rightarrow 0$ is a weak equivalence);
\item the class $\mathcal{F}$ of \textbf{fibrant objects}, made of the objects $X$ for which $X\rightarrow 0$ is a fibration.
\end{itemlist}
These classes of object do not, in general, determine the model structure.

If we are working with a model structure in an abelian or, more generally, in an exact category $\mathcal{E}$, it is natural to require certain compatibility conditions between the model structure and the exact structure. Model structures subject to such compatibility conditions are called \textbf{abelian model structures} (see \cite{Hovey02}) or, more generally, \textbf{exact model structures} (see \cite{MR2811572} and  \cite{Stovicek2013}). We omit what these compatibility conditions are, but the key point for our purposes is that the classes of cofibrant, trivial and fibrant objects of an abelian or exact model structure completely determine the model structure, as stated in the following theorem (the abelian version is due to Hovey \cite[Thm.~2.2]{Hovey02}). This allows us to work with classes of objects rather than classes of morphisms. Moreover, in a \emph{hereditary} exact model category $\mathcal{E}$ the subcategory of fibrant-cofibrant (= bifibrant) objects $\mathcal{C}\cap \mathcal{F}$ is a Frobenius category whose projective-injective objects are $\mathcal{C}\cap \mathcal{W} \cap \mathcal{F}$, and there is a triangle equivalence between the homotopy category $\operatorname{Ho}(\mathcal{E})$ of the exact model category and the stable category $\underline{\mathcal{C}\cap\mathcal{F}}$ (see \cite[Prop.~5.2(4) and Cor.~5.4]{MR2811572}). Recall that the stable category $\underline{\mathcal{G}}$ of a Frobenius exact category $\mathcal{G}$ is the quotient category of $\mathcal{G}$ modulo the ideal of morphisms that factor through a projective-injective object in $\mathcal{G}$.

Recall that a full subcategory $\mathcal{X}$ of an exact category $\mathcal{E}$ is called \textbf{thick} if it is closed under direct summands and has the property that if two out of three of the terms in a conflation of $\mathcal{E}$ belong to $\mathcal{X}$, then so does the third.

\begin{thm}[{Gillespie \cite[Cor.~3.4]{MR2811572}}]
  \label{thm:Gillespie-CWF}
Let $\mathcal{E}$ be a weakly idempotent complete exact category. There is a bijection between exact model structures in $\mathcal{E}$ and triples of subcategories $(\mathcal{C},\mathcal{W},\mathcal{F})$ for which $\mathcal{W}$ is a thick subcategories and $(\mathcal{C}\cap\mathcal{W},\mathcal{F})$ and $(\mathcal{C},\mathcal{W}\cap\mathcal{F})$ are complete cotorsion pairs.
\end{thm}

Let us comment on the assignments. As the notation suggests, given an exact model structure on $\mathcal{E}$, one lets $\mathcal{C}$, $\mathcal{W}$, and $\mathcal{F}$ be the classes of cofibrant, trivial, and fibrant objects. It turns out that $\mathcal{W}$ is a thick subcategory of $\mathcal{E}$ and $(\mathcal{C},\mathcal{W} \cap \mathcal{F})$ and $(\mathcal{C} \cap \mathcal{W}, \mathcal{F})$ are complete cotorsion pairs in $\mathcal{E}$. Conversely, if $(\mathcal{C}, \mathcal{W}, \mathcal{F})$ is a triple of subcategories of $\mathcal{E}$ as in the theorem, then there exists an exact model structure on $\mathcal{E}$ given by:
\begin{itemlist}
\item A morphism $\varphi$ is a weak equivalence if and only if it admits a factorisation $\varphi = \pi\iota$
%\begin{equation*}
%  \xymatrix{
%    X \ar@{>->}[dr]_-{\iota} \ar[rr]^-{\varphi} & {} & Y
%    \\
%    {} & Z \ar@{->>}[ur]_-{\pi} & {}
%  }
%\end{equation*}  
where $\iota$ is an inflation with $\Coker{\iota} \in \mathcal{C} \cap \mathcal{W}$ and $\pi$ is an deflation with $\Ker{\pi} \in \mathcal{W} \cap \mathcal{F}$.
\item Cofibrations are inflations whose cokernel lies in $\mathcal{C}$.
\item Fibrations are deflations whose kernel lies in $\mathcal{F}$.
\end{itemlist}
Such triples $(\mathcal{C}, \mathcal{W}, \mathcal{F})$ are often dubbed \textbf{Hovey triples}. If the associated complete cotorsion pairs $(\mathcal{C},\mathcal{W} \cap \mathcal{F})$ and $(\mathcal{C} \cap \mathcal{W}, \mathcal{F})$ are both hereditary, then $(\mathcal{C}, \mathcal{W}, \mathcal{F})$ is called a \textbf{hereditary Hovey triple} and the corresponding exact model structure is called a \textbf{hereditary exact model structure}.

%%%%%%%%%%%%%%%%%%%%%%%%%%%%%%%%%%%%%%%%%%%%%%%%%%%%%%%%%%

\subsection{Gorenstein projective objects}
\label{Gorenstein-projective-objects}
\phantom{.} \vspace*{1ex}

We now review some basic properties about Gorenstein projective objects in exact categories. The properties recorded in Lemma \ref{lem:GPrj-Ext} and \prpref{GPrj-resolving} below are well-known in the category of modules over a ring, see \cite{HHl04a}. Similar properties have recently been established in the setting of extriangulated categories, a generalisation of exact categories, in \cite{HZZ}. We briefly recall these properties. 

We begin with the definition. Gorenstein projective modules over a ring were first defined in \cite{EEnOJn95b}. This definition can be nevertheless considered in any exact category $\mathcal{E}$. Throughout this subsection, we assume that $\mathcal{E}$ is an exact category with enough projectives.

\begin{dfn}
  \label{dfn:totally-acyclic}
  A \textbf{totally acyclic complex of projectives} in $\mathcal{E}$ is an acyclic complex (see Subsection \ref{efficient}),
\begin{equation*}
  P \,=\, 
  \xymatrix@C=1.6pc{
    \cdots \ar[r] & 
    P^{n-1} \ar[r]^-{\partial_P^{n-1}} &
    P^{n} \ar[r]^-{\partial_P^{n}} &
    P^{n+1} \ar[r] &
    \cdots
  },
\end{equation*}
where each $P^n$ is projective and, for any projective object $Q$ of $\mathcal{E}$, the complex $\Hom{\mathcal{E}}{P}{Q}$ of abelian groups is exact. An object $G$ of $\mathcal{E}$ is called \textbf{Gorenstein projective} if there exists a totally acyclic complex $P$ of projectives in $\mathcal{E}$ with $\Cy{0}{P} \cong G$. Note that in this case, $\Cy{n}{P}$ is Gorenstein projective for every $n \in \mathbb{Z}$. We denote the class of Gorenstein projective objects in $\mathcal{E}$ by $\GPrj{\mathcal{E}}$.
\end{dfn}

\begin{rmk}
  Let $G$ be a Gorenstein projective object in $\mathcal{E}$ and $P$ be a totally acyclic complex of projectives in $\mathcal{E}$ with $\Cy{0}{P} \cong G$. Further, let $X$ be any object in $\mathcal{E}$. We view $\Hom{\mathcal{E}}{P}{X}$ as a cochain complex of abelian groups by setting $\Hom{\mathcal{E}}{P}{X}^i = \Hom{\mathcal{E}}{P^{-i}}{X}$ for all $i$ in $\mathbb{Z}$. For every integer $n$, the sequence $\rho^n = \cdots \to P^{n-3} \to P^{n-2} \to P^{n-1} \to 0$, viewed as a complex with $P^{n-1}$ in degree zero, 
is a projective resolution of $\Cy{n}{P}$ in $\mathcal{E}$, so it can be used to compute $\Ext{\mathcal{E}}{*}{\Cy{n}{P}}{X}$. Since, in the complex $P$, $P^{n-1}$ sits in degree $n-1$, we have
\begin{equation}
  \label{eq:Ext-HHom}
  \Ext{\mathcal{E}}{i}{\Cy{n}{P}}{X}
  \;=\;
  \operatorname{H}^{i-n+1}\Hom{\mathcal{E}}{P}{X} 
  \qquad \text{for} \qquad
  n \in \mathbb{Z} \ \text{ and } \ i>0 \;.
\end{equation}
\end{rmk}

We begin with some easy properties of Gorenstein projective objects. To parse part \prtlbl{b} of the next result, recall that an object $M$ of an exact category is said to have \emph{finite projective dimension} if for some $n \geqslant 0$ there exists an acyclic complex, in the sense of Subsection \ref{efficient}, of the form $0 \to R_n \to \cdots \to R_0 \to M \to 0$ where $R_0,\ldots,R_n$ are projective objects. If no shorter such complex exists, then $n$ is the projective dimension of $M$.

\begin{lem}
\label{lem:GPrj-Ext}
Let $\mathcal{E}$ be an exact category with enough projectives. 
\begin{prt}
\item If $G$ is a Gorenstein projective object in $\mathcal{E}$, then for every object $M$ in $\mathcal{E}$ with finite projective dimension, and for all $i>0$, one has $\Ext{\mathcal{E}}{i}{G}{M} = 0$. 
\item If $\mathcal{E}$ has exact coproducts, then every coproduct of Gorenstein projective objects in $\mathcal{E}$ is Gorenstein projective.
\end{prt}
\end{lem}

\begin{proof}
\proofoftag{a} We prove this by induction on the projective dimension, say $n$, of $M$. If $n=0$ then $M$ is a projective object. Let $P$ be a totally acyclic complex of projectives in $\mathcal{E}$ with $\Cy{0}{P} \cong G$. By \dfnref{totally-acyclic}, the complex $\Hom{\mathcal{E}}{P}{M}$ is exact, so \eqref{Ext-HHom} with $n=0$ yields $\Ext{\mathcal{E}}{i}{G}{M} = 0$ for all $i>0$. Now let $n>0$ and assume that the assertion is true for objects of projective dimension $n-1$. Consider a conflation $M' \rightarrowtail P \twoheadrightarrow M$ with $P$ projective and, consequently,  with $M'$ having projective dimension $n-1$. The long exact sequence obtained by applying $\Hom{\mathcal{E}}{G}{-}$ to this conflation shows that $\Ext{\mathcal{E}}{i}{G}{M} = 0$ for all $i>0$ since $\Ext{\mathcal{E}}{i}{G}{P} = 0$ as $P$ is projective and $\Ext{\mathcal{E}}{i}{G}{M'} = 0$ by the induction hypothesis.

\proofoftag{b} Since $\mathcal{E}$ has exact coproducts, if $\{P_i\}_{i \in I}$ is a family of totally acyclic complexes of projectives in $\mathcal{E}$, then the coproduct $P = \bigoplus_{i \in I} P_i$ is an acyclic complex of projectives with $\Cy{0}{P} \cong \bigoplus_{i \in I} \Cy{0}{P_i}$ (\cite[Cor.~11.7]{Buhler}). To conclude, it suffices to see that, for every $Q$ in $\Prj{\mathcal{E}}$, the complex $\Hom{\mathcal{E}}{P}{Q} \cong \prod_{i \in I}\Hom{\mathcal{E}}{P_i}{Q}$ is exact.
\
\end{proof}

%\hh{I have added \lemref{GPrj-coproduct} and \prpref[Propositions~]{GPrj-resolving} and \prpref[]{GPrj-summand} as we need them in the proof of \thmref{GPrj}.}

%Following Fu, Herzog, Hu, and Zhu \cite[Definition~5.1]{MR4358620} we say that an exact category $\mathcal{E}$ \emph{has exact (co)products} if (set-indexed) (co)products exist in $\mathcal{E}$ and are exact in the sense that the (co)product of any family of conflations in $\mathcal{E}$ is a conflation. 

The following proposition is well-known in both abelian and triangulated settings. It has been proved in \cite{HZZ} in the context of extriangulated categories (these are additive categories endowed with a structure that simultaneously generalise exact and triangulated categories). As remarked in \cite{HZZ} the result was new in the context of exact categories, and we will use it throughout the paper.

\begin{prp}[\cite{HZZ}]
\label{prp:GPrj-resolving}
Let $\mathcal{E}$ be an exact category with enough projectives and let $\GPrj{\mathcal{E}}$ denote the class of Gorenstein projective objects in $\mathcal{E}$. The following assertions hold.
\begin{prt}
\item $\GPrj{\mathcal{E}}$ is closed under extensions.
\item $\GPrj{\mathcal{E}}$ is closed under direct summands.
\item $\GPrj{\mathcal{E}}$ is closed under kernels of deflations. 
\end{prt}
\end{prp}

\begin{proof}
Parts \prtlbl{a} and \prtlbl{b} can be found in \cite[Lem.~6 and Thm.~5(2)]{HZZ}. For \prtlbl{c} note that \cite[Cor.~3]{HZZ} shows that for any conflation $X' \rightarrowtail X \twoheadrightarrow X''$ which remains exact under the functor $\Hom{\mathcal{E}}{-}{Q}$ for any $Q$ in $\Prj{\mathcal{E}}$, one has that if any two of the objects $X',X$ and $X''$ belong to $\GPrj{\mathcal{E}}$ then so does the third. If we consider a conflation as above with $X$ and $X''$ in $\GPrj{\mathcal{E}}$, it follows from the proof of \lemref{GPrj-Ext}\prtlbl{a} that such a conflation remains exact when $\Hom{\mathcal{E}}{-}{Q}$ is applied to it, for any $Q$ in $\Prj{\mathcal{E}}$, so the result follows.
\end{proof}

Subcategories of an exact category satisfying the properties of Proposition \ref{prp:GPrj-resolving} are sometimes referred to as \textit{resolving subcategories}. Note, however, that some authors do not require the closure under direct summands in the definition of resolving subcategory.

The following result is also well-known; since we were not able to find a reference in the setting of exact categories, we have included a proof.

\begin{lem}
  \label{lem:proj-equiv}
  Let $\mathcal{E}$ be an exact category with enough projectives, let $G$ be a Gorenstein projective object in $\mathcal{E}$ and let $P$ and $\tilde{P}$ be totally acyclic complexes of projectives in $\mathcal{E}$ with $\Cy{0}{P} \cong G \,\cong \Cy{0}{\tilde{P}}$. For each $n$ in $\mathbb{Z}$, the Gorenstein projective objects $\Cy{n}{P}$ and $\Cy{n}{\tilde{P}}$ are projectively equivalent, that is, there exist projective objects $Q$ and $\tilde{Q}$ in $\mathcal{E}$ (which depend on $n$) such that $ \tilde{Q} \,\oplus\, \Cy{n}{P} \;\cong\; Q \,\oplus\, \Cy{n}{\tilde{P}}.$
\end{lem}

\begin{proof}
  For $n<0$ this follows immediately from Schanuel's Lemma  for exact categories \cite[(dual of) Prop.~3.4]{MR4697475} applied to the projective resolutions 
  \[\cdots \to P^{-2} \to P^{-1} \to G \to 0\quad \text{and}\quad \cdots \to \tilde{P}^{-2} \to \tilde{P}^{-1} \to G \to 0.\] 
 For $n \geqslant 0$, we prove the projective equivalence of $\Cy{n}{P}$ and $\Cy{n}{\tilde{P}}$ by induction on $n$. For $n=0$ there is nothing to prove since $\Cy{0}{P} \cong G \cong \Cy{0}{\tilde{P}}$. Now assume that $\Cy{n}{P}$ and $\Cy{n}{\tilde{P}}$ are  projectively equivalent for some $n \geqslant 0$ and let \mbox{$\varphi \colon \tilde{Q} \oplus \Cy{n}{P} \to Q \oplus \Cy{n}{\tilde{P}}$} be an isomorphism where $Q$ and $\tilde{Q}$ lie in  $\Prj{\mathcal{E}}$. If we add the trivial conflations
\begin{equation}
  \label{eq:conf1}
  \xymatrix@C=1.5pc{
  \tilde{Q}\, \ar@{>->}[r]^-{=} & 
  \tilde{Q} \ar@{->>}[r] &
  0
  }
  \qquad \text{and} \qquad
  \xymatrix@C=1.5pc{
  Q\, \ar@{>->}[r]^-{=} & 
  Q \ar@{->>}[r] &
  0
  }\phantom{,}
\end{equation}
to the given conflations 
\begin{equation}
  \label{eq:conf2}
  \xymatrix@C=1.5pc{
  \Cy{n}{P}\, \ar@{>->}[r] & 
  P^n \ar@{->>}[r] &
  \Cy{n+1}{P}
  }
  \qquad \text{and} \qquad
  \xymatrix@C=1.5pc{
  \Cy{n}{\tilde{P}}\, \ar@{>->}[r] & 
  \tilde{P}^n \ar@{->>}[r] &
  \Cy{n+1}{\tilde{P}}
  },
\end{equation}    
then by \cite[Prop.~2.9]{Buhler} we get new conflations, which are the rows in the diagram:
\begin{equation}
  \label{eq:conf3}
  \begin{gathered}
  \xymatrix{
  X \ar[d]^-{\varphi}_-{\cong} \ar@{>->}[r] & 
  R \ar@{.>}[d] \ar@{->>}[r] &
  \Cy{n+1}{P} \ar@{.>}[d]
  \\
  \tilde{X} \ar@{>->}[r] & 
  \tilde{R} \ar@{->>}[r] &
  \Cy{n+1}{\tilde{P}}\;.\mspace{-8mu}
  }
  \end{gathered}
\end{equation}
Here we have $X = \tilde{Q} \oplus \Cy{n}{P}$ and $\tilde{X} = Q \oplus \Cy{n}{\tilde{P}}$ and the projective objects $R =  \tilde{Q} \oplus P^n$ and \mbox{$\tilde{R} = Q \oplus \tilde{P}^n$}. For every $L$ in $\Prj{\mathcal{E}}$, the functor~$\Hom{\mathcal{E}}{-}{L}$ maps the conflations in \eqref{conf1} and \eqref{conf2}, and hence also the rows in \eqref{conf3}, to short exact sequences of abelian groups. In particular, $\Hom{\mathcal{E}}{-}{\tilde{R}}$ maps the inflation $X \rightarrowtail R$ to a surjection, and it follows that there exist morphisms (dotted arrows) that make the diagram \eqref{conf3} commutative. Applying \cite[(dual of) Prop.~2.12]{Buhler} to this commutative diagram, where $\varphi$ is an isomorphism, we get a conflation in $\mathcal{E}$ of the form:
\begin{equation*}
  \xymatrix@C=1.5pc{
  R \, \ar@{>->}[r] & 
  \tilde{R} \oplus \Cy{n+1}{P} \ar@{->>}[r] &
  \Cy{n+1}{\tilde{P}}
  }.
\end{equation*}   
where $R$ is projective and $\Cy{n+1}{\tilde{P}}$ is Gorenstein projective. Since $\Ext{\mathcal{E}}{1}{\Cy{n+1}{\tilde{P}}}{R}=0$ by \lemref{GPrj-Ext}\prtlbl{a}, this conflation splits, showing that $\tilde{R} \oplus \Cy{n+1}{P} \cong R \oplus \Cy{n+1}{\tilde{P}}$. Hence, $\Cy{n+1}{P}$ and $\Cy{n+1}{\tilde{P}}$ are projectively equivalent.
\end{proof}

The lemma above shows that for a given Gorenstein projective object $G$ in an exact category $\mathcal{E}$ with enough projectives, the projective-equivalence class of $\Cy{n}P$, where $P$ is a totally acyclic complex with $\Cy{0}P\cong G$ depends only on $G$ and not on the choice of $P$. We denote an arbitrary representative of that projective-equivalence class by $\sigma^n G$. Note that it then makes sense (with some abuse of notation) to consider functors $\Ext{\mathcal{E}}{i}{\sigma^n G}{-}$, since these are independent from the choice of representative $\sigma^n G$. These functors will be important in the following section. 

\begin{dfn}
  \label{dfn:H}
  Let $G$ be a Gorenstein projective object in an exact category $\mathcal{E}$ with enough projectives, and let $n$ be an integer and $i>0$. We define \emph{cohomology functors} (with respect to $G$) as follows:
\begin{equation*}
  \HH{i}{G}{n} \;=\; \Ext{\mathcal{E}}{i}{\sigma^n G}{-}
  \colon \mathcal{E} \longrightarrow \mathsf{Ab} \;.
\end{equation*}
\end{dfn}

%\label{rmk:well-defined}
%By \lemref{proj-equiv} the object $\sigma^n G$ is uniquely determined up to projective summands. This implies that, indeed, the functor $\HH{i}{G}{n}$ is well-defined, as it is independent of the choice of the totally acyclic complex $P$. As noted in \dfnref{totally-acyclic}, each $\sigma^n G$ is a Gorenstein projective object in $\mathcal{E}$. Also 
Note that in the definition above, if the exact category $\mathcal{E}$ is $\Bbbk$-linear, where $\Bbbk$ is a commutative ring, then the functors $\HH{i}{G}{n}$ take values in the category $\lMod{\Bbbk}$ of $\Bbbk$-modules.

The following proposition establishes some expected properties of these cohomology functors.

\begin{prp}
  \label{prp:dimension-shifting}
  Let $G$ be a Gorenstein projective object in an exact category $\mathcal{E}$ with enough projectives. 
  \begin{prt}
  \item For every conflation $X' \rightarrowtail X \twoheadrightarrow X''$ in $\mathcal{E}$ and any integer $n$, there is a long exact sequence:
\begin{equation*}
  \qquad
  \HH{1}{G}{n}(X') \longrightarrow 
  \HH{1}{G}{n}(X) \longrightarrow 
  \HH{1}{G}{n}(X'') \longrightarrow 
  \HH{2}{G}{n}(X') \longrightarrow 
  \HH{2}{G}{n}(X) \longrightarrow 
  \cdots\;.
\end{equation*}

  \item For every $n \in \mathbb{Z}$, $i>0$, and $d \geqslant 0$ there is a natural isomorphism:
\begin{equation*}
   \HH{i}{G}{n} \;\cong\; \HH{i+d}{G}{n+d}\;.
\end{equation*}
  \end{prt}
\end{prp}

\begin{proof}
  \proofoftag{a} This is the long exact sequence that arises by applying the functor $\Ext{\mathcal{E}}{*}{\sigma^n G}{-}$ to the given conflation.
  
\proofoftag{b} Let $P$ be a totally acyclic complex of projectives in $\mathcal{E}$ with $\Cy{0}{P} \cong G$. The complex 
\begin{equation*}
  0 \longrightarrow \Cy{n}{P} \longrightarrow P^{n} \longrightarrow P^{n+1} \longrightarrow \cdots \longrightarrow P^{n+d-1} \longrightarrow \Cy{n+d}{P} \longrightarrow 0
\end{equation*}
is acyclic in $\mathcal{E}$, so dimension shifting yields $\Ext{\mathcal{E}}{i}{\Cy{n}{P}}{-}\cong \Ext{\mathcal{E}}{i+d}{\Cy{n+d}{P}}{-}$, that is, ~$\HH{i}{G}{n}\cong \HH{i+d}{G}{n+d}$.
\end{proof}

%%%%%%%%%%% FINISHED REVIEWING ARTICLE UP TO HERE ON THE 9TH OF AUGUST 2024.
 
%%%%%%%%%%%%%%%%%%%%%%%%%%%%%%%%%%%%%%%%%%%%%%%%%%%%%%%%%%

\section{Exact model structures from Gorenstein projective objects}
\label{sec:projective_model_str}

In this section we construct an exact model structure from any given set $\mathscr{G}$ of Gorenstein projective (or injective) objects, see \thmref{main} (or \thmref{main-op}). In this exact model structure, the trivial objects and the weak equivalences can be detected by certain cohomology functors. Throughout this section, we work in the following setup.

\begin{stp}
  \label{stp:E}
  Let $\mathcal{E}$ be an exact category which is weakly idempotent complete and has enough projectives.
\end{stp}

We build exact model structures in $\mathcal{E}$ using \thmref{Gillespie-CWF}. For that purpose, we need some cotorsion pairs.

\begin{dfn}
  \label{dfn:CW}
Let $\mathscr{G}$ be a class of Gorenstein projective objects in $\mathcal{E}$ and let $(\mathcal{C}_{\mathscr{G}},\mathcal{W}_{\mathscr{G}})$ be the cotorsion pair in $\mathcal{E}$ generated by the class $\mathcal{S}_{\mathscr{G}} = \{\sigma^n G \ | \ G \in \mathscr{G},\, n \in \mathbb{Z} \}$. With the notation from \dfnref{H} we have:
\[
  \mathcal{W}_{\mathscr{G}}=\{W \in \mathcal{E} \ | \ \HH{1}{G}{n}(W)=0 \text{ for all } G \in \mathscr{G},\; n \in \mathbb{Z} \} \quad \text{ and } \quad
  \mathcal{C}_{\mathscr{G}}={}^\perp \mathcal{W}_\mathscr{G} \;.
  \]
\end{dfn}

Note that the notation set up in the definition does not require the class $\mathscr{G}$ to be a set. It is therefore not a priori clear whether the cotorsion pair $(\mathcal{C}_\mathscr{G},\mathcal{W}_\mathscr{G})$ is complete. %However see part \prtlbl{d} of the next result.

\begin{prp}
  \label{prp:CW-properties}
Let $\mathscr{G}$ be a class of Gorenstein projective objects in $\mathcal{E}$. The following statements hold.
\begin{prt}
\item An object $X$ of $\mathcal{E}$ lies in $\mathcal{W}_{\mathscr{G}}$ if and only if\, $\HH{i}{G}{n}(X)=0$ for all $G$ in $\mathscr{G}$, $n$ integer, and $i>0$.
\item The subcategory $\mathcal{W}_{\mathscr{G}}$ of $\mathcal{E}$ is thick.
\item There is an equality $\mathcal{C}_{\mathscr{G}} \cap \mathcal{W}_{\mathscr{G}} = \Prj{\mathcal{E}}$.
\item If $\mathcal{E}$ is efficient and $\mathscr{G}$ is a set, then the hereditary cotorsion pair $(\mathcal{C}_{\mathscr{G}},\mathcal{W}_{\mathscr{G}})$ is complete.
\end{prt}
\end{prp}

\begin{proof}
  \proofoftag{a} The non-trivial implication follows from \prpref{dimension-shifting}\prtlbl{b} since, for $G$ in $\mathscr{G}$, any integer $n$ and $i>0$, we have $\HH{1}{G}{n-i+1} \cong \HH{i}{G}{n}$.
  
\proofoftag{b} The functors $\HH{i}{G}{n}$ are additive, so it  follows from \dfnref{CW} that $\mathcal{W}_\mathscr{G}$ is closed under direct summands. It remains to show that if any two objects in a conflation $X' \rightarrowtail X \twoheadrightarrow X''$ of $\mathcal{E}$ belong to $\mathcal{W}_{\mathscr{G}}$, then so does the third. Using (a), the long exact sequence from \prpref{dimension-shifting}(a) shows that whenever $X'$ and one of $X$ or $X''$ lie in $\mathcal{W}_\mathscr{G}$, then so does the third of them. Now assume that $X$ and $X''$ lie in $\mathcal{W}_\mathscr{G}$. Again by part (a) and \prpref{dimension-shifting}(a) we get \smash{$\HH{i}{G}{n}(X')=0$} for every $G$ in $\mathscr{G}$, $n \in \mathbb{Z}$, and $i \geqslant 2$. To prove that \smash{$\HH{1}{G}{n}(X')$} is zero, note that \prpref{dimension-shifting}(b) gives \smash{$\HH{1}{G}{n}(X') \cong \HH{2}{G}{n+1}(X')$}, which is zero as already seen.

\proofoftag{c} The inclusion $\mathcal{C}_{\mathscr{G}} \supseteq \Prj{\mathcal{E}}$ is trivial, and the inclusion $\mathcal{W}_{\mathscr{G}} \supseteq \Prj{\mathcal{E}}$ holds by 
\lemref{GPrj-Ext}(a). For an object $X$ in $\mathcal{C}_{\mathscr{G}} \cap \mathcal{W}_{\mathscr{G}}$, consider a conflation $W \rightarrowtail P \twoheadrightarrow X$ with $P$ in $\Prj{\mathcal{E}}$. Since $P$ and $X$ lie in $\mathcal{W}_{\mathscr{G}}$, then so does $W$ by part (b). Since $X$ is also in $\mathcal{C}_{\mathscr{G}}$, we have $\Ext{\mathcal{E}}{1}{X}{W}=0$, and hence this conflation splits, showing that $X$ lies indeed in $\Prj{\mathcal{E}}$.

\proofoftag{d} If $\mathscr{G}$ is a set, then so is $\mathcal{S}_{\mathscr{G}}$ (see \dfnref{CW}). Thus, since $\mathcal{E}$ is efficient and the cotorsion pair $(\mathcal{C}_{\mathscr{G}},\mathcal{W}_{\mathscr{G}})$ is generated by a set, it follows that the cotorsion pair is complete by \cite[Cor.~2.15(3)]{SaorinStovicek}. In particular, it follows that $\mathcal{W}_{\mathscr{G}}$ is a cogenerating class in $\mathcal{E}$ (note that we already knew that $\mathcal{C}_\mathscr{G}$ is a generating class in $\mathcal{E}$ since it contains the class of projectives and $\mathcal{E}$ has enough projectives). Note that, since $\mathcal{W}_{\mathscr{G}}$ is thick by part (b), it follows from  \lemref{Stovicek} that $(\mathcal{C}_{\mathscr{G}},\mathcal{W}_{\mathscr{G}})$ is also hereditary.
\end{proof} 

We now prove the main result of this section. We refer the reader to  \rmkref[Remarks~]{efficient} and \rmkref[]{n-is-0} for some comments on the assumptions of the theorem. 

\begin{thm}
  \label{thm:main}
     Let $\mathscr{G}$ be a class of Gorenstein projective objects in a weakly idempotent complete exact category $\mathcal{E}$ with enough projectives and consider the subcategories: 
   \[
  \mathcal{W}_{\mathscr{G}}=\{W \in \mathcal{E} \ | \ \HH{1}{G}{n}(W)=0 \text{ for every } G \in \mathscr{G} \text{ and } n \in \mathbb{Z} \} \quad \text{ and } \quad
  \mathcal{C}_{\mathscr{G}}={}^\perp \mathcal{W}_\mathscr{G} \;.
  \]
If the cotorsion pair $(\mathcal{C}_\mathscr{G},\mathcal{W}_\mathscr{G})$ is complete (for example, when $\mathcal{E}$ is efficient and $\mathscr{G}$ is a set), then there is a hereditary exact model structure on $\mathcal{E}$ such that $\mathcal{C}_{\mathscr{G}}$ is the class of cofibrant objects, $\mathcal{W}_{\mathscr{G}}$ is the class of trivial objects, and every object in $\mathcal{E}$ is fibrant. Furthermore, in that case, we have:
\begin{rqm}
\item An object $X$ in $\mathcal{E}$ is trivial if and only if $\,\HH{i}{G}{n}(X)=0$ for every $G$ in $\mathscr{G}$, $n$ and $i$ integers, with $i>0$.
\item Given a morphism $\varphi$ in $\mathcal{E}$ the following conditions are equivalent:
\begin{eqc}
\item $\varphi$ is a weak equivalence;
\item $\HH{i}{G}{n}(\varphi)$ is an isomorphism for every $G$ in $\mathscr{G}$, $n$ and $i$ integers, with $i>0$;
\item \smash{$\HH{1}{G}{n}(\varphi)$} and \smash{$\HH{2}{G}{n}(\varphi)$} are isomorphisms for every $G$ in $\mathscr{G}$ and $n$ integer.
\end{eqc}
\item $\mathcal{C}_{\mathscr{G}}$ is a Frobenius exact category whose projective-injective objects are precisely $\Prj{\mathcal{E}}$, and the homotopy category of the this model structure $\operatorname{Ho}(\mathcal{E})$ is equivalent to the stable category $\underline{\mathcal{C}}_\mathscr{G}$.
\end{rqm}
\end{thm}

%\hh{We should probably say something about $\operatorname{Ho}(\mathcal{E})$; for example that it is equivalent to the stable category $\mathcal{C}_{\mathscr{G}}/\Prj{\mathcal{E}}$ of the Frobenius category $\mathcal{C}_{\mathscr{G}}$ by Gillespie \cite[Proposition~5.2(4) and Corollary~5.4]{MR2811572}.}

\begin{proof}
Since $\mathcal{E}$ has enough projectives, using \prpref{CW-properties}(c) and our assumption on the completeness of $(\mathcal{C}_\mathscr{G},\mathcal{W}_\mathscr{G})$ (see also the proof of \prpref{CW-properties}(d)) we have two complete hereditary cotorsion pairs:
  \begin{equation*}
     (\mathcal{C}_{\mathscr{G}} \cap \mathcal{W}_{\mathscr{G}},\mathcal{E}) \;=\; (\Prj{\mathcal{E}},\mathcal{E})
     \qquad \text{and} \qquad
     (\mathcal{C}_{\mathscr{G}}, \mathcal{W}_{\mathscr{G}} \cap\; \mathcal{E}) \;=\; 
     (\mathcal{C}_{\mathscr{G}}, \mathcal{W}_{\mathscr{G}}) \;.
  \end{equation*}
Note that indeed \prpref{CW-properties}(d) asserts that our assumption on the completeness of $(\mathcal{C}_\mathscr{G},\mathcal{W}_\mathscr{G})$ is fulfilled if $\mathcal{E}$ is efficient and $\mathscr{G}$ is a set. Moreover, the proof of it makes it clear that once we have the completeness of the cotorsion pair, heredity follows (as stated above). Since $\mathcal{W}_{\mathscr{G}}$ is thick by \prpref{CW-properties}(b), it follows from Gillespie's  \thmref{Gillespie-CWF} (for which we need $\mathcal{E}$ to be weakly idempotent complete) that there is indeed an hereditary exact model structure on $\mathcal{E}$ with cofibrant, trivial, and fibrant objects as asserted. We can now prove the statements about this model structure.
  
\proofoftag{1} The assertion follows immediately from \prpref{CW-properties}(a). 
 
\proofoftag{2} Consider a morphism $\varphi \colon X \to Y$~in~$\mathcal{E}$. Note that the implication \eqclbl{ii}$\;\Rightarrow\;$\eqclbl{iii} is trivial.
 
\proofofimp{i}{ii}  If $\varphi$ is a weak equivalence then, as noted after \thmref{Gillespie-CWF}, we have $\varphi = \pi\iota$ where 
$$\delta\colon\ \ X \smasharrow[\rightarrowtail]{\iota} Z \twoheadrightarrow P\quad {\rm and}\quad \epsilon\colon\ \ W \rightarrowtail Z \smasharrow[\twoheadrightarrow]{\pi} Y$$ 
are conflations with $P$ in $\mathcal{C}_{\mathscr{G}} \cap \mathcal{W}_{\mathscr{G}} = \Prj{\mathcal{E}}$ and $W$ in $\mathcal{W}_{\mathscr{G}}$. Fix an object $G$ in $\mathscr{G}$, and integers $i,n$ with $i>0$. By \prpref{dimension-shifting}(a), since $P$ and $W$ lie in $\mathcal{W}_{\mathscr{G}}$, we have \smash{$\HH{i}{G}{n}(P)=\HH{i}{G}{n}(W)=\HH{i+1}{G}{n}(W)=0$}. From the long exact sequence of cohomology of Proposition \ref{prp:dimension-shifting}(a), we immediately find that $\HH{i}{G}{n}(\iota)$ is an isomorphism for $i>1$ and $\HH{i}{G}{n}(\pi)$ is an isomorphism for $i>0$. Finally, note that since $P$ is projective, $\iota$ is a split monomorphism, and since $\HH{1}{G}{n}(P)=0$, we conclude that also $\HH{1}{G}{n}(\iota)$ is an isomorphism.

\proofofimp{iii}{i}  Every morphism $\varphi \colon X \to Y$ in a model category admits a factorisation $X \smasharrow{\iota} Z \smasharrow{\pi} Y$ where $\iota$ is a cofibration and $\pi$ is a trivial fibration (see the factorisation axiom in \cite[Def.~1.1.3(4)]{modcat}). By assumption, the map \smash{$\HH{i}{G}{n}(\varphi) = \HH{i}{G}{n}(\pi)\HH{i}{G}{n}(\iota)$} is an isomorphism for every $G$ in $\mathscr{G}$, $n$ in $\mathbb{Z}$, and $i=1,2$. As $\pi$ is, in particular, a weak equivalence, we know from the already established implication \eqclbl{i}$\;\Rightarrow\;$\eqclbl{ii} that \smash{$\HH{i}{G}{n}(\pi)$} is an isomorphism for any $i>0$ and, hence, that \smash{$\HH{i}{G}{n}(\iota)$} is an isomorphism for $i=1,2$, for any $G$ in $\mathscr{G}$ and any $n$ in $\mathbb{Z}$. We will now show that $\iota$ is a weak equivalence, and hence the composite $\varphi = \pi\iota$ is a weak equivalence too by the the 2-out-of-3 axiom \cite[Def.~1.1.3(1)]{modcat}. As $\iota$ is a cofibration in the exact model structure on $\mathcal{E}$ there exists a conflation $X \smasharrow[\rightarrowtail]{\iota} Z \twoheadrightarrow C$ in $\mathcal{E}$ where $C$ is a cofibrant object, that is, $C$ lies in $\mathcal{C}_{\mathscr{G}}$ (see discussion after \thmref{Gillespie-CWF}). By \prpref{dimension-shifting}(a) this conflation induces an exact sequence,
\begin{equation*}
  \xymatrix{
    \HH{1}{G}{n}(X) \ar[r]^-{\HH{1}{G}{n}(\iota)}_{\cong} &
    \HH{1}{G}{n}(Z) \ar[r] &
    \HH{1}{G}{n}(C) \ar[r] &
    \HH{2}{G}{n}(X) \ar[r]^-{\HH{2}{G}{n}(\iota)}_{\cong} &
    \HH{2}{G}{n}(Z)
  },
\end{equation*}
for every $G$ in $\mathscr{G}$ and $n$ in $\mathbb{Z}$. As \smash{$\HH{1}{G}{n}(\iota)$} and \smash{$\HH{2}{G}{n}(\iota)$} are isomorphisms, it follows that \smash{$\HH{1}{G}{n}(C)=0$}, and hence $C$ lies in $\mathcal{W}_{\mathscr{G}}$. Thus $C$ lies in the intersection $\mathcal{C}_{\mathscr{G}} \cap \mathcal{W}_{\mathscr{G}}$, which means that $\iota$ is a trivial cofibration (see \cite[Def.~3.1]{MR2811572}). So $\iota$ is a weak equivalence.

\proofoftag{3} The final assertion follows directly from  \cite[Prop.~5.2(4) and Cor.~5.4]{MR2811572}.
\end{proof}

%%%%%%%%%%%%%% VERSION REVISED UNTIL HERE ON THE 28 AUGUST 2024.

%\hh{\rmkref{efficient} and \thmref{main-op} are new; they grew out of and elaborate on \rmkref[Remarks~]{AAA} and \rmkref[]{class vs set}, (which are now perhaps obsolete).}

\begin{rmk}
  \label{rmk:efficient} 
  We will often use the theorem above for efficient exact categories. Nevertheless, we state it in its full generality for two reasons.
\begin{itemlist}
\item In some cases (see, for example, Proposition \ref{prp:injective cotorsion pair is complete}), one can prove directly that the cotorsion pair $(\mathcal{C}_{\mathscr{G}},\mathcal{W}_{\mathscr{G}})$ is complete even if $\mathscr{G}$ is not a set (but a proper class)

\item It allows for potential applications of \thmref{main} to the opposite exact category $\mathcal{E}^\mathrm{op}$ if $\mathcal{E}$ is weakly idempotent complete with enough injectives (note that $\mathcal{E}^\mathrm{op}$ is usually not efficient even if $\mathcal{E}$ is efficient).
\end{itemlist}
Let us explain the second bullet in greater detail. Let $\mathcal{E}$ be a weakly idempotent complete exact category with enough injectives. Then $\mathcal{E}^\mathrm{op}$ is a weakly idempotent complete exact category with enough projectives and hence the definitions and  results in this section apply to the category $\mathcal{E}^\mathrm{op}$. For example, \dfnref{totally-acyclic} dualises as follows: for every \textsl{Gorenstein injective} object $J$ in $\mathcal{E}$ there are well-defined cohomology functors,
\begin{equation}
  \label{eq:HHop}
  \HHop{i}{J}{n} \;=\; \Ext{\mathcal{E}}{i}{-}{\sigma^{-n} J}
  \colon \mathcal{E}^\mathrm{op} \longrightarrow \mathsf{Ab} \;,
\end{equation}
where $\sigma^{n} J$ is an arbitrary representative of the injective-equivalence class of $\Cy{n}I$ for any choice of totally acyclic complex 
%\begin{equation*}
 % I \,=\, 
  %\xymatrix@C=1.6pc{
   % \cdots \ar[r] & 
    %I^{n-1} \ar[r]^-{\partial_I^{n-1}} &
    %I^{n} \ar[r]^-{\partial_I^{n}} &
    %I^{n+1} \ar[r] &
    %\cdots
  %}
%\end{equation*}
of \textsl{injectives} in $\mathcal{E}$ with $\Cy{0}{I} \cong J$. In the dual version of \dfnref{CW} one starts with a class $\mathscr{J}$ of \textsl{Gorenstein injective} objects in $\mathcal{E}$ and considers the cotorsion pair $(\mathcal{W}^{\mathscr{J}},\mathcal{F}^{\mathscr{J}})$ in $\mathcal{E}$ \textsl{cogenerated} by the class $\mathcal{S}^{\mathscr{J}} = \{\sigma^n J \ | \ J \in \mathscr{J},\, n \in \mathbb{Z} \}$,~i.e.
\[
  \mathcal{W}^{\mathscr{J}}= \{W \in \mathcal{E} \ | \ \HHop{1}{J}{n}(W)=0 \text{ for every } J \in \mathscr{J} \text{ and } n \in \mathbb{Z} \}   \quad \text{ and } \quad 
 \mathcal{F}^{\mathscr{J}}=(\mathcal{W}^\mathscr{J})^\perp \;.
\]
It is also useful to note the analogue of \eqref{Ext-HHom}, which is:
\begin{equation}
  \label{eq:Ext-HHom-inj}
  \Ext{\mathcal{E}}{i}{X}{\Cy{-n}{J}}
  \;=\;
  \operatorname{H}^{i-n}\Hom{\mathcal{E}}{X}{J} 
  \qquad \text{for} \qquad
  n \in \mathbb{Z} \ \text{ and } \ i>0 \;.
\end{equation}
\end{rmk}

In view of \rmkref{efficient}, we get the following dual version of \thmref{main}.

\begin{thm}
  \label{thm:main-op}
  Let $\mathscr{J}$ be a class of Gorenstein injective objects in a weakly idempotent complete exact category $\mathcal{E}$ with enough injectives and consider the subcategories: 
 \[
  \mathcal{W}^{\mathscr{J}}= \{W \in \mathcal{E} \ | \ \HHop{1}{J}{n}(W)=0 \text{ for every } J \in \mathscr{J} \text{ and } n \in \mathbb{Z} \}   \quad \text{ and } \quad 
 \mathcal{F}^{\mathscr{J}}=(\mathcal{W}^\mathscr{J})^\perp \;.
\]
If the cotorsion pair $(\mathcal{W}^\mathscr{J},\mathcal{F}^\mathscr{J})$ is complete (for example, when $\mathcal{E}^\mathrm{op}$ is efficient and $\mathscr{J}$ is a set), then there is an hereditary exact model structure on $\mathcal{E}$ such that every object in $\mathcal{E}$ is cofibrant, $\mathcal{W}^{\mathscr{J}}$ is the class of trivial objects, and $\mathcal{F}^\mathscr{J}$ is the class of fibrant objects. Furthermore, in that case, we have:
\begin{rqm}
\item An object $X$ in $\mathcal{E}$ is trivial if and only if $\,\HHop{i}{J}{n}(X)=0$ for every $J$ in $\mathscr{J}$, $n$ and $i$ integers, with $i>0$.
\item Given a morphism $\varphi$ in $\mathcal{E}$ the following conditions are equivalent:
\begin{eqc}
\item $\varphi$ is a weak equivalence;
\item $\HHop{i}{J}{n}(\varphi)$ is an isomorphism for every  $J$ in $\mathscr{J}$, $n$ and $i$ integers, with $i>0$;
\item \smash{$\HHop{1}{J}{n}(\varphi)$} and \smash{$\HH{2}{J}{n}(\varphi)$} are isomorphisms for every $J$ in $\mathscr{J}$ and $n$ integer.
\end{eqc}
\item $\mathcal{F}^{\mathscr{J}}$ is a Frobenius exact category whose projective-injective objects are precisely $\Inj{\mathcal{E}}$, and the homotopy category of the this model structure $\operatorname{Ho}(\mathcal{E})$ is equivalent to the stable category $\underline{\mathcal{F}}^\mathscr{J}$.
\end{rqm}
\end{thm}

%\hh{\rmkref{n-is-0} is new. It addresses to what extend it suffices to consider $n=0$ in \thmref[Theorems~]{main} and \thmref[]{main-op}.}

\begin{rmk}
  \label{rmk:n-is-0}
In the setting of \thmref{main}, given an object $X$ in $\mathcal{E}$ we have by \dfnref{CW} that
\begin{equation*}
  X \in \mathcal{W}_{\mathscr{G}}
  \quad \iff \quad
  \HH{1}{G}{n}(X) \mspace{1mu}=\mspace{1mu} \Ext{\mathcal{E}}{1}{\sigma^n G}{X} \mspace{1mu}=\mspace{1mu} 0 \ \quad \forall\,G \in \mathscr{G} \ \text{ and } \ \forall\, n \in \mathbb{Z}\;.
\end{equation*}  
It follows from the proof of \prpref{CW-properties}(b) that it is crucial to consider all integers $n$ when computing the Ext-ortogonal of $\{\sigma^nG\colon n\in\mathbb{Z}\}$; indeed this is the reason why $\mathcal{W}_\mathscr{G}$ is thick, and without a thick subcategory of trivial objects we cannot hope to derive an exact model structure from it. Nevertheless, we will see that in some special cases (e.g.~\lemref{Ext-comparison-Prj}), we have
\begin{equation}
  \label{eq:important-equivalence}
  X \in \mathcal{W}_{\mathscr{G}}
  \quad \iff \quad
  \HH{1}{G}{0}(X) \mspace{1mu}=\mspace{1mu} \Ext{\mathcal{E}}{1}{G}{X} \mspace{1mu}=\mspace{1mu} 0 \ \quad \forall\,G \in \mathscr{G}.
\end{equation}  
In such cases, the proof of \thmref{main} carries on to show that conditions \eqclbl{i}--\eqclbl{iii} become also equivalent to
\begin{eqc}
\item[\eqclbl{iv}] \smash{$\HH{1}{G}{0}(\varphi)$} and \smash{$\HH{2}{G}{0}(\varphi)$} are isomorphisms for every $G \in \mathscr{G}$.
\end{eqc}
Similar observations can be made for \thmref{main-op}, should such a special situation arise (e.g.~\lemref{Ext-comparison-Inj}).
\end{rmk}

\section{Objectwise exact structure on \texorpdfstring{$\lMod{Q,A}$}{Q,A-Mod}}
\label{sec:GProj-Proj-in-Q,A-Mod}

For a suitable small $\Bbbk$-preadditive category $Q$, the category $\lMod{Q,A}$ of $\Bbbk$-linear functors from $Q$ to the category $\lMod{A}$ of modules over a $\Bbbk$-algebra $A$ admits model structures whose homotopy categories behave similarly to a (classical) derived category of a ring, as shown in \cite{HJ-JLMS}. We are interested in endowing this functor category with various exact structures defined objectwise and study the corresponding homotopy categories. For this purpose, we apply \thmref{main}. We study homological properties relative to a fixed objectwise exact structure on $\lMod{Q,A}$, such as (Gorenstein) projectivity and (Gorenstein) injectivity, making it possible to apply Theorem \ref{thm:main}.

\begin{stp}
\label{stp:HJ}
Throughout this section, we work with the following setup from \cite{HJ-JLMS}.

\begin{itemlist}
\item $\Bbbk$ is a commutative, noetherian, and hereditary ring (e.g.~$\Bbbk$ is a field or $\Bbbk=\mathbb{Z}$).
\item $Q$ is a small Hom-finite, locally bounded $\Bbbk$-preadditive category with a Serre functor $\mathbb{S}$ and the strong retraction property (see \cite[Setup~2.5 and Def.~7.3]{HJ-JLMS}). The hom-sets  in $Q$, which are $\Bbbk$-modules, are denoted by $Q(p,q)$ for $p$ and $q$ in $Q$.
\item The pseudo-radical $\mathfrak{r}$ in $Q$ is nilpotent; see \cite[Lem.~7.7]{HJ-JLMS}.
\item $A$ is any $\Bbbk$-algebra.
\end{itemlist}
\end{stp}

\begin{exa}
A typical example of one such category $Q$ is the mesh category of the repetitive quiver of $\mathbb{A}_n$, see \cite[Theorem 8.11]{HJ-JLMS}. This includes in particular the category $Q$ whose representations over a ring $\mathbb{K}$ are just complexes over $\mathbb{K}$: this is the mesh category of the repetitive quiver of $\mathbb{A}_2$. 
\end{exa}

\subsection{Preliminaries on \texorpdfstring{$Q$}{Q}-shaped derived categories}
\label{Prelim-on-Q-shaped}
\phantom{.} \vspace*{1ex}

As in \cite[Def.~3.1]{HJ-JLMS} we consider the abelian category of $\Bbbk$-linear functors
\begin{equation}\nonumber
%  \label{eq:lModQA}
  \lMod{Q,A} \,=\, \operatorname{Fun}_\Bbbk(Q,\lMod{A})\;,
\end{equation}
which  we call the \textbf{category of $Q$-shaped $A$-modules}.
We write $\lMod{Q} = \lMod{Q,\Bbbk}$ and $\rMod{Q} = \lMod{Q^\mathrm{op}}$ for short. 
By \cite[Cor.~3.9]{HJ-JLMS} there is for every $q$ in $Q$ an adjoint triple $(\Fq{q},\Eq{q},\Gq{q})$ of exact functors as follows:
  \begin{equation}
  \label{eq:FEG-def}
  \xymatrix@C=4pc{
    \lMod{Q,A}
    \ar[r]^-{\Eq{q}}
    &
    \lMod{A}
    \ar@/_1.8pc/[l]_-{\Fq{q}}
    \ar@/^1.8pc/[l]^-{\Gq{q}}     
  }
  \qquad \text{given by} \qquad
  {\setlength\arraycolsep{1.5pt}
   \renewcommand{\arraystretch}{1.2}
  \begin{array}{rcl}
  \Fq{q}(M) &=& Q(q,-) \otimes_\Bbbk M \\
  \Eq{q}(X) &=& X\mspace{1.5mu}(q) \\ 
  \Gq{q}(M) &=& \Hom{\Bbbk}{Q(-,q)}{M}\;.
  \end{array}
  }
  \end{equation}
In fact, it follows from the conditions imposed on $Q$ (namely, Hom-finiteness and the existence of a Serre functor $\mathbb{S}$) that $\Gq{\mathbb{S}q}\cong \Fq{q}$ (see \cite[Lem.~3.4]{HJ-TAMS}) for any $q$ in $Q$. Since $\mathbb{S}$ is an autoequivalence, the family of functors $\{\Fq{q}\}_{q\in Q}$ is a permutation of the family $\{\Gq{q}\}_{q\in Q}$. By \cite[Con.~5.6 and Prop.~5.7]{HJ-TAMS} there are useful short exact sequences
\begin{equation}
  \label{eq:varepsilon-eta}
  \mathfrak{f}_X\colon \ \ \xymatrix@C=1.3pc{
    0 \ar[r] & \mathbb{K}(X) \ar[r] & \mathbb{F}(X) \ar[r]^-{\varepsilon^X} & X \ar[r] & 0
  }
  \qquad \text{and} \qquad
 \mathfrak{g}_X\colon\ \  \xymatrix@C=1.3pc{
    0 \ar[r] & X \ar[r]^-{\eta^X} & \mathbb{G}(X) \ar[r] & \mathbb{C}(X) \ar[r] & 0
  }
\end{equation}  
where the functors $\mathbb{F}$, $\mathbb{G}$ are defined by
\begin{equation}
  \label{eq:FF-and-GG}
  \textstyle
  \mathbb{F}(X) \,=\, \bigoplus_{q \in Q} \Fq{q}(X(q))
  \qquad \text{and} \qquad
  \mathbb{G}(X) \,=\, \prod_{q \in Q} \Gq{q}(X(q)) 
\end{equation}
and where $\epsilon$ and $\eta$ are the unique natural transformations defined by the universal properties of coproducts and products, respectively, such that $\epsilon^X_q$ is the counit of $(\Fq{q},\Eq{q})$ at $X$ and $\eta^X_q$ is the unit of $(\Eq{q},\Gq{q})$ at $X$. The remaining objects in the sequences are defined by $\mathbb{K}(X) = \Ker{\varepsilon^X}$ and $\mathbb{C}(X) = \Coker{\eta^X}$. In \cite[Prop.~5.7]{HJ-TAMS}, these sequences are shown to be objectwise split, i.e.~applying $\Eq{q}$ to them yields split exact sequences, for any $q$ in $Q$. 

There is an important family of objects in $\lMod{Q}$ (respectively, $\rMod{Q}$), called the \textbf{left} (respectively, \textbf{right}) \textbf{stalk functors} at an object $q$ of $Q$ (see \cite[Def.~7.9 and Lem.~7.10]{HJ-JLMS}). They are defined by
\begin{equation}
  \label{eq:stalks}
\stalkco{q} = Q(q,-)/\mathfrak{r}(q,-)
  \qquad \textnormal{(respectively,} \quad 
\stalkcn{q} = Q(-,q)/\mathfrak{r}(-,q) \textnormal{)}\;.
\end{equation}
Considering \smash{$\operatorname{Ext}^*_Q$} to be the Ext functor in the abelian category $\lMod{Q}$, and \smash{$\operatorname{Tor}^Q_*$} to be the left derived of the tensor product introduced in \cite{OberstRohrl}, we define (co)homology functors $\lMod{Q,A} \to \lMod{A}$ as in \cite[Def.~7.11]{HJ-JLMS}:
\begin{equation}
  \label{eq:HHhj}
  \cHhj{i}{q} \,=\, \Ext{Q}{i}{\stalkco{q}}{-} 
  \qquad \textnormal{and} \qquad
  \hHhj{i}{q} \,=\, \Tor{Q}{i}{\stalkcn{q}}{-} \;.
\end{equation}

It turns out that in the category $\lMod{Q}$, the projective dimension of an object is finite if and only if so is the injective dimension. This is because $\lMod{Q}$ is locally Gorenstein, in the sense of \cite[Def.~2.18]{MR2404296}. This is a consequence of \cite[Thm.~4.6]{MR3719530} (see also \cite[Thm.~2.4]{HJ-JLMS}). Using the forgetful functor $(-)^\natural \colon \lMod{Q,A} \to \lMod{Q}$, we define the class $\mathbb{E}$ of \textbf{exact objects}\footnote{Note that our notation for the subcategory of exact objects differs slightly from that of \cite{HJ-JLMS}.} in $\lMod{Q,A}$ by setting
%It follows from \cite[Thm.~4.6]{MR3719530} that the category $\lMod{Q}$ is locally Gorenstein, in the sense of \cite[Def.~2.18]{MR2404296} (see also \cite[Thm.~2.4]{HJ-JLMS}). In particular, any object in this category has finite projective dimension if and only if it has finite injective dimension. The class $\mathbb{E}$ of \textbf{exact objects} in $\lMod{Q,A}$ is in \cite[Def.~4.1]{HJ-JLMS} defined as
\begin{equation}
  \label{eq:E}
  \mathbb{E} \,=\, \{X\in\lMod{Q,A}\;|\; \text{$X^\natural$ has finite projective/injective dimension in $\lMod{Q}$} \} \;.
\end{equation}

As shown in \cite{HJ-JLMS}, $({}^\perp\mathbb{E},\mathbb{E},\lMod{Q,A})$ and $(\lMod{Q,A},\mathbb{E},\mathbb{E}^\perp)$ are hereditary Hovey triples in $\lMod{Q,A}$ (see Subsection~\ref{exact-model-structures} for this terminology) and the associated model categories (see \thmref{Gillespie-CWF}) have the same weak equivalences (see \cite[Thm.~6.1 and Prop.~6.3]{HJ-JLMS}). The $Q$-shaped derived category of $A$ is then defined in \cite[Def.~6.4]{HJ-JLMS} to be the homotopy category of either of these model categories:
\begin{equation*}
  \QSD{Q}{A} \;=\; \operatorname{Ho}(\lMod{Q,A}) . 
\end{equation*}
The fact that these homotopy categories coincide and that, moreover, they are understood as stable categories of Frobenius exact categories (${}^\perp\mathbb{E}$ and $\mathbb{E}^\perp$) is shown in \cite[Thm.~6.5]{HJ-JLMS}. Furthermore, in \cite[Thms.~7.1 and 7.2]{HJ-JLMS}, it is shown how trivial objects and weak equivalences in either of the two model structures on $\lMod{Q,A}$ mentioned above can be characterised in terms of the (co)homology functors in \eqref{HHhj}.

\subsection{Properties of objectwise exact structures on \texorpdfstring{$\lMod{Q,A}$}{Q,A-Mod}}
\phantom{.} \vspace*{1ex}

In the previous subsection (\ref{Prelim-on-Q-shaped}) we briefly outlined the theory developed in \cite{HJ-JLMS} for the \textsl{abelian} category $\lMod{Q,A}$ where $Q$ is as in \stpref{HJ}. We shall now consider $\lMod{Q,A}$ as an \textsl{exact}  category. More precisely, as mentioned in \cite[Ex.~13.11]{Buhler}, any exact structure in $\exact{}$ in $\lMod{A}$ induces an exact structure in $\lMod{Q,A}$, which we shall denote by $\exact{Q}$. Indeed, a short exact sequence $\delta$ will be declared to be a conflation in $\exact{Q}$ if and only if $\Eq{q}(\delta)$ is a conflation of $\exact{}$; here $\Eq{q}$ is the evaluation functor from \eqref{FEG-def}. Our aim in this subsection is to show that the properties of $\exact{}$ often transfer to properties of $\exact{Q}$. 

\subsubsection{Exact functors}
As noted in Subsection~\ref{Prelim-on-Q-shaped}, the functors $\Fq{q}$, $\Eq{q}$, and $\Gq{q}$ from \eqref{FEG-def} are exact when $\lMod{Q,A}$ and $\lMod{A}$ are considered as abelian categories. As \lemref{FEG-exact} shows, even more is true, and we shall make use of this in the proof of \prpref{Prj-transfer}.

\begin{lem}
  \label{lem:FEG-exact}
  Let $\exact{}$ be an exact structure in $\lMod{A}$. For each $q \in Q$ the functors
  \begin{equation*}
    \Eq{q} \colon (\lMod{Q,A},\exact{Q}) \longrightarrow (\lMod{A},\exact{})
  \qquad \text{and} \qquad    
    \Fq{q}, \Gq{q} \colon (\lMod{A},\exact{})\longrightarrow (\lMod{Q,A},\exact{Q})
  \end{equation*}
  are exact (i.e., they preserve conflations).
\end{lem}

\begin{proof}
The functor $\Eq{q}$ is exact by definition of $\exact{Q}$. As noted in Subsection \ref{Prelim-on-Q-shaped} there is an isomorphism $\Gq{q} \cong \Fq{\mspace{3mu}\mathbb{S}^{-1}q}$ where $\mathbb{S}$ is the Serre functor on $Q$. Thus, it suffices to show that the functor $\Fq{q}$ preserves conflations. Let $\mu$ be a conflation in $\exact{}$. To prove that $\Fq{q}(\mu)$ is a conflation in $\exact{Q}$ we must argue that $\Eq{p}\Fq{q}(\mu)$ is a conflation in $\exact{}$ for every $p$ in $Q$. By definition we have $\Eq{p}\Fq{q}(\mu) = Q(q,p)\otimes_\Bbbk \mu$. By \stpref{HJ}, $Q(q,p)$ is a finitely generated projective $\Bbbk$-module, so it is a direct summand of a finite direct sum of copies of $\Bbbk$. Consequently, $Q(q,p)\otimes_\Bbbk \mu$ is direct summand of a finite direct sum of copies of the conflation $\Bbbk \otimes_\Bbbk \mu \cong \mu$. In any exact category, a finite direct sum of conflations is a conflation, and a direct summand of a conflation is again a conflation, see \cite[Prop.~2.9 and Cor.~2.18]{Buhler}, so it follows that $Q(q,p)\otimes_\Bbbk \mu$ is a conflation.
\end{proof}

\begin{lem}
  \label{lem:exact-coprod}
  For an exact structure $\exact{}$ in $\lMod{A}$, if $(\lMod{A},\exact{})$ has exact (co)products then so does $(\lMod{Q,A},\exact{Q})$.
\end{lem}

\begin{proof}
  This follows immediately from the definition of the exact structure $\exact{Q}$ and from the fact that each evaluation functor $\Eq{q}$ preserves (co)products.
\end{proof}

\begin{rmk}
  \label{rmk:exact-prod}
  We shall mostly be interested in exact structures $\exact{}$ in $\lMod{A}$ for which the category $(\lMod{A},\exact{})$, and hence also $(\lMod{Q,A},\exact{Q})$ by \prpref{Prj-transfer}(a) below, admits enough projectives. In this case, the categories $(\lMod{A},\exact{})$ and $(\lMod{Q,A},\exact{Q})$ automatically have exact products (but not necessarily exact coproducts). This follows from \cite[Exerc.~11.10]{Buhler}, combined with fact that the Hom-functor commutes with products and that products are exact in $\lMod{\Bbbk}$. 
\end{rmk}

\subsubsection{Projective and injective objects}
\label{relative-Prj}
We now investigate how the exact functors from \lemref{FEG-exact} interact with projective and injective objects in the exact categories $(\lMod{Q},\exact{})$ and $(\lMod{Q,A},\exact{Q})$.

\begin{prp}
\label{prp:Prj-transfer}
Let $\exact{}$ be an exact structure in $\lMod{A}$.
\begin{prt}
\item If $(\lMod{A},\exact{})$ admits enough projectives, then so does $(\lMod{Q,A},\exact{Q})$. Moreover, an object is projective in $(\lMod{Q,A},\exact{Q})$ if and only if it is a direct summand of an object of the form $\bigoplus_{q\in Q}\Fq{q}(T_q)$ where $T_q$ is projective in $(\lMod{A},\exact{})$ for all $q$ in $Q$.

\item If $(\lMod{A},\exact{})$ admits enough injectives, then so does $(\lMod{Q,A},\exact{Q})$. Moreover, an object is injective in $(\lMod{Q,A},\exact{Q})$ if and only if it is a direct summand of an object of the form $\prod_{q\in Q}\Gq{q}(I_q)$ where $I_q$ is injective in $(\lMod{A},\exact{})$ for all $q$ in $Q$.

\item If $R$ is projective in $(\lMod{Q,A},\exact{Q})$, then $\Eq{q}(R) = R(q)$ is projective in $(\lMod{A},\exact{})$ for every $q \in Q$.

\item If $J$ is injective in $(\lMod{Q,A},\exact{Q})$, then $\Eq{q}(J) = J(q)$ is injective in $(\lMod{A},\exact{})$ for every $q \in Q$.
\end{prt}
\end{prp}

\begin{proof}
We prove (a) and (c); the other statements are dual. 

\proofoftag{a} Assume that the exact category $(\lMod{A},\exact{})$ admits enough projectives. We will show that 
\begin{equation*}
\textstyle
\mathscr{P} \,=\, \big\{\bigoplus_{q\in Q}\Fq{q}(T_q) \;\big|\; q\in Q,\; T_q \in \Prj{(\lMod{A},\exact{})} \big\}
\end{equation*}
is a class of projective generators for the exact category $(\lMod{Q,A},\exact{Q})$ and, therefore, that any projective object in this exact category is a direct summand of an object in $\mathscr{P}$ (since a deflation towards a projective splits). First observe that there is a natural isomorphism $\Hom{Q,A}{\Fq{q}(T)}{-} \cong \Hom{A}{T}{\Eq{q}(-)}$ of functors $(\lMod{Q,A},\exact{Q}) \to \lMod{\Bbbk}$ between exact categories. The latter functor is exact when $T$ is projective in $(\lMod{A},\exact{})$, see \lemref{FEG-exact}, so it follows that, indeed, every object of $\mathscr{P}$ is projective in $(\lMod{Q,A},\exact{Q})$. 

Let $X$ be an object in $\lMod{Q,A}$ and consider for each $q$ in $Q$ a deflation $\varphi_q\colon T_q \twoheadrightarrow \Eq{q}(X)$ with $T_q$ projective in $(\lMod{A},\exact{})$, which exists by assumption. Consider the coproduct $\bigoplus_{q\in Q}\Fq{q}(T_q)$, with component embeddings given by $\iota_p\colon \Fq{p}(T_p)\to \bigoplus_{q\in Q}\Fq{q}(T_q)$, for each $p$ in $Q$. Using the universal property of the coproduct, we consider the unique morphism $\Phi\colon \bigoplus_{q\in Q}\Fq{q}(T_q) \to X$ in $\lMod{Q,A}$ for which \smash{$\Phi \circ \iota_p=\varepsilon_p^X\circ\Fq{p}(\varphi_p)$}, where $\varepsilon_p\colon \Fq{p}\Eq{p}\to \mathrm{Id}_{\lMod{Q,A}}$ is the counit of the adjunction $(\Fq{p},\Eq{p})$. It remains to see that $\Phi$ is a deflation in $(\lMod{Q,A},\exact{Q})$ i.e., that $\Eq{p}(\Phi)$ is a deflation in $(\lMod{A},\exact{})$ for every $p$ in $Q$. Note that \smash{$\Eq{p}(\varepsilon^X_p)$} is a split epimorphism (see e.g.~\cite[Chap.~IV.1, Thm.~1(ii)]{Mac}) and hence a deflation by \cite[Cor.~7.5]{Buhler}, as $\lMod{A}$ is clearly weakly idempotent complete. Since $\varphi_p$ is a deflation and $\Eq{p}$ and $\Fq{p}$ are exact by \lemref{FEG-exact}, it follows that also $\Eq{p}\Fq{p}(\varphi_p)$ is a deflation and, thus, so is the composition $\Eq{p}(\varepsilon_p^X)\circ \Eq{p}\Fq{p}(\varphi_p)=\Eq{p}(\Phi)\circ \Eq{p}(\iota_p)$. It now follows from \cite[Prop.~7.6]{Buhler} that $\Eq{p}(\varphi)$ is a deflation, as wanted. 

\proofoftag{b} Observe that if $R$ is a projective object in $(\lMod{Q,A},\exact{Q})$, then $\Hom{A}{\Eq{q}(R)}{-}$ maps conflations in $(\lMod{A},\exact{})$ to short exact sequences in $\lMod{\Bbbk}$, thus proving that $\Eq{q}(R)$ is projective in $\lMod{A}$. Indeed, this holds because $\Hom{A}{\Eq{q}(R)}{-} \,\cong\, \Hom{Q,A}{R}{\Gq{q}(-)}$, where $\Gq{q}$ preserves conflations by \lemref{FEG-exact}. 
\end{proof}

\begin{prp}
  \label{prp:efficient-transfer}
Let $\exact{}$ be an exact structure in $\lMod{A}$. If the exact category $(\lMod{A},\exact{})$ is efficient and admits enough projectives, then $(\lMod{Q,A},\exact{Q})$ is efficient and admits enough projectives.
\end{prp}

\begin{proof}
  We verify that $(\lMod{Q,A},\exact{Q})$ meet the axioms of an efficient exact category, following \cite[Def.~2.6 and Prop.~2.7]{SaorinStovicek}. Firstly, since $\lMod{Q,A}$ is cocomplete, transfinite compositions of inflations exist. Moreover, such a transfinite composition $\varphi$ is an inflation because, for every $q$ in $Q$, the morphism $\Eq{q}(\varphi)$ is itself a transfinite composition of inflations in $\exact{}$, as colimits are computed objectwise in $\lMod{Q,A}$ and $\Eq{q}$ is exact. Furthermore, since $\lMod{Q,A}$ is a Grothendieck category, every object is small with respect to the class of all morphisms in $\lMod{Q,A}$ (\cite[Lem.~4.4 and Subsect.~1.1]{SaorinStovicek}) and, therefore, every object is small with respect to the class of all inflations of $\exact{Q}$. Using \prpref{Prj-transfer}(a), we conclude that $(\lMod{Q,A},\exact{Q})$ is an exact category with enough projectives and with the two properties listed above. By \cite[Prop.~2.7]{SaorinStovicek} this suffices to conclude that $(\lMod{Q,A},\exact{Q})$ is efficient.
\end{proof}

\begin{rmk}
\label{rmk:Frobenius}
It follows from \cite[Lems.~3.4 and 3.7]{HJ-TAMS} that there are natural isomorphisms of functors
\begin{equation*}
  \textstyle
  \prod_{q \in Q}\Gq{q}
  \;\cong\;
  \bigoplus_{q \in Q}\Gq{q}
  \;\cong\;
  \bigoplus_{q \in Q}\Fq{\mathbb{S}^{-1}q} \;.
 \end{equation*}
 This is useful to compare projective and injective objects. For example, suppose that $\exact{}$ is a Frobenius exact structure in $\lMod{A}$, i.e. $\exact{}$ admits enough projectives, enough injectives and these two classes coincide. If $R$ is a projective in $(\lMod{Q,A},\exact{Q})$ then, by \prpref{Prj-transfer}(a), it is a summand of an object of the form $\bigoplus_{q\in Q} \Fq{q}(T_q)\cong \prod_{q \in Q}\Gq{\mathbb{S}q}(T_q)$ with $T_q$ projective (and therefore injective) in $(\lMod{A},\exact{})$ for all $q$ in $Q$. By \prpref{Prj-transfer}(b) this shows that $R$ is also injective in $(\lMod{Q,A},\exact{Q})$. Similarly, every injective object in $(\lMod{Q,A},\exact{Q})$ is projective, and hence $\exact{Q}$ is a Frobenius exact structure in $\lMod{Q,A}$.
\end{rmk}

\subsubsection{Gorenstein projective and Gorenstein injective objects}\label{relative-GPrj}
\prpref[Proposition~]{Prj-transfer} provides descriptions of the projective/injective objects in the exact category $(\lMod{Q,A},\exact{Q})$ in terms of the projective/injective objects in $(\lMod{A},\exact{})$. We will now prove a similar description of the Gorenstein projective/injective objects in the exact category $(\lMod{Q,A},\exact{Q})$ in terms of the Gorenstein projective/injective objects in $(\lMod{A},\exact{})$; see Subsection~\ref{Gorenstein-projective-objects}. We obtain this in \thmref[Theorems~]{GPrj} and \thmref[]{GInj}. Knowing what such objects look like will allow us to apply \thmref[Theorems~]{main} and \thmref[]{main-op} to the exact category $\mathcal{E} = (\lMod{Q,A},\exact{Q})$, where $\mathscr{G}$ and $\mathscr{J}$ are suitable classes of stalk functors on projective and injective objects in $(\lMod{A},\exact{})$, see \corref[Corollaries~]{GPrj-stalk} and \corref[]{GInj-stalk}. 

\begin{lem}
  \label{lem:GPrj}
  Let $\exact{}$ be an exact structure in $\lMod{A}$ such that $(\lMod{A},\exact{})$ has enough projectives and exact coproducts. If $\{M_q\}_{q \in Q}$ is a family of Gorenstein projective objects in $(\lMod{A},\exact{})$, then $\prod_{q \in Q}\Gq{q}(M_q)$ is a Gorenstein projective object in $(\lMod{Q,A},\exact{Q})$.
\end{lem}

\begin{proof}
Let $\mathbb{S}$ be the Serre functor on $Q$ (see \stpref{HJ}). As noted in \rmkref{Frobenius} there are isomorphisms
\begin{equation*}
  \textstyle
  \prod_{q \in Q}\Gq{q}(M_q)
  \;\cong\;
  \bigoplus_{q \in Q}\Gq{q}(M_q)
  \;\cong\;
  \bigoplus_{q \in Q}\Fq{\mspace{3mu}\mathbb{S}^{-1}q}(M_q) \;.
\end{equation*}
By \lemref{exact-coprod}, $(\lMod{Q,A},\exact{Q})$ has exact coproducts, so it follows from \lemref{GPrj-Ext}(b) that the class of Gorenstein projective objects in $(\lMod{Q,A},\exact{Q})$ is closed under coproducts. It then suffices to argue that for any object $p$ in $Q$ and every Gorenstein projective object $M$ in $(\lMod{A},\exact{})$, the object $\Fq{p}(M)$ is
Gorenstein projective in $(\lMod{Q,A},\exact{Q})$. Let $P$ be a totally acyclic complex of projectives in $(\lMod{A},\exact{})$ with $\Cy{0}{P} \cong M$. We know from \lemref{FEG-exact} and \prpref{Prj-transfer}(a) that $\Fq{p}$ preserves conflations and projective objects, so $\Fq{p}(P)$ is an acyclic complex of projectives in $(\lMod{Q,A},\exact{Q})$ with $\Cy{0}{\Fq{p}(P)} \cong \Fq{p}(M)$. To prove that $\Hom{Q,A}{\Fq{p}(P)}{R}$ is exact for every projective object $R$ in $(\lMod{Q,A},\exact{Q})$, we use the isomorphism $\Hom{Q,A}{\Fq{p}(P)}{R}\cong\Hom{A}{P}{\Eq{p}(R)}$ given by the adjoint pair $(\Fq{p},\Eq{p})$. Note that the latter complex is exact since $P$ is a totally acyclic complex of projectives in $(\lMod{A},\exact{})$, and $\Eq{p}(R) = R(p)$ is projective in $(\lMod{A},\exact{})$ by \prpref{Prj-transfer}(c).
\end{proof}

\begin{lem}
  \label{lem:exact-eqc}
  Let $\exact{}$ be an exact structure in $\lMod{A}$ such that $(\lMod{A},\exact{})$ has enough projectives. For a complex $P$ in the category $(\lMod{Q,A},\exact{Q})$ the following conditions are equivalent:
\begin{eqc}
\item $\Hom{Q,A}{P}{R}$ is exact for every projective object $R$ in $(\lMod{Q,A},\exact{Q})$.
\item $\Hom{A}{P(q)}{T}$ is exact for every projective object $T$ in $(\lMod{A},\exact{})$ and $q$ in $Q$. 
\end{eqc}
\end{lem}

\begin{proof}
By \prpref{Prj-transfer}(a) every projective object $R$ in $(\lMod{Q,A},\exact{Q})$ is a direct summand of an object of the form $\bigoplus_{p \in Q}\Fq{p}(T_p)$, for some family $\{T_p\}_{p \in Q}$ of projective objects in $(\lMod{A},\exact{})$. Therefore, condition \eqclbl{i} is equivalent to $\Hom{Q,A}{P}{\bigoplus_{p \in Q}\Fq{p}(T_p)}$ being exact for every family $\{T_p\}_{p \in Q}$ of projective objects in $(\lMod{A},\exact{})$. By \cite[Lems.~3.4 and 3.7]{HJ-TAMS} one has
\begin{equation*}
  \textstyle
  \bigoplus_{p \in Q}\Fq{p}(T_p)
  \;\cong\;
  \prod_{p \in Q}\Fq{p}(T_p)
  \;\cong\;
  \prod_{p \in Q}\Gq{\mathbb{S}p}(T_p)
\end{equation*}
where $\mathbb{S}$ is the Serre functor on $Q$, see \stpref{HJ}. Consequently, one has
\begin{equation*}
 \textstyle
 \Hom{Q,A}{P}{\bigoplus_{p \in Q}\Fq{p}(T_p)} 
 \;\cong\; 
 \prod_{p \in Q} \Hom{Q,A}{P}{\Gq{\mathbb{S}p}(T_p)} \;,
\end{equation*} 
and since products are exact in $\lMod{\Bbbk}$, it follows that condition \eqclbl{i} is then equivalent to $\Hom{Q,A}{P}{\Gq{q}(T)}$ being exact for every projective object $T$ in $(\lMod{A},\exact{})$
and $q$ in $Q$. Since we have
\begin{equation*}
  \Hom{Q,A}{P}{\Gq{q}(T)}
  \;\cong\;
  \Hom{A}{\Eq{q}(P)}{T}
  \;=\;
  \Hom{A}{P(q)}{T}
\end{equation*}
by the adjunction $(\Eq{q},\Gq{q})$, it follows that conditions \eqclbl{i} and \eqclbl{ii} are equivalent.
\end{proof}

\begin{thm}
  \label{thm:GPrj}
  Let $\exact{}$ be an exact structure in $\lMod{A}$ such that $(\lMod{A},\exact{})$ has enough projectives and exact coproducts. An object $X$ in the exact category $(\lMod{Q,A},\exact{Q})$ is Gorenstein projective if and only if $X(q)$ is Gorenstein projective in the exact category $(\lMod{A},\exact{})$ for every $q$ in $Q$.
\end{thm}

\begin{proof}  
Let $X$ be a Gorenstein projective object in $(\lMod{Q,A},\exact{Q})$, $P$ a totally acyclic complex of projectives in $(\lMod{Q,A},\exact{Q})$ with $\Cy{0}{P} \cong X$, and $q$ an object in $Q$. By \lemref{FEG-exact} and \prpref{Prj-transfer}(c)
 the functor $\Eq{q}$ preserves conflations and projective objects, showing that $P(q)$ is an acyclic complex of projectives in $(\lMod{A},\exact{})$ with $\Cy{0}{P(q)} \cong X(q)$. To see that $X(q)$ is Gorenstein projective in $(\lMod{A},\exact{})$ it remains to see that $\Hom{A}{P(q)}{T}$ is exact for every projective object $T$ in $(\lMod{A},\exact{})$. However, this follows from the assertion \eqclbl{i}$\;\Rightarrow\;$\eqclbl{ii} in \lemref{exact-eqc}.
  
Let now $X$ be an object in $\lMod{Q,A}$ for which $X(q)$ is Gorenstein projective in $(\lMod{A},\exact{})$ for every $q$ in $Q$. To prove that $X$ is Gorenstein projective in $(\lMod{Q,A},\exact{Q})$ we will construct a totally acyclic complex $P$ of projectives in $(\lMod{Q,A},\exact{Q})$ with $\Cy{0}{P} = X$, by splicing together two acyclic complexes,
\begin{equation*}
    \varrho^- \;=\;
    \xymatrix@C=1.0pc{
    \cdots \ar[r] &
    P^{-2} \ar[r] &
    P^{-1} \ar[r] &
    X \ar[r] & 0
    }
  \quad \text{ and } \quad 
    \varrho^+ \;=\;
    \xymatrix@C=1.0pc{
    0 \ar[r] &
    X \ar[r] &
    P^0 \ar[r] &
    P^1 \ar[r] & \cdots 
    },
\end{equation*}  
in $(\lMod{Q,A},\exact{Q})$ where each $P^n$ is projective, such that $\Hom{A}{\varrho^-(q)}{T}$ and $\Hom{A}{\varrho^+(q)}{T}$ are exact in $\lMod{\Bbbk}$ for every projective $T$ in $(\lMod{A},\exact{})$ and $q$ in $Q$. This will then suffice in view of \lemref{exact-eqc}. 

By \prpref{Prj-transfer}(a), $(\lMod{Q,A},\exact{Q})$ has enough projectives and we can take $\varrho^-$ to be an augmented projective resolution of $X$ in $(\lMod{Q,A},\exact{Q})$. Now, for any $q$ in $Q$, the functor $\Eq{q}$ preserves conflations and projective objects (see \lemref{FEG-exact} and \prpref{Prj-transfer}(c)) and, thus, $\varrho^-(q)$ is an augmented projective resolution of $X(q)$ in $(\lMod{A},\exact{})$. By assumption, $X(q)$ is Gorenstein projective in $(\lMod{A},\exact{})$ and therefore, by \lemref{GPrj-Ext}, we have $\Ext{\mathcal{E}}{i}{X(q)}{T}=0$ for any projective $T$ in $(\lMod{A},\exact{})$ and $i>0$, showing that $\Hom{A}{\varrho^-(q)}{T}$ is exact for every $T$ in $(\lMod{A},\exact{})$.

To construct $\varrho^+$, it suffices to build a conflation $\rho$ in $(\lMod{Q,A},\exact{Q})$, satisfying the following properties.
\begin{rqm}
\item $\rho$ is of the form $0 \to X \to R  \to X' \to 0$ with $R$ being a projective object in $(\lMod{Q,A},\exact{Q})$.
\item $X'(q)$ is Gorenstein projective in $(\lMod{A},\exact{})$ for every $q$ in $Q$.
\item $\Hom{A}{\rho(q)}{T}$ is exact for every projective $T$ in $(\lMod{A},\exact{})$ and $q$ in $Q$. 
\end{rqm}
Note that condition (2) ensures that we can iterate the construction of such conflations, thus building $\varrho^+$. To construct $\rho$ note that as each $X(q)$ is Gorenstein projective, there are conflations in $(\lMod{A},\exact{})$ of the form
\begin{equation*}
    \xymatrix{
    0 \ar[r] &
    X(q) \ar[r]^-{\iota_q} &
    P_q \ar[r]^-{\pi_q} &
    C_q \ar[r] & 0
    },
\end{equation*}  
with $P_q$ projective and $C_q$ Gorenstein projective in $(\lMod{A},\exact{})$. By assumption, $(\lMod{A},\exact{})$ admits enough projectives, so it follows from \rmkref{exact-prod} that products are exact in $(\lMod{Q,A},\exact{Q})$. Since $\Gq{q}$ preserves conflations by \lemref{FEG-exact}, we get an induced conflation in $(\lMod{Q,A},\exact{Q})$, where the first term is $\mathbb{G}(X)$ from (\ref{eq:FF-and-GG})
\begin{equation*}
    \xymatrix@C=1.7pc{
    0 \ar[r] &
    \mathbb{G}(X)=\prod_{q \in Q}\Gq{q}(X(q)) 
    \ar[rr]^-{\iota:=\prod_{q \in Q}\Gq{q}(\iota_q)} & &
    P:=\prod_{q \in Q}\Gq{q}(P_q) 
    \ar[rr]^-{\pi:=\prod_{q \in Q}\Gq{q}(\pi_q)} & &
    C:=\prod_{q \in Q}\Gq{q}(C_q) \ar[r] & 0.
    }
\end{equation*}  
As noted in \rmkref{Frobenius} there is an isomorphism $P\cong \bigoplus_{q \in Q}\Fq{\mathbb{S}^{-1}q}(P_q)$, and so $P$ is projective in $(\lMod{Q,A},\exact{Q})$ by \prpref{Prj-transfer}(a). The objects $\mathbb{G}(X)$ and $C$ are Gorenstein projective in $(\lMod{Q,A},\exact{Q})$ by \lemref{GPrj}. We now consider the exact sequences $\mathfrak{g}_X$ in $\lMod{Q,A}$ from (\ref{eq:varepsilon-eta}), which we recall to be objectwise split i.e.~for each $q$ in $Q$ the sequence $\Eq{q}(\mathfrak{g}_X)$ splits in $\lMod{A}$. In particular, $\Eq{q}(\mathfrak{g}_X)$ is a conflation in $(\lMod{A},\exact{})$, and hence $\mathfrak{g}_X$ is a conflation in $(\lMod{Q,A},\exact{Q})$. Seting $X' = \Coker{(\iota\,\eta^X)}$, we have a commutative diagram with exact rows and columns in the abelian category $\lMod{Q,A}$ given by the Snake Lemma:
\begin{equation}
  \label{eq:3x3}
  \hspace*{-1.5ex}
  \begin{gathered}
  \xymatrix@R=1.5pc{
    {} & 0 \ar[d] & 0 \ar[d] & 0 \ar[d] & {}
    \\
    0 \ar[r] & 
    X \ar[r]^-{=} \ar[d]^-{\eta^X} & 
    X \ar[r] \ar[d]^-{\iota\mspace{1mu}\eta^X} & 
    0 \ar[r] \ar[d] & 0
    \\
    0 \ar[r] & 
    \mathbb{G}(X) 
    \ar[r]^-{\iota} \ar[d] & 
    P \ar[r]^-{\pi} \ar[d]^-{\kappa} & 
    C \ar[r] \ar[d]^-{=} & 0
    \\
    0 \ar[r] & 
    \mathbb{C}(X) \ar[r] \ar[d] & 
    X' \ar[r]^-{\lambda} \ar[d] & 
    C \ar[r] \ar[d] & 0
    \\
    {} & 0 & 0 & 0 & {}    
  }
  \end{gathered}
\end{equation}
We argue that all rows and columns in \eqref{3x3} are conflations in $(\lMod{Q,A},\exact{Q})$. Indeed, the first row and the third column are trivial conflations (see \cite[Def.~2.1 {[E$0$]} and {[E$0^\mathrm{op}$]}]{Buhler}), and we have already seen that the second row and the first column are conflations in $(\lMod{Q,A},\exact{Q})$. We now show that  the third row is also a conflation and the \mbox{$3\!\times\!3$}-Lemma for exact categories (see \cite[Cor.~3.6(ii)]{Buhler}, with rows and columns interchanged) will then guarantee that the middle column in \eqref{3x3}, which we will call $\rho$, is also a conflation in $(\lMod{Q,A},\exact{Q})$. Indeed, to argue that the third row is a conflation in $(\lMod{Q,A},\exact{Q})$, it suffices to show that $\lambda$ is a deflation. However, as $\lambda\kappa = \pi$ is a deflation, \cite[Prop.~7.6]{Buhler} yields hat $\lambda$ is a deflation, as wanted.

It remains to show that this conflation $\rho$ has the properties (1)--(3). As already noted, $P$ is a projective object in $(\lMod{Q,A},\exact{Q})$. Fix a projective object $T$ in $(\lMod{A},\exact{})$ and $q$ in $Q$. Let us prove property (2). Since $\Eq{q}$ is exact (\lemref{FEG-exact}), applying it to the third row in \eqref{3x3} yields a conflation 
\begin{equation*}
  0 \longrightarrow \mathbb{C}(X)(q) \longrightarrow X'(q) \longrightarrow C(q) \longrightarrow 0
\end{equation*} 
in $(\lMod{A},\exact{})$. To see that $X'(q)$ is Gorenstein projective in $(\lMod{A},\exact{})$ it suffices by \prpref{GPrj-resolving} to show that $\mathbb{C}(X)(q)$ and $C(q)$ are Gorenstein projective. As noted above, both $\mathbb{G}(X)$ and $C$ are Gorenstein projective in $(\lMod{Q,A},\exact{Q})$ so, by the first part of this proof, we have that $\mathbb{G}(X)(q)$ and $C(q)$ are Gorenstein projective in $(\lMod{A},\exact{})$. Since $\Eq{q}(\mathfrak{g}_X)$ splits, $\mathbb{C}(X)(q)$ is a direct summand of $\mathbb{G}(X)(q)$ and it is therefore Gorenstein projective, again by \prpref{GPrj-resolving}. 
%This is the only place in the proof we use the assumption that $\lMode{A}{\tau}$ has exact coproducts.
Finally, let us prove property (3). The conflation $\rho$ yields a conflation $\rho(q)=\Eq{q}(\rho)$ in the exact category $\mathcal{E}= (\lMod{A},\exact{})$. As $X'(q)$ is Gorenstein projective in $\mathcal{E}$ by condition (2), and $T$ is projective in $\mathcal{E}$, one has $\Ext{\mathcal{E}}{1}{X'(q)}{T}=0$ by Lemma \ref{lem:GPrj-Ext}. It follows that the sequence $\Hom{A}{\rho(q)}{T}$ is exact.
\end{proof}

We consider now the following functors defined using the left and right stalks mentioned in equation \eqref{stalks}:
\begin{equation*}
\Sq{q} 
\;=\; 
\stalkco{q} \otimes_\Bbbk - 
\;\cong\;
\Hom{\Bbbk}{\stalkcn{q}}{-}
\colon \lMod{A} \longrightarrow \lMod{Q,A}.
\end{equation*}
Recall from \cite[Prop.~7.15]{HJ-JLMS} that $\Sq{q}$ admits a left adjoint $\Cq{q}$ and a right adjoint $\Kq{q}$.

\begin{cor}
  \label{cor:GPrj-stalk}
Let $\exact{}$ be an exact structure in $\lMod{A}$ such that $(\lMod{A},\exact{})$ has enough projectives and exact coproducts. Then $\Sq{q}(T)$ is a Gorenstein projective object in $(\lMod{Q,A},\exact{Q})$ for every Gorenstein projective (in particular, for every projective) object $T$ in $(\lMod{A},\exact{})$ and $q$ in $Q$.
\end{cor}

\begin{proof}
By \cite[Lem.~7.10 and Prop.~7.15]{HJ-JLMS} one has $\Sq{q}(T)(q) = T$ and $\Sq{q}(T)(p) = 0$ for $p \neq q$ in $Q$. The objects $T$ and $0$ are Gorenstein projective in $(\lMod{A},\exact{})$, so the claim then follows from \thmref{GPrj}.
\end{proof}

The next theorem and its corollary are proved dually to \thmref{GPrj} and \corref{GPrj-stalk}.

\begin{thm}
  \label{thm:GInj}
  Let $\exact{}$ be an exact structure in $\lMod{A}$ such that $(\lMod{A},\exact{})$ has enough injectives and exact products. An object $X$ in the exact category $(\lMod{Q,A},\exact{Q})$ is Gorenstein injective if and only if $X(q)$ is Gorenstein injective in the exact category $(\lMod{A},\exact{})$ for every $q$ in $Q$. \qed
\end{thm}

\begin{cor}
  \label{cor:GInj-stalk}
Let $\exact{}$ be an exact structure in $\lMod{A}$ such that $(\lMod{A},\exact{})$ has enough injectives and exact products. Then $\Sq{q}(I)$ is a Gorenstein injective object in $(\lMod{Q,A},\exact{Q})$ for every Gorenstein injective (in particular, for every injective) object $I$ in $(\lMod{A},\exact{})$ and $q$ in $Q$. \qed
\end{cor}

\subsection{Examples}
\label{examples}
\phantom{.} \vspace*{1ex}

\noindent
To apply \thmref{GPrj} and its \corref{GPrj-stalk} we need an exact structure $\exact{}$ in $\lMod{A}$ such that 
\begin{rqm}
\item[($*$)] $(\lMod{A},\exact{})$ has enough projectives and exact coproducts.
\end{rqm}
In this case, if $\mathscr{T}$ is a class of projective objects in $(\lMod{A},\exact{})$, then \corref{GPrj-stalk} shows that 
\begin{equation}
  \label{eq:G-stalk}
  \mathscr{G} \,=\, 
  \mathscr{G}(\mathscr{T}) \,=\,
  \{\Sq{q}(T) \,|\, T \in \mathscr{T},\; q \in Q\}
\end{equation}  
is a class of Gorenstein projective objects in $(\lMod{Q,A},\exact{Q})$. Our goal in the next section is to apply \thmref{main} to such the class $\mathscr{G}$ in the exact category $\mathcal{E} = (\lMod{Q,A},\exact{Q})$. As stated in \thmref{main}, this can be done in either of the following situations:
\begin{rqm}
\item[($\dagger$)] If $\mathscr{T}$ (and hence $\mathscr{G}$) is a \textsl{set} and $(\lMod{Q,A},\exact{Q})$ is efficient or, 
\item[($\ddagger$)] If the cotorsion pair $(\mathcal{C}_\mathscr{G},\mathcal{W}_\mathscr{G})$ is complete.
\end{rqm}
Thus, we are searching for situations where \,($*$)\,+\,($\dagger$)\, or \,($*$)\,+\,($\ddagger$)\, holds. In this section, we present examples of such situations, see \prpref{Mod-theta}(d) and \exaref{Prj-PPrj-Mod}.

\begin{dfn}
  \label{dfn:tau-exact}
  Let $\theta$ be a class of left $A$-modules. There is an exact structure on $\lMod{A}$ where the conflations are the short exact sequences $0 \to M' \to M \to M'' \to 0$ in the abelian category $\lMod{A}$ for which 
\begin{equation*}
  0 \longrightarrow
  \Hom{A}{T}{M'} \longrightarrow 
  \Hom{A}{T}{M} \longrightarrow 
  \Hom{A}{T}{M''} \longrightarrow 
  0 
\end{equation*}  
is exact in $\lMod{\Bbbk}$ for every $T$ in $\theta$ (see e.g.~\cite[Prop.~3.2]{MR4658661}). We call this the \textbf{\mbox{$\theta$-exact} structure} on $\lMod{A}$. When considering $\lMod{A}$ endowed with this exact structure, we write $\lMode{A}{\theta}$. The exact structure induced in $\lMod{Q,A}$ by the $\theta$-exact structure on $\lMod{A}$ is called the \textbf{objectwise $\theta$-exact structure} on $\lMod{Q,A}$. When considering $\lMod{Q,A}$ endowed with this exact structure, we write $\lMode{Q,A}{\theta}$.
\end{dfn}

Note that, for any class $\theta$, the objects in $\theta$ are projective in the exact category $\lMode{A}{\theta}$, and it is easy to see that 
\begin{equation}
  \label{eq:saturation}
 \lMode{A}{\theta} \;=\; \lMode{A}{\Prjp{\lMode{A}{\theta}}} \;.
\end{equation}  
We call $\Prjp{\lMode{A}{\theta}}$ the \textbf{saturation} of $\theta$ and say that $\theta$ is \textbf{saturated} if $\theta=\Prjp{\lMode{A}{\theta}}$.  It follows that, in order to consider all exact structures of this kind, it suffices to consider saturated classes in $\lMod{A}$.  Every saturated class of $A$-modules contains $\lPrj{A}$ (the projective modules in the abelian category $\lMod{A}$) and is closed under coproducts and direct summands. 
In fact, we can say more under some assumptions on  $\theta$.

\begin{lem}\label{lem:saturation}\label{lem:precovering}
Let $\theta$ be a class of left $A$-modules.
\begin{enumerate}
    \item If $\theta$ is a precovering class in $\lMod{A}$, then $\lMode{A}{\theta}$ has enough projectives.
    \item If $\theta$ is a set (not a proper class), then $\lMode{A}{\theta}$ has enough projectives and $\Prjp{\lMode{A}{\theta}}=
  \mathrm{Add}(A \,\oplus\, \bigoplus_{T\in\theta}T).$
\end{enumerate}
\end{lem} 
\begin{proof}
(1): Given an object $M$ in $\lMode{A}{\theta}$, consider a $\theta$-precover $\alpha\colon T\longrightarrow M$. Let $\beta\colon A^{(I)}\longrightarrow M$ be an epimorphism of $A$-modules, for a suitable set $I$. We claim that $\gamma\colon T\oplus A^{(I)}\longrightarrow M$ sending $(t,a)$ to $\gamma(t,a)=\alpha(t)+\beta(a)$, for any $(t,a)$ in  $T\oplus A^{(I)}$ is a deflation in $\lMode{A}{\theta}$. It is clearly an epimorphism, and for any $T'$ in $\theta$ and any $f\colon T'\longrightarrow M$, the map $f$ factors through the $\theta$-precover $\alpha$. Thus, $f$ factors through $\gamma$, and $\Hom{A}{T'}{\gamma}$ is surjective, as wanted. 

(2) Note that, without loss of generality, we may assume $\theta=\operatorname{Add}(T)$ for $T=\oplus_{T\in\theta}T$, since $\lMode{A}{\theta}=\lMode{A}{\operatorname{Add}(T)}$. Assuming this, it follows that $\theta$ is precovering and, by (1),  we get that $\lMode{A}{\theta}$ has enough projectives. Now, for any $P$ in $\Prjp{\lMode{A}{\theta}}$, consider a $\operatorname{Add}(A\oplus T)$-precover $\psi$ of $P$. This is, in particular, an $\operatorname{Add}(T)=\theta$-precover and, thus, $\mathsf{Hom}(T',\psi)$ is surjective for all $T'$ in $\theta$, showing that $\psi$ is a deflation. Since $P$ is projective, $\psi$ splits, as wanted.
%Set \mbox{$M=A \,\oplus\, \bigoplus_{T\in\theta}T$} and let $P$ be an object in $\Prjp{\lMode{A}{\theta}}$. Consider the set $H=\Hom{A}{M}{P}$ and the (universal) $\mathrm{Add}(M)$-precover $\psi\colon M^{(H)}\longrightarrow P$. If we show that $\psi$ is a deflation in $\lMode{A}{\theta}$, then it splits since $P$ is projective, thus proving our claim. First, since $A$ is a direct summand of $M$, the map $\psi$ is surjective. Secondly, since $M^{(H)}$ is a direct sum of projectives in $\lMode{A}{\theta}$, $M^{(H)}$ is projective. Now, for any $T$ in $\theta$, the homomorphism $\Hom{A}{T}{\psi}$ is surjective by the fact that $\psi$ is an $\mathrm{Add}(M)$-precover, thus proving that $\psi$ is a deflation, as wanted.
\end{proof}

Note that if $\theta$ is the saturation of a set $\alpha$ of finitely presented $A$-modules containing $A$, then from \lemref{saturation} we get that $\alpha$ is a set of projective generators in $\lMode{A}{\theta}$.

The following result lists some useful properties of the exact categories $\lMode{A}{\theta}$ and $\lMode{Q,A}{\theta}$.

\begin{prp}
\label{prp:Mod-theta}
Let $\theta$ be a saturated class of left $A$-modules. The following statements hold.
\begin{prt}
\item The exact categories $\lMode{A}{\theta}$ and $\lMode{Q,A}{\theta}$ have exact products.

\item If the exact category $\lMode{A}{\theta}$ has exact coproducts, then so does $\lMode{Q,A}{\theta}$.

\item If $\theta$ is precovering, then $\lMode{A}{\theta}$ has enough projectives and one has:
\begin{rqm}
\item $\lMode{Q,A}{\theta}$ has enough projectives.
\item $\Prjp{\lMode{Q,A}{\theta}}\,=\,\operatorname{Add}\,\{\,\Fq{q}(T) \;|\; T \in \theta,\;  q \in Q \,\}
%=\operatorname{Add}\,\{\,\Gq{q}(T) \;|\; T \in \theta,\;  q \in Q \,\}
$.

\item If $P$ is a projective object in $\lMode{Q,A}{\theta}$, then $\Eq{q}(P)=P(q)$ is projective in $\lMode{A}{\theta}$ for all $q$ in $Q$.
\end{rqm}

\item If $\theta$ be is the saturation of a class of finitely presented $A$-modules, then $\lMode{A}{\theta}$ and $\lMode{Q,A}{\theta}$ have enough projectives and are efficient; in particular, they have exact coproducts.
\end{prt}
\end{prp}

\begin{proof}
Part (a) follows immediately from \dfnref{tau-exact}, and part (b) is a special case of \lemref{exact-coprod}.

\proofoftag{c} \lemref{precovering} shows that $\lMode{A}{\theta}$ has enough projectives if $\theta$ is precovering. Parts (1) and (2) are now special cases \prpref{Prj-transfer}(a)
% and \cite[Lem.~3.4]{HJ-TAMS}; the latter asserts that $\Gq{q} \cong \Fq{\mspace{3mu}\mathbb{S}^{-1}q}$ where $\mathbb{S}$ is the Serre functor on $Q$ (see \stpref{HJ}). 
and part (3) is a special case of \prpref{Prj-transfer}(c).

\proofoftag{d} Since there is only a set of isomorphism classes of finitely presented $A$-modules, $\theta$ is also the saturation of a set(!) of finitely presented $A$-modules. Note that $\lMode{A}{\theta}$ (and, therefore, by (c), $\lMode{Q,A}{\theta}$) has enough projectives by \lemref{saturation} and \eqref{saturation}. We show that $\lMode{A}{\theta}$ is efficient and the efficiency of $\lMode{Q,A}{\theta}$ then follows   \prpref{efficient-transfer}. Suppose now that $\theta=\Prjp{\lMode{A}{\alpha}}$ for some set $\alpha$ of finitely presented $A$-modules. Since $\Hom{A}{T}{-}$ commutes with direct limits for any $T$ in $\alpha$ and direct limits are exact in the category of $\Bbbk$-modules, it follows that transfinite compositions of inflations in $\lMode{A}{\theta}$ are still inflations. Since $\lMod{A}$ is a Grothendieck category, every object is small with respect to all morphisms (see \cite[Lem.~4.4 and Subsect.~1.1]{SaorinStovicek}) and, thus, small with respect to the class of inflations in $\lMode{A}{\theta}$. This concludes the proof that $\lMode{A}{\theta}$ is efficient, by \cite[Prop.~2.7]{SaorinStovicek}.
\end{proof}

\begin{exa}
\label{exa:Prj-PPrj-Mod}
  We mention a few canonical examples of saturated classes of $A$-modules $\theta$ for which the exact categories $\lMode{A}{\theta}$ and $\lMode{Q,A}{\theta}$ have enough projectives and are efficient (in particular, they have exact coproducts).
  \begin{prt}
  \item If $\theta = {}_A\operatorname{Prj}$ is the class of projective $A$-modules, then the conflations in $\lMode{A}{\theta}$ are precisely the usual exact sequences in the abelian structure. As $\theta$ is the saturation of the singleton set $\{A\}$ (see  \lemref{saturation}) $\lMode{A}{\theta}$ and $\lMode{Q,A}{\theta}$ are efficient. (This is well-known from \cite[Ex.~2.8(1)]{SaorinStovicek} as, in this case, $\lMode{A}{\theta}$ and $\lMode{Q,A}{\theta}$ are just the abelian categories $\lMod{A}$ and $\lMod{Q,A}$, which are Grothendieck).
  
  \item If $\theta = {}_A\operatorname{PPrj}$ is the class of pure projective $A$-modules, then the conflations in $\lMode{A}{\theta}$ are precisely the pure exact sequences \cite[Def.~6.5]{mta}. As $\theta$ is the saturation of the class of finitely presented $A$-modules, see \lemref{saturation}, both $\lMode{A}{\theta}$ and $\lMode{Q,A}{\theta}$ are efficient with enough projectives by \prpref{Mod-theta}(d). 

  \item If $\theta = \lMod{A}$ is the class of all $A$-modules, then the conflations in $\lMode{A}{\theta}$ are precisely the split exact sequences. The class $\theta$ is trivially precovering and, therefore, $\lMode{A}{\theta}$ and $\lMode{Q,A}{\theta}$ have enough projectives by \prpref{Mod-theta}(c). It is also easy to check that the split exact structure in a Grothendieck abelian category (in fact, in any accessible additive category with set-indexed coproducts, see \cite[Ex.~2.8(2)]{SaorinStovicek}), like the category of $A$-modules, is efficient. Thus, both $\lMode{A}{\theta}$ and $\lMode{Q,A}{\theta}$ are efficient. Moreover, it follows from \rmkref{Frobenius} that both $\lMode{A}{\theta}$ and $\lMode{Q,A}{\theta}$ are Frobenius.
  \end{prt}
  
%More generally, for any set(!) $\mathbb{T}$ of $A$-modules that contains $A$, it is easily seen that the class $\tau = \operatorname{Add}(\mathbb{T})$ satisfies conditions (1) and (2) in \prpref{Prj}. Further, if every module in $\mathbb{T}$ is finitely generated, then $\lMode{A}{\tau}$ and $\lMode{Q,A}{\tau}$ have exact coproducts. Note that the examples in (a) and (b) have this form with $\mathbb{T}=\{A\}$ and $\mathbb{T}$ a set of representatives for the isomorphism classes of all finitely presented $A$-modules.
%The class $\tau = {}_A\operatorname{Flat}$ of flat $A$-modules also satisfies conditions (1) and (2) in \prpref{Prj}; indeed, every $A$-module has a flat cover by Bican, El Bashir, and Enochs \cite{BEE-01}. However, we do not expect $\lMode{A}{\tau}$ and $\lMode{Q,A}{\tau}$ to have exact coproducts in this case.
\end{exa}

\begin{prp}\label{prp:injective cotorsion pair is complete}
Consider the class $\theta=\lMod{A}$ and the exact categories $\lMode{A}{\theta}$ and $\lMode{Q,A}{\theta}$. Then, the injective (complete) cotorsion pair $(\lMode{Q,A}{\theta},\Inj{\lMode{Q,A}{\theta}})$ coincides with the cotorsion pair in $\lMode{Q,A}{\theta}$ generated by the set
%Set $\theta = \lMod{A}$. The cotorsion pair in the exact category $\lMode{Q,A}{\theta}$ generated by the class 
\begin{equation*}
  \mathscr{G} \,=\, 
  \{\Sq{p}(M) \;|\; M \in \theta,\; p \in Q\}.
\end{equation*}  
%is the injective cotorsion pair $(\lMode{Q,A}{\theta},\Inj{(\lMode{Q,A}{\theta})})$, and it is therefore complete.
\end{prp}

\begin{proof}
Let $\mathscr{E}$ denote the exact structure in $\lMode{A}{\theta}$ (recall that this is the split exact structure) and $\exact{Q}$ denote the exact structure in $\lMode{Q,A}{\theta}$ (recall that this is the objectwise split exact structure). We show that, in the exact category $\lMode{Q,A}{\theta}$, we have an equality $\mathscr{G}^\perp=\Inj{\lMode{Q,A}{\theta}}$. One inclusion is clear. For the other, we take $X$ in $\mathscr{G}^\perp$ and we will show that $X\cong \prod_{q\in Q}\Gq{q}(I_q)$ for a family of injective objects $(I_q)_{q\in Q}$ in $\lMode{A}{\theta}$, concluding the proof by Proposition \ref{prp:Prj-transfer}(b). Note that every object in $\lMode{A}{\theta}$ is injective, and we will in fact show that $X\cong \prod_{q\in Q}\Gq{q}\Kq{q}(X)$, where $\Kq{q}$ is the right adjoint of $\Sq{q}$, for any $q$ in $Q$ (see \cite[Prop.~7.15]{HJ-JLMS}).

We first claim that $\Kq{q}$ sends a conflation $\xi$ in $\Ext{\exact{Q}}{1}{Z}{X}$ to a conflation $\Kq{q}(\xi)$ in $\mathscr{E}$. Applying $\Hom{\exact{Q}}{\Sq{q}(M)}{-}$ for an arbitrary left $A$-module $M$ to the the sequence $\xi$ we get a short exact sequence $\Hom{\exact{Q}}{\Sq{q}(M)}{\xi}$ since $\Ext{\exact{Q}}{1}{\mathscr{G}}{X}=0$. By adjunction, we know that $\Hom{\exact{Q}}{\Sq{q}(M)}{\xi}\cong \Hom{\mathscr{E}}{M}{\Kq{q}(\xi)}=\Hom{A}{M}{\Kq{q}(\xi)}$. This shows, as wanted, that $\Kq{q}(\xi)$ is a conflation in $\mathscr{E}$.

Recall the short exact sequence $\mathfrak{g}_X$ from (\ref{eq:varepsilon-eta}). This is an objectwise split exact sequence in $\lMod{Q,A}$ and, hence, a conflation in $\exact{Q}$. Note that by the paragraph above, we have that $\Kq{q}(\mathfrak{g}_X)$ is a conflation in $\exact{}$ and, therefore, $\Kq{q}(\eta^x)$ is a split monomorphism, for any $q$ in $Q$.  Now, we have
%Now, for each $p$ in $Q$ there is by \cite[Prop.~7.18]{HJ-JLMS} a canonical monomorphism $\iota^X_p \colon \Kq{p}(X) \to X(p)$. We will show that this is even a split monomorphism. Note that 
\begin{equation}
  \label{eq:Kp-prod}
  \textstyle
  \Kq{q}(\mathbb{G}(X)) \;=\;
  \Kq{q}\big(\prod_{p \in Q} \Gq{p}(X(p)) \big) \;\cong\;
  \prod_{p \in Q} \Kq{q}(\Gq{p}(X(p))) \;\cong\;  
  X(q)\;.
\end{equation}
where the equality holds by the definition of $\mathbb{G}$ (see \eqref{FF-and-GG}), while the first isomorphism follows from the fact that $\Kq{p}$, being a right adjoint, preserves products. The second isomorphism holds by \cite[Prop.~7.28(b)]{HJ-JLMS}, where it is shown that $\Kq{q}\Gq{p}=0$ if $p\neq q$ and $\Kq{q}\Gq{q}=\mathrm{id}$, for any $p$ and $q$ in $Q$. Since these isomorphisms are natural, we identify $\Kq{q}(\eta^X)$ with a split monomorphism $\Kq{q}(X)\rightarrow X(q)$, and we choose a left inverse $\pi_q$ for it. We define the composition $\varphi_q\colon X\rightarrow \Gq{q}\Kq{q}(X)$ as follows
%It follows that the map $\iota^X_p$ can be identified with \smash{$\Kq{p}(\eta^X)$}, which we know is split monic by the first part of the proof.
%Now, choose for each $p \in Q$ a left inverse $\pi_p$ of $\iota^X_p$ and let $\varphi_p$ be the composite
\begin{equation*}
\varphi_q\colon   \xymatrix@C=2.5pc{
    X \ar[r]^-{\eta_q^X} &
    \Gq{q}\Eq{q}(X) \,=\, \Gq{q}(X(q)) \ar[r]^-{\Gq{q}(\pi_q)} &
    \Gq{q}\Kq{q}(X)
  },
\end{equation*}
where $\eta^X_p$ is the unit of the adjunction $(\Eq{p},\Gq{p})$. Again using \cite[Prop.~7.28(b)]{HJ-JLMS} we see that
\begin{equation*}
  \label{eq:Kp-varphi}
  \Kq{q}(\varphi_q) \;=\; 
  \Kq{q}(\Gq{q}(\pi_q) \circ \eta_q^X) \;=\;
  \Kq{q}\Gq{q}(\pi_q) \circ \Kq{q}(\eta_q^X) \;=\;
  \pi_q \circ \Kq{q}(\eta_q^X) \;=\; \mathrm{id}_{\Kq{q}(X)}\;.
\end{equation*}
The family $\{\varphi_q\}_{q \in Q}$ induces a unique morphism $\varphi \colon X \to \prod_{q \in Q} \Gq{q}\Kq{q}(X)$ by the universal property of the product. It is clear that, since $\Kq{q}$ commutes with products, \cite[Prop.~7.28(b)]{HJ-JLMS} also implies that $\Kq{q}(\varphi)=\Kq{q}(\varphi_q)=\mathrm{id}_{\Kq{q}(X)}$.

%with $\rho_p\varphi = \varphi_p$ for every $p \in Q$ where $\rho_p \colon \prod_{q \in Q} \Gq{q}\Kq{q}(X) \twoheadrightarrow \Gq{p}\Kq{p}(X)$ is the projection corresponding to $p$. As in \eqref{Kp-prod} one gets
%\begin{equation*}
 % \textstyle
 % \Kq{p}(\prod_{q \in Q} \Gq{q}\Kq{q}(X)) 
 % \;=\; 
  %\Kq{p}(X) 
  %\qquad \text{and hence} \qquad \Kq{p}(\rho_p) \;=\; %\mathrm{id}_{\Kq{p}(X)}\;. 
%\end{equation*}
%As $\rho_p\varphi = \varphi_p$ it follows that $\Kq{p}(\rho_p)\Kq{p}(\varphi) = \Kq{p}(\varphi_p)$. As shown above, $\Kq{p}(\rho_p)$ and $\Kq{p}(\varphi_p)$ are both the identity map on $\Kq{p}(X)$ and hence so is $\Kq{p}(\varphi)$.

We prove that $\varphi$ is an isomorphism. Since for any $q$ in $Q$, $\Kq{q}$ is left exact (as it is a right adjoint), we have that $\Kq{q}(\Ker{\varphi})=\Ker{\Kq{q}(\varphi)}$, which is $0$ since $\Kq{q}(\varphi)$ is an isomorphism. The functors $\{\Kq{q}\colon q\in Q\}$ vanish on an object if and only if that object is zero (see \cite[Prop.~7.19]{HJ-JLMS}) and, hence, we conclude that $\varphi$ is a monomorphism. Consider now the short exact sequence in $\lMod{Q,A}$ induced by $\varphi$. Applying $\Kq{q}$ to it, we get a long exact sequence
\[
\textstyle
0\longrightarrow K_q(X)\stackrel{K_q(\varphi)}{\longrightarrow}K_q(\prod_{p\in Q}G_pK_p(X))\stackrel{0}{\longrightarrow}K_q(\Coker{\varphi})\longrightarrow \mathbb{R}^1\Kq{q}(X)\longrightarrow\cdots\,,
\]
where the third arrow represents the zero map since $K_q(\varphi)$ is an isomorphism. It follows from \cite[Lem.~4.3]{HJ-JLMS} (taking $G=\stalkco{p}$ therein) that the first right derived functor of $\Kq{q}$ satisfies $\mathbb{R}^1\Kq{q}(X) \cong \Ext{\lMod{Q,A}}{1}{\Sq{p}(A)}{X}$. We claim that the latter is zero, showing then that also $\Kq{q}(\Coker{\varphi})=0$ for all $q$ in $Q$ and, thus, that $\Coker{\varphi}$ is an epimorphism, as wanted. Indeed, note that 
  $  \Ext{\lMod{Q,A}}{1}{\Sq{q}(A)}{X} 
    \,\cong\,
    \Ext{\exact{Q}}{1}{\Sq{q}(A)}{X}.$
Since $A$ is projective, every short exact sequence $0 \to X \to E \to \Sq{q}(A) \to 0$ in the abelian exact structure is automatically objectwise split exact, and hence a conflation in $\lMode{Q,A}{\theta}$. The latter Ext-group then vanishes by assumption on $X$, as wanted.
\end{proof}

\subsection{The canonical totally acyclic complex of projectives / injectives}
\label{the-canonical-tac}
\phantom{.} \vspace*{1ex}

The following construction is taken from \cite[Sect.~5]{HJ-TAMS}. It is \emph{a priori} a construction in the abelian category $\lMod{Q,A}$, but we argue that it is still applicable when $\lMod{Q,A}$ is equipped with the objectwise exact structure $\exact{Q}$ induced by an exact structure $\exact{}$ in $\lMod{A}$. The aim is to build, at least in some cases, totally acyclic complexes with respect to the exact structure $\exact{Q}$ whose zero-th cycle is a fixed Gorenstein projective object.

\begin{con}
  \label{con:sigma-computation}
Let $\exact{}$ be an exact structure in $\lMod{A}$ such that $(\lMod{A},\exact{})$ has enough projectives and exact coproducts. Consider an object $X$ in $\lMod{Q,A}$ with the property:
\begin{center}
($\dagger$) $X(q)$ is a projective object in $(\lMod{A},\exact{})$ for every $q \in Q$.
\end{center}
By \thmref{GPrj} this implies that $X$ is a Gorenstein projective object in the exact category $(\lMod{Q,A},\exact{Q})$, so by \dfnref{totally-acyclic} there exists a totally acyclic complex of projectives in $(\lMod{Q,A},\exact{Q})$ with $X$ as its zeroth cycle. Below we give an explicit construction of such a complex $\mathbb{T}(X)$, which we call the canonical totally acyclic complex of projectives for $X$. Note that this construction is functorial in $X$. 

First we observe that under the assumption $(\dagger)$, we have that
\begin{enumerate}
    \item the objects $\mathbb{F}(X)$ and $\mathbb{G}(X)$ from (\ref{eq:FF-and-GG}) are projective objects in $(\lMod{Q,A},\exact{Q})$;
    \item for every $q$ in $Q$, the objects $\mathbb{K}(X)(q)$ and $\mathbb{C}(X)(q)$ (see Subsection \ref{Prelim-on-Q-shaped}) are projective object in $(\lMod{A},\exact{})$.
\end{enumerate}
Indeed, as each $X(q)$ is projective in $(\lMod{A},\exact{})$, the definition of $\mathbb{F}(X)$ and \prpref{Prj-transfer}(a) shows that $\mathbb{F}(X)$ is projective in $(\lMod{Q,A},\exact{Q})$. In particular, each $\mathbb{F}(X)(q)$ is projective in $(\lMod{A},\exact{})$ by \prpref{Prj-transfer}(c), and since $\mathbb{K}(X)(q)$ is a direct summand of $\mathbb{F}(X)(q)$, it follows that $\mathbb{K}(X)(q)$ is projective as well. The definition of $\mathbb{G}(X)$, combined with the isomorphisms in \rmkref{Frobenius} and another application of \prpref{Prj-transfer}(a), shows that $\mathbb{G}(X)$ is projective in $(\lMod{Q,A},\exact{Q})$. As in the previous case, it now follows that each $\mathbb{C}(X)(q)$ is projective in $(\lMod{A},\exact{})$.

By part (2) above, the objects $\mathbb{K}(X)$ and $\mathbb{C}(X)$ have the same property ($\dagger$) as $X$. In particular they are Gorenstein projective in $(\lMod{Q,A},\exact{Q})$ by \thmref{GPrj}, so the procedure above can be iterated. Thus, if we write the $i$-fold compositions of $\mathbb{K}$ and $\mathbb{C}$ as $\mathbb{K}^i \,=\, \mathbb{K} \circ \cdots \circ \mathbb{K}$
and $\mathbb{C}^i \,=\, \mathbb{C} \circ \cdots \circ \mathbb{C}$, then the objects $\mathbb{K}^i(X)$ and $\mathbb{C}^i(X)$ have the property that $\mathbb{K}^i(X)(q)$ and $\mathbb{C}^i(X)(q)$ are projective in $(\lMod{A},\exact{})$ for every $q$ in $Q$, and $\mathbb{F}(\mathbb{K}^i(X))$ and $\mathbb{G}(\mathbb{C}^i(X))$ are projective in $(\lMod{Q,A},\exact{Q})$. Note that the short exact sequences of the form
\begin{equation*}
  \xymatrix@C=1.0pc{
    0 \ar[r] & \mathbb{K}^{i+1}(X) \ar[r] & \mathbb{F}(\mathbb{K}^i(X)) \ar[r] & \mathbb{K}^i(X) \ar[r] & 0
  }
  \quad \text{ and } \quad
  \xymatrix@C=1.0pc{
    0 \ar[r] & \mathbb{C}^i(X) \ar[r] & \mathbb{G}(\mathbb{C}^i(X)) \ar[r] & \mathbb{C}^{i+1}(X) \ar[r] & 0
  }
\end{equation*} 
for $i \geqslant 0$ are all objectwise split (see Subsection \ref{Prelim-on-Q-shaped}) and, therefore, they are conflations in $\exact{Q}$. If we glue them together, we obtain an acyclic complex of projectives in $(\lMod{Q,A},\exact{Q})$ of the form
\begin{equation*}
  \xymatrix@R=1.0ex@C=1.1pc{
    {} & {} &
    \mathbb{T}^{-3}(X) \ar@{=}[d] &
    \mathbb{T}^{-2}(X) \ar@{=}[d] &
    \mathbb{T}^{-1}(X) \ar@{=}[d] &
    \mathbb{T}^0(X) \ar@{=}[d] &
    \mathbb{T}^1(X) \ar@{=}[d] &
    \mathbb{T}^2(X) \ar@{=}[d] &
    \\  
    \mathbb{T}(X) \,=  & \cdots \ar[r] &
    \mathbb{F}(\mathbb{K}^2(X)) \ar[r] &
    \mathbb{F}(\mathbb{K}(X)) \ar[r] &
    \mathbb{F}(X) \ar[r] &
    \mathbb{G}(X) \ar[r] &
    \mathbb{G}(\mathbb{C}(X)) \ar[r] &
    \mathbb{G}(\mathbb{C}^2(X)) \ar[r] &
    \cdots
  }
\end{equation*}
 with $\Cy{0}{\,(\mathbb{T}(X))} = X$. This complex is furthermore totally acyclic by Lemma \ref{lem:exact-eqc}, since for all $q$ in $Q$, $\mathbb{T}(X)(q)$ is a splicing of split exact sequences. We call this complex the \textbf{canonical totally acyclic complex of projectives for $X$.} Following the notation set before \dfnref{H} we can choose the following representatives of the projective equivalence class of the cycles of $\mathbb{T}(X)$ as follows:
\begin{equation}
  \label{eq:sigma-computation}
  \begin{gathered}
  \left\{
  \setlength{\arraycolsep}{2pt}
  \renewcommand{\arraystretch}{1.1}    
  \begin{array}{rclclcl}
    \sigma^{-n}X &=& \Cy{-n}{\,(\mathbb{T}(X))} &=& \mathbb{K}^n(X) & \ \ \text{for} \ \ & n>0    \\
    \sigma^0X &=& \Cy{0}{\,(\mathbb{T}(X))} &=& X & & 
    \\
    \sigma^nX &=& \Cy{n}{\,(\mathbb{T}(X))} &=& \mathbb{C}^n(X) & \ \ \text{for} \ \ & n>0\;.
  \end{array}
  \right.
  \end{gathered}
\end{equation}
Note that since $\mathbb{K}^0$ and $\mathbb{C}^0$ are both equal to the identity functor we have $\mathbb{K}^0(X) = X = \mathbb{C}^0(X)$.

Dually, if an object $Y$ is an object of $\lMod{Q,A}$ with the property:
\begin{center}
($\ddagger$) $Y(q)$ is an injective object in $(\lMod{A},\exact{})$ for every $q \in Q$,
\end{center}
then we know from \thmref{GInj} that $Y$ is a Gorenstein injective object in the exact category $(\lMod{Q,A},\exact{Q})$. In this case, $\mathbb{T}(Y)$ becomes the canonical totally acyclic complex of \textsl{injectives} in $(\lMod{Q,A},\exact{Q})$ for the object $Y$, and the formulas \eqref{sigma-computation} are still valid, cf.~\rmkref{efficient}.
\end{con}

\begin{rmk} Under the weaker assumption that each $X(q)$ is a \textsl{Gorenstein projective} object in $(\lMod{A},\exact{})$, $X$ would still be a Gorenstein projective object  in $(\lMod{Q,A},\exact{Q})$ by \thmref{GPrj}, and hence is the zeroth cycle of a totally acyclic complex of projectives in $(\lMod{Q,A},\exact{Q})$. However, under this weaker assumption, the complex $\mathbb{T}(X)$ would \textsl{not}, in general, consist of projectives in $(\lMod{Q,A},\exact{Q})$ because part (1) of the construction fails.
\end{rmk}

Recall from \cite[Def.~5.1]{HJ-TAMS} that the \textbf{support} of an object $X$ in $\lMod{Q,A}$ is $\supp{X} = \{q \in Q \;|\; X(q) \neq 0 \}$.

\begin{rmk}
  \label{rmk:values}
  Adopt the setup of \conref{sigma-computation}, including the assumption ($\dagger$). One has:
\begin{prt}
\item The functors $\mathbb{F}(X)$, $\mathbb{G}(X)$, $\mathbb{K}(X)$, and $\mathbb{C}(X)$ always take values in $\operatorname{Add} \{X(q) \,|\, q \in Q \} \subseteq \Prj{(\lMod{A},\exact{})}$.
\item If $X$ has finite support, then $\mathbb{F}(X)$, $\mathbb{G}(X)$, $\mathbb{K}(X)$, and $\mathbb{C}(X)$ have finite support and these functors even take values in $\operatorname{add} \{X(q) \,|\, q \in Q \}$.
\end{prt}  
Indeed, for the functors $\mathbb{F}(X)$ and $\mathbb{G}(X)$, the claim in part (a) follows immediately from the definitions in \eqref{FF-and-GG} and the isomorphisms in \rmkref{Frobenius}. For the functors $\mathbb{K}(X)$ and $\mathbb{C}(X)$, the claim in part (a) now follows as the sequences in \eqref{varepsilon-eta} are objectwise split. Part (b) follows immediately from \cite[Lem.~5.8]{HJ-TAMS} and its proof. 

Note that if $X$ has finite support, then successive applications of part (b) shows that for every $i \geqslant 0$ the functors $\mathbb{F}^i(X)$, $\mathbb{G}^i(X)$, $\mathbb{K}^i(X)$ and $\mathbb{C}^i(X)$ also have finite support and they take values in $\operatorname{add} \{X(q) \,|\, q \in Q \}$.
\end{rmk}

We now establish a useful property of the canonical totally acyclic complex of projectives.

\begin{prp}
  \label{prp:canonical-tac-Prj}
Let $\exact{}$ be an exact structure in $\lMod{A}$ such that $(\lMod{A},\exact{})$ has enough projectives and exact coproducts. If $U$ is an object in $\lMod{Q}$ such that $U(q)$ is a projective $\Bbbk$-module for every $q \in Q$, and $T$ is projective in the exact category $(\lMod{A},\exact{})$, then the functor $U \otimes_\Bbbk T$ in $\lMod{Q,A}$ takes values in $\Prj{(\lMod{A},\exact{})}$ and 
\begin{equation*}
  \mathbb{T}(U \otimes_\Bbbk T)
  \;\cong\;
  \mathbb{T}(U) \otimes_\Bbbk T
  \qquad \text{and} \qquad
  \sigma^n(U \otimes_\Bbbk T) \;\cong\; 
  (\sigma^n U) \otimes_\Bbbk T
  \ \text{ for } \
  n \in \mathbb{Z}\;.
\end{equation*}
\end{prp}

\begin{proof}
  For each $q \in Q$, since $U(q)$ lies in $\Prj{(\lMod{\Bbbk})} = \operatorname{Add}(\Bbbk)$ we have that 
  $(U \otimes_\Bbbk T)(q) \,=\, U(q) \otimes_\Bbbk T$ lies in  
  $\operatorname{Add}(T) \,\subseteq\, \Prj{(\lMod{A},\exact{})}$.  
  It then makes sense to consider its canonical totally acyclic complex of projectives in $(\lMod{Q,A},\exact{Q})$ following  \conref{sigma-computation}. To prove our assertion, it suffices (by construction) to argue that if we apply the functor $- \otimes_\Bbbk T$ to the canonical short exact sequences 
\begin{equation}
  \label{eq:canonical-ses-1}
  \xymatrix@C=1.5pc{
    \mathbb{K}(U) \ar@{>->}[r] & \mathbb{F}(U) \ar@{->>}[r]^-{\varepsilon^{U}} & U
  }
  \qquad \text{and} \qquad
  \xymatrix@C=1.5pc{
    U \ar@{>->}[r]^-{\eta^{U}} & \mathbb{G}(U) \ar@{->>}[r] & \mathbb{C}(U)
  }
\end{equation}  
in the abelian category $\lMod{Q}$, then we get, up to isomorphism, the canonical conflations  
\begin{equation}
  \label{eq:canonical-ses-2}
  \xymatrix@C=1.7pc{
    \mathbb{K}(U \otimes_\Bbbk T) \ar@{>->}[r] & \mathbb{F}(U \otimes_\Bbbk T) \ar@{->>}[r]^-{\varepsilon^{U \otimes_\Bbbk T}} & U \otimes_\Bbbk T
  }
  \quad \text{ and } \quad
  \xymatrix@C=1.7pc{
    U \otimes_\Bbbk T\ar@{>->}[r]^-{\eta^{U \otimes_\Bbbk T}} & \mathbb{G}(U \otimes_\Bbbk T) \ar@{->>}[r] & \mathbb{C}(U \otimes_\Bbbk T)
  }
\end{equation}  
in the exact category $(\lMod{Q,A},\exact{Q})$. First, observe that since the short exact sequences in \eqref{canonical-ses-1} are objectwise split, by tensoring with $T$ we get new short exact sequences with the same property, namely
\begin{equation}
  \label{eq:canonical-ses-3}
  \xymatrix@C=1.8pc{
    \mathbb{K}(U) \otimes_\Bbbk T \ar@{>->}[r] & \mathbb{F}(U) \otimes_\Bbbk T \ar@{->>}[r]^-{\varepsilon^{U} \otimes_\Bbbk T} & U \otimes_\Bbbk T
  }
  \quad \text{and} \quad
  \xymatrix@C=1.8pc{
    U \otimes_\Bbbk T\ar@{>->}[r]^-{\eta^{U} \otimes_\Bbbk T} & \mathbb{G}(U) \otimes_\Bbbk T \ar@{->>}[r] & \mathbb{C}(U) \otimes_\Bbbk T.
  }
\end{equation}  
The fact that they are objectwise split implies that they  are conflations in $(\lMod{Q,A},\exact{Q})$. For $q$ in $Q$, the definition \eqref{FEG-def} of $\Fq{q}$ and the associativity of the tensor product for $\Bbbk$-modules yield:
\begin{equation*}
  \Fq{q}(U(q)) \otimes_\Bbbk T
  \;\cong\;
  (Q(q,-) \otimes_\Bbbk U(q)) \otimes_\Bbbk T
  \,\cong\, 
  Q(q,-) \otimes_\Bbbk (U(q) \otimes_\Bbbk T)
  \,\cong\, 
  \Fq{q}((U\otimes_\Bbbk T)(q)) \;.
\end{equation*}
It now follows from the definition of $\mathbb{F}$, see \eqref{FF-and-GG}, and the fact that $- \otimes_\Bbbk T$ commutes with coproducts, that $\mathbb{F}(U) \otimes_\Bbbk T \cong \mathbb{F}(U \otimes_\Bbbk T)$. This isomorphism identifies $\varepsilon^U \otimes_\Bbbk T$ with $\varepsilon^{U \otimes_\Bbbk T}$, so it follows that the first sequence in \eqref{canonical-ses-2} is isomorphic to the first sequence in \eqref{canonical-ses-3}. Using the definition of $\mathbb{G}$, see \eqref{FF-and-GG}, and the isomorphism from \rmkref{Frobenius}, it now follows that $\mathbb{G}(U) \otimes_\Bbbk T \cong \mathbb{G}(U \otimes_\Bbbk T)$. This isomorphism identifies $\eta^U \otimes_\Bbbk T$ with $\eta^{U \otimes_\Bbbk T}$, so the second sequence in \eqref{canonical-ses-2} is isomorphic to the second sequence in \eqref{canonical-ses-3}.
\end{proof}

We define some useful functors, one of them based on Definition \ref{dfn:H}, generalising those from \eqref{HHhj}.

\begin{dfn}
  \label{dfn:new-HH}
  For a Gorenstein projective object $U$ of $\rMod{Q}$, $q$ in $Q$, $n$ and integer and $i>0$ we set
\begin{equation*} 
  \sHH{i}{U}{n} \;=\;
  \Ext{Q}{i}{\sigma^n U}{-} 
  \colon \lMod{Q} \longrightarrow \lMod{\Bbbk}\ \ \ \ \ \ \ \ \ \ 
   \sHHt{i}{U}{n} \;=\;
  \Tor{Q}{i}{\sigma^n U}{-} 
  \colon \lMod{Q} \longrightarrow \lMod{\Bbbk}\;.
\end{equation*}  
Note that \smash{$\sHH{i}{U}{n}$} is nothing but the functor $\HH{i}{U}{n}$ from \dfnref{H} in the special case where $\mathcal{E}$ is the abelian category $\lMod{Q}$. It is worth noting that the functors $\sHH{i}{U}{n}$ and $\sHHt{i}{U}{n}$ can be viewed as functors from $\lMod{Q,A} \to \lMod{A}$.
\end{dfn}
 
\begin{rmk}
\label{rmk:comparison-of-H}
The functors from \dfnref{new-HH} generalise the functors from    \eqref{HHhj}. Indeed, we have
\begin{equation*}  
  \sHH{i}{\stalkco{q}}{0} \;=\; \cHhj{i}{q}
  \qquad \text{and} \qquad
  \sHHt{i}{\stalkcn{q}}{0} \;=\; \hHhj{i}{q} \;.
\end{equation*} 
%Note that the category $Q$ is not reflected in the symbols..... 
\end{rmk}
By combining \prpref{canonical-tac-Prj} with the adjunctions from \cite[Props.~3.4 and 3.8]{HJ-JLMS}, we get two useful Ext formulas. Note that \corref{Ext-formula-exact-Prj} gives a formula for Ext in the (general) exact category $(\lMod{Q,A},\exact{Q})$, whereas \corref{Ext-formula-abelian-Prj} only gives a (different) formula for Ext in the abelian category $\lMod{Q,A}$.

\begin{cor}  
  \label{cor:Ext-formula-exact-Prj}
Let $\exact{}$ be an exact structure in $\lMod{A}$ such that $(\lMod{A},\exact{})$ has enough projectives and exact coproducts. Let $U$ be an object in $\lMod{Q}$ such that the $\Bbbk$-module $U(q)$ is projective for every $q$ in $Q$, let $T$ be a projective object in $(\lMod{A},\exact{})$, and $X$ be any object in $\lMod{Q,A}$. For every integer $n$ and $i > 0$ we have
\begin{equation*}
\Ext{(\lMod{Q,A},\exact{Q})}{i}{\sigma^n(U \otimes_\Bbbk T)}{X} 
\;\cong\; 
\Ext{Q}{i}{\sigma^n U}{\Hom{A}{T}{X}} \;.
\end{equation*}
In the notation of \dfnref[Definitions~]{H} and \dfnref[]{new-HH},
this reads
  $\HH{i}{U \otimes_\Bbbk T}{n}(X)\cong \sHH{i}{U}{n}(\Hom{A}{T}{X})$
where \textnormal{`$\Hletterone$'} and \textnormal{`$\mathbb{H}$'} are, respectively, the cohomology functors in the exact category $(\lMod{Q,A},\exact{Q})$  and in the abelian category $\lMod{Q}$.
\end{cor}

\begin{proof}
  Let $\mathbb{T}(U)$ be the canonical totally acyclic complex of projectives for $U$ in $\lMod{Q}$, see \conref{sigma-computation}. We know from \prpref{canonical-tac-Prj} that $\mathbb{T}(U) \otimes_\Bbbk T\cong \mathbb{T}(U \otimes_\Bbbk T)$, as totally acyclic complexes of projective objects in $(\lMod{Q,A},\exact{Q})$. Together with \eqref{Ext-HHom}, and using an adjunction from \cite[Prop.~3.4]{HJ-JLMS}, we get
  \begin{align*}
      \Ext{(\lMod{Q,A},\exact{Q})}{i}{\sigma^n(U \otimes_\Bbbk T)}{X} &
  \;\cong\; \operatorname{H}^{i-n+1}\Hom{Q,A}{\mathbb{T}(U \otimes_\Bbbk T)}{X}\; \cong\;
  \operatorname{H}^{i-n+1}\Hom{Q,A}{\mathbb{T}(U) \otimes_\Bbbk T}{X}\;\cong \\ 
 & \;\cong\;
  \operatorname{H}^{i-n+1}\Hom{Q}{\mathbb{T}(U)}{\Hom{A}{T}{X}}
  \;\cong\; 
  \Ext{Q}{i}{\sigma^n U}{\Hom{A}{T}{X}} \;. \qedhere
\end{align*}  
\end{proof}

%\begin{rmk}
%  \label{rmk:A-structure}
%For a Gorenstein projective object $U$ in the abelian category $\lMod{Q}$, the cohomology functor \smash{$\HH{i}{U}{n}(-) = \Ext{Q}{i}{\sigma^n U}{-}$} is defined to be a functor $\lMod{Q} \to \lMod{\Bbbk}$, see \dfnref{H} and \rmkref{well-defined}. However, $\HH{i}{U}{n}$ can be extended to a functor $\lMod{Q,A} \to \lMod{A}$, and this is used in the next result. Indeed, if $X \in \lMod{Q,A}$ then we can compute \smash{$\Ext{Q}{i}{\sigma^n U}{X}$} using a projective resolution $\cdots \to P_1 \to P_0 \to \sigma^n U \to 0$ in $\lMod{Q}$, and for each $i$ the $\Bbbk$-module $\Hom{Q}{P_i}{X}$ of $\Bbbk$-linear natural transformations from $P_i$ to $X$ has a natural $A$-module structure as $X$ takes values in $\lMod{A}$ (and not just in $\lMod{\Bbbk}$).
%\end{rmk}

\begin{cor}  
  \label{cor:Ext-formula-abelian-Prj}
Let $U$ be an object in $\lMod{Q}$ such that the $\Bbbk$-module $U(q)$ is projective for every $q \in Q$, let $L$ be a projective left $A$-module, and $X$ be any object in $\lMod{Q,A}$. For every integer $n$ and $i > 0$ we have an isomorphism
\begin{equation*}
\Ext{Q,A}{i}{\sigma^n(U \otimes_\Bbbk L)}{X} \;\cong\; 
\Hom{A}{L}{\Ext{Q}{i}{\sigma^n U}{X}}.
\end{equation*}
In the notation of \dfnref[Definitions~]{H} and \dfnref[]{new-HH}
this reads $\HH{i}{U \otimes_\Bbbk L}{n}(X)
  \ \cong \
  \Hom{A}{L}{\sHH{i}{U}{n}(X)}$
where \textnormal{`$\Hletterone$'} and \textnormal{`$\mathbb{H}$'} are, respectively, the cohomology functors in the abelian category $\lMod{Q,A}$ and in the abelian category $\lMod{Q}$. 
\end{cor}

\begin{proof}
  Let $\mathbb{T}(U)$ be the canonical totally acyclic complex of projectives for $U$ in $\lMod{Q}$, see \conref{sigma-computation}. A chain of isomorphisms as in the proof of   \corref{Ext-formula-exact-Prj} follows, again using an adjunction from \cite[Prop.~3.8]{HJ-JLMS} and using the fact that $\Hom{A}{L}{-} \colon \lMod{A} \to \lMod{\Bbbk}$ is exact and hence commutes with cohomology.
\begin{align*}
  \Ext{Q,A}{i}{\sigma^n(U \otimes_\Bbbk L)}{X} 
  &\cong 
  \operatorname{H}^{i-n+1}\Hom{Q,A}{\mathbb{T}(U \otimes_\Bbbk T)}{X}\cong \operatorname{H}^{i-n+1}\Hom{Q,A}{\mathbb{T}(U) \otimes_\Bbbk L}{X}
  \\
  \cong
  \operatorname{H}^{i-n+1}\Hom{A}{L}{\Hom{Q}{\mathbb{T}(U)}{X}} &\cong 
  \Hom{A}{L}{\operatorname{H}^{i-n+1}\Hom{Q}{\mathbb{T}(U)}{X}}
  \cong 
  \Hom{A}{L}{\Ext{Q}{i}{\sigma^n U}{X}} \;. \qedhere
\end{align*}  
\end{proof}

\prpref{canonical-tac-Prj} and \corref[Corollaries~]{Ext-formula-exact-Prj} and \corref[]{Ext-formula-abelian-Prj} have dual versions, which are given in \prpref[]{canonical-tac-Inj}, \corref[]{Ext-formula-exact-Inj}, and \corref[]{Ext-formula-abelian-Inj} below. Note that in these dual results, the object $U$ in $\lMod{Q}$ is replaced by an object $R$ in $\rMod{Q}$.

\begin{prp}
  \label{prp:canonical-tac-Inj}
Let $\exact{}$ be an exact structure in $\lMod{A}$ such that $(\lMod{A},\exact{})$ has enough injectives and exact products. If $R$ is an object in $\rMod{Q}$ be an object such that $R(q)$ is projective for every $q \in Q$, and $I$ is injective in the exact category $(\lMod{A},\exact{})$, then the functor $\Hom{\Bbbk}{R}{I} \in \lMod{Q,A}$ takes values in $\Inj{(\lMod{A},\exact{})}$ and 
\begin{equation*}
  \mathbb{T}(\Hom{\Bbbk}{R}{I})
  \;\cong\;
  \upSigma\mspace{1mu} \Hom{\Bbbk}{\mathbb{T}(R)}{I}
  \qquad \text{and} \qquad
  \sigma^{-n}\Hom{\Bbbk}{R}{I} \;\cong\; 
  \Hom{\Bbbk}{\sigma^n R}{I}
  \ \text{ for } \
  n \in \mathbb{Z}\;.
\end{equation*}
\end{prp}

\begin{proof}
This proof is similar to that of \prpref{canonical-tac-Prj}. However, note that a shift involved since one has
\begin{equation*}
  \Cy{-n}{\,(\upSigma\mspace{1mu}\Hom{\Bbbk}{\mathbb{T}(R)}{I})}
  \;\cong\;  
  \Cy{-n+1}{\,\Hom{\Bbbk}{\mathbb{T}(R)}{I}} 
  \;\cong\;
  \Hom{\Bbbk}{\Cy{n}{\,(\mathbb{T}(R))}}{I} \;.
\end{equation*}
Hence, it is the shifted complex $\upSigma\mspace{1mu}\Hom{\Bbbk}{\mathbb{T}(R)}{I}$, whose zeroth cycle is
\begin{equation*}
  \Cy{0}{\,(\upSigma\mspace{1mu}\Hom{\Bbbk}{\mathbb{T}(R)}{I})}
  \;\cong\;
  \Hom{\Bbbk}{\Cy{0}{\,(\mathbb{T}(R))}}{I} 
  \;=\;
  \Hom{\Bbbk}{R}{I}\;,
\end{equation*}
 that becomes the canonical totally acyclic complex of injectives for $\Hom{\Bbbk}{R}{I}$ in $(\lMod{Q,A},\exact{Q})$.
%{\color{red}
%We have:
%\begin{equation*}
%  \Hom{\Bbbk}{\mathbb{T}(R)}{I}^{-n} \;=\;
%  \Hom{\Bbbk}{\mathbb{T}^{n}(R)}{I}
%  \qquad \text{and} \qquad
%  \partial^{-n}_{\Hom{}{\mathbb{T}(R)}{I}} \;=\;
%  \Hom{\Bbbk}{\partial^{n-1}_{\mathbb{T}(R)}}{I} \;,
%\end{equation*}  
%so left exactness of $\Hom{\Bbbk}{-}{I}$ (see also the end of Subsection~\ref{efficient}) yields:
%\begin{align*}
%  \Cy{-n}{\,\Hom{\Bbbk}{\mathbb{T}(R)}{I}} &\;=\;
%  \Ker{\partial^{-n}_{\Hom{}{\mathbb{T}(R)}{I}}} \;=\;
%  \Ker{\Hom{\Bbbk}{\partial^{n-1}_{\mathbb{T}(R)}}{I}} 
%  \\
%  &\;\cong\;
%  \Hom{\Bbbk}{\Coker{\partial^{n-1}_{\mathbb{T}(R)}}}{I} \;=\;
%  \Hom{\Bbbk}{\Cy{n+1}{\,(\mathbb{T}(R))}}{I} \;.
%\end{align*}
%Thus, the shifted complex $\upSigma\mspace{1mu}\Hom{\Bbbk}{\mathbb{T}(R)}{I}$ satisfies
%\begin{equation*}
%  \Cy{-n}{\,(\upSigma\mspace{1mu}\Hom{\Bbbk}{\mathbb{T}(R)}{I})} \;\cong\;
%  \Hom{\Bbbk}{\Cy{n}{\,(\mathbb{T}(R))}}{I} \;.
%\end{equation*}
%}
\end{proof}

\begin{cor}  
  \label{cor:Ext-formula-exact-Inj}  
Let $\exact{}$ be an exact structure in $\lMod{A}$ such that $(\lMod{A},\exact{})$ has enough injectives and exact products. Let $R$ be an object in $\rMod{Q}$ such that the $\Bbbk$-module $R(q)$ is projective for every $q$ in $Q$, let $I$ be an injective object in $(\lMod{A},\exact{})$, and $X$ be any object in $\lMod{Q,A}$. For every integer $n$ and $i > 0$ we have
\begin{equation*}
  \Ext{(\lMod{Q,A},\exact{Q})}{i}{X}{\sigma^{-n}\Hom{\Bbbk}{R}{I}}
\;\cong\; 
\Ext{Q^\mathrm{op}}{i}{\sigma^n R}{\Hom{A}{X}{I}} \;.
\end{equation*}
In the notation of \eqref{HHop} and \dfnref[Definitions~]{H} and \dfnref[]{new-HH}, this reads  
  $\HHop{i}{\Hom{\Bbbk}{R}{I}}{n}(X)
  \ \cong \
  \sHH{i}{R}{n}(\Hom{A}{X}{I})$,
where \textnormal{`$\check{\Hletterone}$'} and \textnormal{`$\mathbb{H}$'} are, respectively, the (dual) cohomology functor in the exact category $(\lMod{Q,A},\exact{Q})$ and the (``projective'') cohomology functor in the abelian category $\lMod{Q^\mathrm{op}} = \rMod{Q}$.
\end{cor}

\begin{proof}
  Similar to the proof of \corref{Ext-formula-exact-Prj} (using   \prpref{canonical-tac-Inj} instead of \prpref[]{canonical-tac-Prj}).
\end{proof}

%%%%%%%%%%%%%%%%%%%%%%%%%%%%%%%%%%%%%%%%%%%%%%%%%%%%%%%%%%
%           DO NOT DELETE THE PROOF BELOW                %
%%%%%%%%%%%%%%%%%%%%%%%%%%%%%%%%%%%%%%%%%%%%%%%%%%%%%%%%%%
%\begin{proof}
%Let $P$ be the canonical totally acyclic complex of projectives for $R \in \rMod{Q}$, see \conref{sigma-computation}. We know from \prpref{canonical-tac-Inj} that $\upSigma\mspace{1mu}\Hom{\Bbbk}{P}{I}$ is a (even the canonical) totally acyclic complex of injectives for $\Hom{\Bbbk}{R}{I} \in (\lMod{Q,A},\exact{Q})$. In view of \eqref{Ext-HHom-inj} and \eqref{Ext-HHom}, this explains the first and last isomorphisms below. The second isomorphism is trivial, and the third isomorphism follows from ...
%\begin{align*}  
%  \Ext{(\lMod{Q,A},\exact{Q})}{i}{X}{\sigma^{-n}\Hom{\Bbbk}{R}{I}}
%  &\;\cong\;
%  \operatorname{H}^{i-n}\Hom{Q,A}{X}{\,\upSigma\mspace{1mu}\Hom{\Bbbk}{P}{I}}
%  \;\cong\;
%  \operatorname{H}^{i-n+1}\Hom{Q,A}{X}{\Hom{\Bbbk}{P}{I}}
%  \\
%  &\;\cong\;
%  \operatorname{H}^{i-n+1}\Hom{Q^\mathrm{op}}{P}{\Hom{A}{X}{I}}
%  \;\cong\; 
%  \Ext{Q^\mathrm{op}}{i}{\sigma^n R}{\Hom{A}{X}{I}} \;. \qedhere
%\end{align*}  
%\end{proof}

\begin{cor}  
  \label{cor:Ext-formula-abelian-Inj}
Let $R$ be an object in $\rMod{Q}$ such that the $\Bbbk$-module $R(q)$ is projective for every $q$ in $Q$, let $I$ be an injective left $A$-module, and $X$ be any object in $\lMod{Q,A}$. For every integer $n$ and $i > 0$ we have an isomorphism
\begin{equation*}
\Ext{Q,A}{i}{X}{\sigma^{-n}\Hom{\Bbbk}{R}{I}} \;\cong\; 
\Hom{A}{\Tor{Q}{i}{\sigma^n R}{X}}{I} \;.
\end{equation*}
In the notation of \eqref{HHop} and \dfnref{new-HH} this reads $\HHop{i}{\Hom{\Bbbk}{R}{I}}{n}(X) \cong \Hom{A}{\sHHt{i}{R}{n}(X)}{I}$,
where \textnormal{`$\check{\Hletterone}$'} is the (``injective'') cohomology functor in the abelian category $\lMod{Q,A}$.
\end{cor}

\begin{proof}
Let $\mathbb{T}(R)$ be the canonical totally acyclic complex of projectives for $R$ in $\rMod{Q}$, see \conref{sigma-computation}. We know from \prpref{canonical-tac-Inj} that $\upSigma\mspace{1mu}\Hom{\Bbbk}{\mathbb{T}(R)}{I}$ is the canonical totally acyclic complex of injectives for the functor $\Hom{\Bbbk}{R}{I}$ of $\lMod{Q,A}$. We have then the following sequence of isomorphisms
\begin{align*}  
  \Ext{Q,A}{i}{X}{\sigma^{-n}\Hom{\Bbbk}{R}{I}}
  &\;\cong\;
  \operatorname{H}^{i-n}\Hom{Q,A}{X}{\,\upSigma\mspace{1mu}\Hom{\Bbbk}{\mathbb{T}(R)}{I}}
  \;\cong\;
  \operatorname{H}^{i-n+1}\Hom{Q,A}{X}{\Hom{\Bbbk}{\mathbb{T}(R)}{I}}
  \\
  \;\cong\;
  \operatorname{H}^{i-n+1}\Hom{A}{\mathbb{T}(R) \otimes_Q X}{I}
  &\;\cong\; 
  \Hom{A}{\operatorname{H}^{n-i-1}(\mathbb{T}(R) \otimes_Q X)}{I} \;\cong\; \Hom{A}{\Tor{Q}{i}{\sigma^n R}{X}}{I}. 
\end{align*}  
Indeed, \eqref{Ext-HHom-inj} explains the first isomorphism, the second isomorphism is clear, the third one follows from the second adjunction in \cite[Prop.~3.8]{HJ-JLMS} and the fourth isomorphism holds as the functor $\Hom{A}{-}{I} \colon \lMod{A} \to \lMod{\Bbbk}$ is exact and, thus, commutes with cohomology, that is, $\operatorname{H}^{\ell}\Hom{A}{-}{I} \cong \Hom{A}{\operatorname{H}^{-\ell}(-)}{I}$. The final isomorphism uses the definition of $\Tor{Q}{*}{-}{X}$ as the left derived functors of $- \otimes_Q X \colon \rMod{Q} \to \lMod{A}$, finishing the proof.  
\end{proof}

\section{The \texorpdfstring{$Q$}{Q}-shaped derived category revisited and generalised}\label{sec:Q-shaped derived category}

In this section, we show how \thmref[Theorems~]{main} and \thmref[]{main-op} can be applied to recover the main results of \cite{HJ-JLMS} on $Q$-shaped derived categories. This is explained in \thmref[Theorems~]{Q-shaped-relative-Prj} and \thmref[]{Q-shaped-relative-Inj} and subsequent examples.

\subsection{The projective exact model structure on the exact category \texorpdfstring{$(\lMod{Q,A},\exact{Q})$}{(lMod(Q,A),exact(Q))}}
\label{projective-model-structure}
\phantom{.} \vspace*{1ex}

We begin with the result promised in Remark \ref{rmk:n-is-0}, setting up the groundwork to make sure an equivalence of of the form (\ref{eq:important-equivalence}) holds.
\begin{lem}
  \label{lem:Ext-comparison-Prj}
Let $\exact{}$ be an exact structure in $\lMod{A}$ such that $(\lMod{A},\exact{})$ has enough projectives and exact coproducts. Let $\mathscr{T}$ be any class of projective objects in $(\lMod{A},\exact{})$ and let $i>0$ be a fixed integer. For every $X$ in the exact category $(\lMod{Q,A},\exact{Q})$, the following conditions are equivalent:
\begin{eqc}
\item $\HH{i}{\Sq{q}(T)}{n}(X) := \Ext{\exact{Q}}{i}{\sigma^n\Sq{q}(T)}{X}=0$ for every $T$ in $\mathscr{T}$, $q$ in $Q$, and $n$ integer;

\item $\HH{i}{\Sq{q}(T)}{0}(X) := \Ext{\exact{Q}}{i}{\Sq{q}(T)}{X}=0$ for every $T$ in $\mathscr{T}$ and $q$ in $Q$.
\end{eqc}
\end{lem}

\begin{proof}
  Clearly \eqclbl{i} implies \eqclbl{ii} as $\sigma^0\Sq{q}(T) = \Sq{q}(T)$. Let us now assume condition \eqclbl{ii} and fix $T$ in $\mathscr{T}$, $q$ in  $Q$, and an integer $n$. Note that the additivity of the Ext-functors implies that condition $(ii)$ also holds when $T$ is an object of $\operatorname{add}\mathscr{T}$. By definition we have $\Sq{q}(T) = \stalkco{q} \otimes_\Bbbk T$ and so it follows from \prpref{canonical-tac-Prj} that 
  $\sigma^n\Sq{q}(T) \cong (\sigma^n\stalkco{q}) \otimes_\Bbbk T$
where $\sigma^n\Sq{q}(T)$ is computed in $(\lMod{Q,A},\exact{Q})$ and $\sigma^n\stalkco{q}$ in the abelian category $\lMod{Q}$.
%$P \otimes_\Bbbk T$ is the canonical totally acyclic complex of projectives for $\Sq{q}(T)$ in the exact category $(\lMod{Q,A},\exact{Q})$, and thus
%\begin{equation*}  
%  \sigma^n\Sq{q}(T) 
%  \,=\, 
%  \Cy{n}{\,(P \otimes_\Bbbk T)} 
%  \,=\,
%  (\Cy{n}{P}) \otimes_\Bbbk T 
%  \,=\,
%  (\sigma^n\stalkco{q}) \otimes_\Bbbk T \;.
%\end{equation*}  
Denote by $U$ the object $\sigma^n\stalkco{q}$ of $\lMod{Q}$. By the isomorphism above, our goal is to prove that
\begin{equation}
  \label{eq:Ext-goal}
  \Ext{\exact{Q}}{i}{U \otimes_\Bbbk T}{X} \,=\, 0 \;.
\end{equation}
Applying \cite[Construction 7.21(a)]{HJ-JLMS} (with $A=\Bbbk$) to the object $U$ in $\lMod{Q}$ we may define inductively, for $l\geq 0$, a family of short exact sequences in the abelian category $\lMod{Q}$ of the form
\begin{equation}
  \label{eq:xi-ell}
  \textstyle
  \xi^\ell \;=\;
  0 \longrightarrow \bigoplus_{p \in Q} \, \Sq{p}\Kq{p}(U^\ell) \longrightarrow U^\ell \longrightarrow U^{\ell+1} \longrightarrow 0
\end{equation}
 where $U^0 = U$. The functor $\stalkco{q}$ has finite support, namely $\{q\}$, and takes values in $\{0,\Bbbk\}$, so \eqref{sigma-computation} and \rmkref{values} show that the functor $U^0 = U = \sigma^n\stalkco{q}$ has finite support and takes values in the category $\operatorname{add}(\Bbbk) = \lprj{\Bbbk}$ of finitely generated projective $\Bbbk$-modules. As $\Bbbk$ is assumed to be noetherian and hereditary, see \stpref{HJ}, the class $\mathcal{G} = \lprj{\Bbbk}$ is closed under extensions and submodules, and hence successive applications of \cite[Prop.~7.20(a)]{HJ-JLMS} (with $A=\Bbbk$) show that $\Kq{p}(U^\ell)$ and  $U^\ell(p)$ lie in $\lprj{\Bbbk}$ for every $p$ in $Q$ and $\ell \geqslant 0$. Since $U^{\ell+1}(p)$ is projective for every $p$ in $Q$, the sequence \eqref{xi-ell}
 is objectwise split, and hence $\xi^\ell \otimes_\Bbbk T$ is an objectwise split exact sequence in $\lMod{Q,A}$. In particular, it is a conflation in the exact category $(\lMod{Q,A},\exact{Q})$. By \cite[Prop.~7.15]{HJ-JLMS} we have
\begin{equation*}
  \Sq{p}\Kq{p}(U^\ell) \otimes_\Bbbk T
  \;=\;
  (\stalkco{p} \otimes_\Bbbk \Kq{p}(U^\ell)) \otimes_\Bbbk T
  \;\cong\;
  \stalkco{p} \otimes_\Bbbk (\Kq{p}(U^\ell) \otimes_\Bbbk T)
  \;=\;
  \Sq{p}(\Kq{p}(U^\ell) \otimes_\Bbbk T)\;,
\end{equation*}
and since $-\otimes_\Bbbk T$ commutes with direct sums, we conclude that the conflation $\xi^\ell \otimes_\Bbbk T$ is given by
\begin{equation}
  \label{eq:xi-ell-T}
  \textstyle
  \xi^\ell \otimes_\Bbbk T \;=\;
  0 \longrightarrow \bigoplus_{p \in Q} \, \Sq{p}(T^\ell_{\!p}) \longrightarrow U^\ell \otimes_\Bbbk T \longrightarrow U^{\ell+1} \otimes_\Bbbk T \longrightarrow 0 \;,
\end{equation}
where $T^\ell_{\!p} = \Kq{p}(U^\ell) \otimes_\Bbbk T$. Since $\Kq{p}(U^\ell)$ is a projective $\Bbbk$-module, we get $T^\ell_{\!p}$ lies in $\operatorname{add}(T) \subseteq \operatorname{add}\mathscr{T}$ for every $\ell \geqslant 0$ and $p$ in $Q$, so our assumption  yields:
\begin{equation}
  \label{eq:Ext-stalks}
  \textstyle
  \Ext{\exact{Q}}{i}{\bigoplus_{p \in Q} \, \Sq{p}(T^\ell_{\!p})}{X} \,\cong\,
  \prod_{p \in Q}\Ext{\exact{Q}}{i}{\Sq{p}(T^\ell_{\!p})}{X} \,=\, 0 \;.
\end{equation}
By \stpref{HJ}, there is  $N>0$ such that $\mathfrak{r}^N=0$. As in the proof of \cite[Thm.~7.25]{HJ-JLMS} it follows that $U^N=0$ and thus, trivially, we get $\Ext{\exact{Q}}{i}{U^N \otimes_\Bbbk T}{X} \,=\, 0$. Suppose now that for some $0 \leqslant \ell < N-1$ we know that \smash{$\Ext{\exact{Q}}{i}{U^{\ell+1} \otimes_\Bbbk T}{X} = 0$}. As \eqref{xi-ell-T} is an $\exact{Q}$-conflation, it induces a long exact sequence, part of which reads:
\begin{equation*}
  \textstyle
  \Ext{\exact{Q}}{i}{U^{\ell+1} \otimes_\Bbbk T}{X} 
  \longrightarrow
  \Ext{\exact{Q}}{i}{U^\ell \otimes_\Bbbk T}{X} 
  \longrightarrow
  \Ext{\exact{Q}}{i}{\bigoplus_{p \in Q} \, \Sq{p}(T^\ell_{\!p})}{X} \;. 
\end{equation*}
The first group is zero by hypothesis, and the last group is zero by \eqref{Ext-stalks}. It follows that the middle group is zero. Inductively, we conclude that $\Ext{(\lMod{Q,A},\exact{Q})}{i}{U^\ell \otimes_\Bbbk T}{X} \,=\, 0$ for all $l\leq N$ and, thus, for $\ell=0$, as wanted.
\end{proof}

\begin{thm}
  \label{thm:Q-shaped-relative-Prj}
Let $\exact{}$ be an exact structure in $\lMod{A}$ such that $(\lMod{A},\exact{})$ has enough projectives and exact coproducts and let $\mathscr{T}$ be a class of projective objects in $(\lMod{A},\exact{})$. Consider in the exact category $(\lMod{Q,A},\exact{Q})$ the cotorsion pair $(\mathcal{C}(\mathscr{T}),\mathcal{W}(\mathscr{T}))$ generated by the class
\begin{equation*}
  \mathscr{G} \,=\, \mathscr{G}(\mathscr{T}) 
  \,=\, \{ \Sq{q}(T) \;|\; T \in \mathscr{T},\; q \in Q \} \;.
\end{equation*}
If this cotorsion pair is complete, then there is a hereditary exact model structure on $(\lMod{Q,A},\exact{Q})$ where the classes $\mathcal{W}(\mathscr{T})$ of trivial objects and of weak equivalences are described in parts (1) and (2) below, the class of cofibrant objects is $\mathcal{C}(\mathscr{T}) = {}^\perp(\mathcal{W}(\mathscr{T}))$, and every object in $(\lMod{Q,A},\exact{Q})$ is fibrant. 
\begin{rqm}
\item An object $X$ in $\lMod{Q,A}$ is trivial if and only if it satisfies the equivalent conditions:
\begin{eqc}
\setlength{\itemsep}{0pt}
\item $\HH{i}{\Sq{q}(T)}{n}(X)=0$ for all $T$ in $\mathscr{T}$, $q$ in $Q$, $n$ integer and $i>0$;

\item[\eqclbl{i$\,'$}] $\sHH{i}{\stalkco{q}}{n}(\Hom{A}{T}{X})=0$ for all $T$ in $\mathscr{T}$, $q$ in $Q$, $n$ integer and $i>0$;

\item $\HH{1}{\Sq{q}(T)}{0}(X)=0$ for all $T$ in $\mathscr{T}$ and $q$ in $Q$;

\item[\eqclbl{ii$\,'$}] $\sHH{1}{\stalkco{q}}{0}(\Hom{A}{T}{X})=0$ for all $T$ in $\mathscr{T}$ and $q$ in $Q$.
\end{eqc}

\item  A morphism $\varphi$ in $\lMod{Q,A}$ is a weak equivalence if and only if it satisfies the equivalent conditions:

\begin{eqc}
\setlength{\itemsep}{0pt}
\item $\HH{i}{\Sq{q}(T)}{n}(\varphi)$ is an isomorphism for all $T$ in $\mathscr{T}$, $q$ in $Q$, $n$ integer and $i>0$;.

\item[\eqclbl{i$\,'$}] $\sHH{i}{\stalkco{q}}{n}(\Hom{A}{T}{\varphi})$ is an isomorphism for all $T$ in $\mathscr{T}$, $q$ in $Q$, $n$ integer and $i>0$;

\item $\HH{1}{\Sq{q}(T)}{0}(\varphi)$ and $\HH{2}{\Sq{q}(T)}{0}(\varphi)$ are isomorphisms for all $T$ in $\mathscr{T}$ and $q$ in $Q$;

\item[\eqclbl{ii$\,'$}] $\sHH{1}{\stalkco{q}}{0}(\Hom{A}{T}{\varphi})$ and $\sHH{2}{\stalkco{q}}{0}(\Hom{A}{T}{\varphi})$ are isomorphisms for all $T$ in $\mathscr{T}$ and $q$ in $Q$.
\end{eqc}

\item $\mathcal{C}(\mathscr{T})$ is a Frobenius exact category whose class of projective-injective objects is precisely the class 
\begin{equation*}
  \Prj{(\lMod{Q,A},\exact{Q})} \;=\; 
  \operatorname{Add}\,\{\,\Fq{q}(T) \;|\; T \in \Prj{(\lMod{A},\exact{})},\;  q \in Q \,\} \;,
\end{equation*}
and the homotopy category $\operatorname{Ho}(\lMod{Q,A},\exact{Q})$ is equivalent to the stable category $\underline{\mathcal{C}}(\mathscr{T})$.
\end{rqm}
The functors $\Hletterone$ appearing in conditions of type \eqclbl{i} and  \eqclbl{ii} are cohomology functors in  $(\lMod{Q,A},\exact{Q})$, while the functors $\mathbb{H}$ appearing in conditions of type \eqclbl{i$\,'$} and \eqclbl{ii$\,'$} are cohomology functors in the abelian category $\lMod{Q}$.
\end{thm}

\begin{proof}
  The exact category $\mathcal{E}:=(\lMod{Q,A},\exact{Q})$ is weakly idempotent complete (as $\lMod{Q,A}$ is abelian) and has enough projectives by \prpref{Prj-transfer}(a). By \corref{GPrj-stalk}, $\mathscr{G} = \mathscr{G}(\mathscr{T})$ is a class of Gorenstein projective objects in $\mathcal{E}$. Thus, we are in the setup of \thmref{main}. Consider the cotorsion pair $(\mathcal{C}_\mathscr{G},\mathcal{W}_\mathscr{G})$ in $\mathcal{E}$ from \dfnref{CW}:
\begin{equation*}
  \mathcal{W}_{\mathscr{G}}=\{W \in \mathcal{E} \ | \ \HH{1}{\Sq{q}(T)}{n}(W)=0 \text{ for all } T \in \mathscr{T},\; q \in Q,\; n \in \mathbb{Z} \} \quad \text{ and } \quad
  \mathcal{C}_{\mathscr{G}}={}^\perp \mathcal{W}_\mathscr{G} \;.
\end{equation*}
\lemref{Ext-comparison-Prj} with $i=1$ shows that one has $\mathcal{W}_\mathscr{G} = \mathcal{W}(\mathscr{T})$, so the cotorsion pairs $(\mathcal{C}_\mathscr{G},\mathcal{W}_\mathscr{G})$ and $(\mathcal{C}(\mathscr{T}),\mathcal{W}(\mathscr{T}))$ coincide. We have assumed completenes of $(\mathcal{C}(\mathscr{T}),\mathcal{W}(\mathscr{T}))$, so the cotorsion pair $(\mathcal{C}_\mathscr{G},\mathcal{W}_\mathscr{G})$ is also complete. We can  therefore apply \thmref{main}, which yields an hereditary exact model structure on $\mathcal{E}$ where $\mathcal{C}_\mathscr{G} = \mathcal{C}(\mathscr{T})$ is the class of cofibrant objects, $\mathcal{W}_\mathscr{G} = \mathcal{W}(\mathscr{T})$ is the class of trivial
objects, and every object is fibrant. 

(1) It follows from \thmref{main}(1) and \lemref{Ext-comparison-Prj} (see also \rmkref{n-is-0}) that conditions  \eqclbl{i} and \eqclbl{ii} are equivalent and that these equivalent conditions characterise the trivial objects. Since one has $\Sq{q}(T) = \stalkco{q} \otimes_\Bbbk T$, \corref{Ext-formula-exact-Prj} with $U=\stalkco{q}$ yield the equivalences \eqclbl{i}$\;\Leftrightarrow\;$\eqclbl{i$\,'$} and \eqclbl{ii}$\;\Leftrightarrow\;$\eqclbl{ii$\,'$}.

(2) Similar to part (1), using \thmref{main}(2),  \lemref{Ext-comparison-Prj}, \rmkref{n-is-0}, and \corref{Ext-formula-exact-Prj}.

(3) This follows directly from \thmref{main}(3) in view of \prpref{Prj-transfer}(a).
\end{proof}

\begin{prp}
  \label{prp:assumptions-ok-Prj}
  Consider the case where the exact category $(\lMod{A},\exact{})$ is $\lMode{A}{\theta}$, and hence $(\lMod{Q,A},\exact{Q})$ is $\lMode{Q,A}{\theta}$, where $\theta$ is the saturation of a class of finitely presented $A$-modules. Further let $\mathscr{T}$ be any set(!) of modules from $\theta$. In this case, the assumptions in \thmref{Q-shaped-relative-Prj} hold, that is:
\begin{itemlist}
\item The category $(\lMod{A},\exact{}) = \lMode{A}{\theta}$ has enough projectives and exact coproducts.
\item The cotorsion pair $(\mathcal{C}(\mathscr{T}),\mathcal{W}(\mathscr{T}))$ in $(\lMod{Q,A},\exact{Q}) = \lMode{Q,A}{\theta}$ is complete.
\end{itemlist}  
Hence the conclusions in \thmref{Q-shaped-relative-Prj} are valid.
\end{prp}

\begin{proof}
  Since $\theta$ be the saturation of a class of finitely presented $A$-modules, \prpref{Mod-theta}(d) shows that the exact categories $\lMode{A}{\theta}$ and $\lMode{Q,A}{\theta}$ have enough projectives and are efficient; in particular, they have exact coproducts. By definition, the cotorsion pair $(\mathcal{C}(\mathscr{T}),\mathcal{W}(\mathscr{T}))$ is generated by $\mathscr{G} = \mathscr{G}(\mathscr{T})$, which is a set as $\mathscr{T}$ is assumed to be a set. Hence this cotorsion pair is complete by \cite[Cor.~2.15(3)]{SaorinStovicek}.
\end{proof}

\begin{exa}\label{exa:Q-shaped derived category revisited}
  Let $\theta = {}_A\operatorname{Prj}$, in which case $\lMode{A}{\theta}$ and $\lMode{Q,A}{\theta}$ are just the abelian categories $\lMod{A}$ and $\lMod{Q,A}$, see \exaref{Prj-PPrj-Mod}(a). Note that with $\mathscr{T}=\{A\}$ and hence
\begin{equation*}
  \mathscr{G} \,=\, \mathscr{G}(\mathscr{T}) 
  \,=\, \{ \Sq{q}(A) \;|\; q \in Q \} \;,
\end{equation*}
  part (1) of \thmref{Q-shaped-relative-Prj} specialises, in view of \rmkref{comparison-of-H} and \prpref{assumptions-ok-Prj}, to yield the equivalence of conditions (i), (ii), (ii') in \cite[Thm.~7.1]{HJ-JLMS} whereas part (2) of \thmref{Q-shaped-relative-Prj} specialises to yield the equivalence of (i), (ii), (ii') in \cite[Thm.~7.2]{HJ-JLMS}. Part (3) of \thmref{Q-shaped-relative-Prj} recovers the leftmost equivalence in \cite[Thm.~6.5]{HJ-JLMS}.
\end{exa}

\subsection{The injective exact model structure on the exact category \texorpdfstring{$(\lMod{Q,A},\exact{Q})$}{(lMod(Q,A),exact(Q))}}
\label{injective-model-structure}
\phantom{.} \vspace*{1ex}

The results of the previous subsection are dualisable. We include the dual statements for completeness.
\begin{lem}
  \label{lem:Ext-comparison-Inj}
Let $\exact{}$ be an exact structure in $\lMod{A}$ such that $(\lMod{A},\exact{})$ has enough injectives and exact products. Let $\mathscr{I}$ be any class of injective objects in $(\lMod{A},\exact{})$ and let $i>0$ be a fixed integer. For every $X$ in the exact category $(\lMod{Q,A},\exact{Q})$, the following conditions are equivalent:
\begin{eqc}
\item $\HHop{i}{\Sq{q}(I)}{n}(X) := \Ext{(\lMod{Q,A},\exact{Q})}{i}{X}{\sigma^{-n}\Sq{q}(I)}=0$ for every $I$ in $\mathscr{I}$, $q$ in $Q$, and $n$ integer;

\item $\HHop{i}{\Sq{q}(I)}{0}(X) = \Ext{(\lMod{Q,A},\exact{Q})}{i}{X}{\Sq{q}(I)}$ for every $I$  in $\mathscr{I}$ and $q$ in $Q$.
\end{eqc}
\end{lem}

\begin{proof}
  Similar to the proof of \lemref{Ext-comparison-Prj}.
\end{proof}

\begin{thm}
  \label{thm:Q-shaped-relative-Inj}
Let $\exact{}$ be an exact structure in $\lMod{A}$ such that $(\lMod{A},\exact{})$ has enough injectives and exact products and let $\mathscr{I}$ be a class of injective objects in $(\lMod{A},\exact{})$. Consider in the exact category $(\lMod{Q,A},\exact{Q})$ the cotorsion pair $(\mathcal{W}(\mathscr{I}),\mathcal{F}(\mathscr{I}))$ cogenerated by the class
\begin{equation*}
  \mathscr{J} \,=\, \mathscr{J}(\mathscr{I}) 
  \,=\, \{ \Sq{q}(I) \;|\; I \in \mathscr{I},\; q \in Q \} \;.
\end{equation*}
If this cotorsion pair is complete, then there is a hereditary exact model structure on $(\lMod{Q,A},\exact{Q})$ where the classes $\mathcal{W}(\mathscr{I})$ of trivial objects and of weak equivalences are described in parts (1) and (2) below, the class of fibrant objects is $\mathcal{F}(\mathscr{I}) = \mathcal{W}(\mathscr{I})^\perp$, and every object in $(\lMod{Q,A},\exact{Q})$ is cofibrant. 
\begin{rqm}
\item An object $X$ in $\lMod{Q,A}$ is trivial if and only if it satisfies the equivalent conditions:
\begin{eqc}
\setlength{\itemsep}{0pt}
\item $\HHop{i}{\Sq{q}(I)}{n}(X)=0$ for all $I$ in $\mathscr{I}$, $q$ in $Q$, $n$ integer, and $i>0$;

\item[\eqclbl{i$\,'$}] $\sHH{i}{\stalkcn{q}}{n}(\Hom{A}{X}{I})=0$ for all $I$ in $\mathscr{I}$, $q$ in $Q$, $n$ integer, and $i>0$;

\item $\HHop{1}{\Sq{q}(I)}{0}(X)=0$ for all $I$ in $\mathscr{I}$ and $q$ in $Q$;

\item[\eqclbl{ii$\,'$}] $\sHH{1}{\stalkcn{q}}{0}(\Hom{A}{X}{I})=0$ for all $I$ in $\mathscr{I}$, $q$ in $Q$.
\end{eqc}

\item  A morphism $\varphi$ in $\lMod{Q,A}$ is a weak equivalence if and only if it satisfies the equivalent conditions:

\begin{eqc}
\setlength{\itemsep}{0pt}
\item $\HHop{i}{\Sq{q}(I)}{n}(\varphi)$ is an isomorphism for all $I$ in $\mathscr{I}$, $q$ in $Q$, $n$ integer, and $i>0$;

\item[\eqclbl{i$\,'$}] $\sHH{i}{\stalkcn{q}}{n}(\Hom{A}{\varphi}{I})$ is an isomorphism for all $I$ in $\mathscr{I}$, $q$ in $Q$, $n$ integer, and $i>0$;

\item $\HHop{1}{\Sq{q}(I)}{0}(\varphi)$ and $\HH{2}{\Sq{q}(I)}{0}(\varphi)$ are isomorphisms for all $I$ in $\mathscr{I}$ and $q$ in $Q$;

\item[\eqclbl{ii$\,'$}] $\sHH{1}{\stalkcn{q}}{0}(\Hom{A}{\varphi}{I})$ and $\sHH{2}{\stalkcn{q}}{0}(\Hom{A}{\varphi}{I})$ are isomorphisms for all $I$ in $\mathscr{I}$ and $q$ in $Q$.
\end{eqc}

\item $\mathcal{F}(\mathscr{I})$ is a Frobenius exact category whose class of projective-injective objects is precisely the class 
\begin{equation*}
  \Inj{(\lMod{Q,A},\exact{Q})} \;=\; 
  \operatorname{Prod}\,\{\,\Gq{q}(I) \;|\; I \in \Inj{(\lMod{A},\exact{})},\;  q \in Q \,\} \;,
\end{equation*}
and the homotopy category $\operatorname{Ho}(\lMod{Q,A},\exact{Q})$ is equivalent to the stable category $\underline{\mathcal{F}}(\mathscr{I})$.
\end{rqm}
The functors $\check{\Hletterone}$ appearing in conditions of type \eqclbl{i} and  \eqclbl{ii} are cohomology functors in the exact category $(\lMod{Q,A},\exact{Q})$, while the functors $\mathbb{H}$ appearing in conditions of type \eqclbl{i$\,'$} and \eqclbl{ii$\,'$} are cohomology functors in the abelian category $\lMod{Q}$.
\end{thm}

\begin{proof}
Note that $\Sq{q}(I) = \Hom{\Bbbk}{\stalkcn{q}}{I}$. Now the proof is similar to that of \thmref{Q-shaped-relative-Prj}, using \thmref{main-op} instead of \thmref{main} and 
\corref{Ext-formula-exact-Inj} (with $R=\stalkcn{q}$) instead of \corref{Ext-formula-exact-Prj} (with $U=\stalkco{q}$). 
\end{proof}

\begin{exa}\label{exmp:injective abelian model}
  Let $\theta = {}_A\operatorname{Prj}$, in which case $\lMode{A}{\theta}$ and $\lMode{Q,A}{\theta}$ are just the abelian categories $\lMod{A}$ and $\lMod{Q,A}$, see \exaref{Prj-PPrj-Mod}(a). Note that the category $\lMod{A}$ has enough injectives and exact products.  Let $\mathscr{I}=\{I\}$ where $I$ is any injective cogenerator of $\lMod{A}$ and set
\begin{equation*}
  \mathscr{J} \,=\, \mathscr{J}(\mathscr{I}) 
  \,=\, \{ \Sq{q}(I) \;|\; q \in Q \} \;.
\end{equation*}
  By \corref{Ext-formula-exact-Inj} with $R=\stalkcn{q}$ we have $\HHop{i}{\Sq{q}(I)}{n}(X)
  \ \cong \
  \Hom{A}{\sHHt{i}{\stalkcn{q}}{n}(X)}{I}$. Since $I$ is an injective cogenerator, it follows that for every object $X$ and morphism $\varphi$ in $\lMod{Q,A}$ one has
  \begin{align*}
    \HHop{i}{\Sq{q}(I)}{n}(X)=0
    & \quad \iff \quad 
    \sHHt{i}{\stalkcn{q}}{n}(X)=0
    \\
    \HHop{i}{\Sq{q}(I)}{n}(\varphi) \text{ is an isomorphism}
    &\quad \iff \quad
    \sHHt{i}{\stalkcn{q}}{n}(\varphi) \text{ is an isomorphism}.
  \end{align*}
  Recall from \rmkref{comparison-of-H} that for $n=0$ we have $\sHHt{i}{\stalkcn{q}}{0}=\hHhj{i}{q}$. In view of this, the cotorsion pair $(\mathcal{W}(\mathscr{I}),\mathcal{F}(\mathscr{I}))$ cogenerated by the class $\mathscr{J}$, coincides with the cotorsion pair $(\mathbb{E},\mathbb{E}^{\perp})$ from \cite{HJ-JLMS} (see equation \eqref{E} and the discussion below it), which is proven to be complete in \cite[Thm.~5.9]{HJ-JLMS}.

  These considerations show that part (1) of \thmref{Q-shaped-relative-Inj} specialises to yield the equivalence of conditions (i), (iii), (iii') in \cite[Thm.~7.1]{HJ-JLMS} whereas part (2) of \thmref{Q-shaped-relative-Inj} specialises to yield the equivalence of (i), (iii), (iii') in \cite[Thm.~7.2]{HJ-JLMS}. Part (3) of \thmref{Q-shaped-relative-Inj} recovers the rightmost equivalence in \cite[Thm.~6.5]{HJ-JLMS}.
\end{exa}

\section{Between the \texorpdfstring{$Q$}{Q}-shaped homotopy category and the \texorpdfstring{$Q$}{Q}-shaped derived category}
\label{sec:q-shaped homotopy}

Note that exact structures in $\lMod{A}$ can be ordered by the inclusion of the classes of conflations. The minimal element for this partial order is clearly the split exact structure, while the maximal element is the abelian structure of $\lMod{A}$. In the notation of the previous section the minimal element is the exact structure in $\lMode{A}{\lMod{A}}$ and the maximal element is $\lMode{A}{\lPrj{A}}$. These exact structures give rise to a collection of partially ordered exact structures in $\lMod{Q,A}$. These range from the \textit{objectwise split} exact structure in $\lMod{Q,A}$ and the abelian exact structure in $\lMod{Q,A}$ (which we write simply as $\lMod{Q,A}$ instead of $\lMode{Q,A}{\lPrj{A}}$). We begin in a more general setting.
\subsection{Homotopical functors and Verdier localisations}
\label{homotopical-functors}
\phantom{.} \vspace*{1ex}

Let $\mathcal{A}$ be an additive category, and $\exact{}_1$ and $\exact{}_2$ two exact structures on $\mathcal{A}$. If $(\mathcal{A}_1,\exact{}_1)$ and $(\mathcal{A}_2,\exact{}_2)$ are endowed with exact model structures, recall that a functor $F\colon (\mathcal{A}_1,\exact{}_1)\longrightarrow(\mathcal{A}_2,\exact{}_2)$ is called \textbf{homotopical} if it preserves weak equivalences. It is proven in \cite[Lem.~4.3]{MR2811572} that any homotopical functor $F\colon (\mathcal{A}_1,\exact{}_1)\longrightarrow(\mathcal{A}_2,\exact{}_2)$ induces a unique functor $F_*=\operatorname{Ho}(F)\colon\operatorname{Ho}(\mathcal{A}_1,\exact{}_1)\longrightarrow\operatorname{Ho}(\mathcal{A}_2,\exact{}_2)$, between the associated homotopy categories, that makes the following diagram commute:
\begin{equation}\label{eq:Ho-Ho}
\begin{tikzcd}
	{(\mathcal{A}_1,\exact{}_1)} && {(\mathcal{A}_2,\exact{}_2)} \\
	\\
	{\operatorname{Ho}(\mathcal{A}_1,\exact{}_1)} && {\operatorname{Ho}(\mathcal{A}_2,\exact{}_2)}.
	\arrow["F", from=1-1, to=1-3]
	\arrow["{q_1}"', from=1-1, to=3-1]
	\arrow["{q_2}", from=1-3, to=3-3]
	\arrow["{F_*}", dashed, from=3-1, to=3-3]
\end{tikzcd}
\end{equation}

 Assume now that $\exact{}_2\subseteq \exact{}_1$ and that  $(\mathcal{A},\exact{}_1)$ and $(\mathcal{A},\exact{}_2)$ have enough projectives and enough injectives. For each $\ell=1,2$ suppose that there are two hereditary Hovey triples of the form $(\mathcal{C}_\ell,\mathcal{W}_\ell,\mathcal{A})$ and $(\mathcal{A},\mathcal{W}_\ell,\mathcal{F}_\ell)$ in  $(\mathcal{A},\exact{}_\ell)$, with the same trivial objects, such that $\mathcal{C}_\ell\cap\mathcal{W}_\ell=\Prj{(\mathcal{A},\exact{}_\ell)}$ and $\mathcal{W}_\ell\cap \mathcal{F}_\ell=\Inj{(\mathcal{A},\exact{}_\ell)}$.

\begin{lem}
  \label{lem:inclusions-1}\label{lem:inclusions-2}
  In the setup described above, suppose that $\mathcal{W}_1 \supseteq \mathcal{W}_2$. Then we have that 
  \begin{enumerate}
      \item $\mathcal{C}_1 \subseteq \mathcal{C}_2$ and $\mathcal{F}_1\subseteq\mathcal{F}_2$;
      \item $\mathcal{C}_1\cap \mathcal{W}_1 =\Prj{(\mathcal{A},\exact{}_1)} \subseteq \Prj{(\mathcal{A},\exact{}_2)} = \mathcal{C}_2\cap \mathcal{W}_2$ and $\mathcal{F}_1\cap\mathcal{W}_1 =\Inj{(\mathcal{A},\exact{}_1)}\subseteq\Inj{(\mathcal{A},\exact{}_2)}=\mathcal{F}_2\cap\mathcal{W}_2$;
      \item the identity functor $I \colon (\mathcal{A},\exact{}_2) \longrightarrow (\mathcal{A},\exact{}_1)$, which is exact by the assumption $\exact{}_2 \subseteq \exact{}_1$, is homotopical.
  \end{enumerate}
  
\end{lem}

\begin{proof}
    (1) For an object $X$ in $\mathcal{C}_1$ we must show that $\Ext{\exact{}_2}{1}{X}{W}=0$ for every $W$ in $\mathcal{W}_2$, i.e.~that any conflation 
    $$\xi = 0 \to W \to E \to X \to 0$$ 
    in $\mathscr{E}_2$ splits. By assumption, $\exact{}_2 \subseteq \exact{}_1$ and $\mathcal{W}_1 \supseteq \mathcal{W}_2$, so $\xi$ is also a conflation in $\exact{}_1$ and $W$ is also an object in $\mathcal{W}_1$. Since $X$ lies in $\mathcal{C}_1$, we have $\Ext{(\mathcal{A},\exact{}_1)}{1}{X}{W}=0$, and so $\xi$ splits. The proof of $\mathcal{F}_1\subseteq \mathcal{F}_2$ is dual.

 (2) The assumptions on the model exact structures guarantee that $\mathcal{C}_\ell\cap\mathcal{W}_\ell=\Prj{(\mathcal{A},\exact{}_\ell)}$. The containments $\Prj{(\mathcal{A},\exact{}_1)}\subseteq \Prj{(\mathcal{A},\exact{}_2)}$ and $\Inj{(\mathcal{A},\exact{}_1)}\subseteq\Inj{(\mathcal{A},\exact{}_2)}$ follow from the fact that $\exact{}_2\subseteq \exact{}_1$.

 (3) Let $\varphi$ be a weak equivalence in $(\mathcal{A},\exact{}_2)$. As mentioned in Subsection~\ref{exact-model-structures} this means that 
 $\varphi$ admits a factorisation $\varphi = \pi\iota$ where $\iota$ is an inflation in $\exact{}_2$ with $\Coker{\iota}$ in $\mathcal{C}_2 \cap \mathcal{W}_2 = \Prj{(\mathcal{A},\exact{}_2)}$ and $\pi$ is an deflation in $\exact{}_2$ with $\Ker{\pi}$ in $\mathcal{W}_2$. Note that $\iota$ is also an inflation in $\exact{}_1$ and $\pi$ is a deflation in $\exact{}_1$ since $\exact{}_2 \subseteq \exact{}_1$. Since $\mathcal{W}_1 \supseteq \mathcal{W}_2$ we also have that $\Ker{\pi}$ lies in $\mathcal{W}_1$. Further, we have $\mathcal{C}_2 \cap \mathcal{W}_2 \subseteq \mathcal{W}_2 \subseteq \mathcal{W}_1$, and so $\iota$ is an inflation in $\exact{}_1$ with $\Coker{\iota}$ in $\mathcal{W}_1$. This implies that $\iota$ is a weak equivalence in $(\mathcal{A},\exact{}_1)$ by \cite[proof of Thm.~3.3]{MR2811572}, see also \cite[Lem.~5.8]{Hovey02}. Since $\varphi = \pi\iota$ is a composition of weak equivalences in $(\mathcal{A},\exact{}_1)$, it is itself a weak equivalence in $(\mathcal{A},\exact{}_1)$.
\end{proof}

% The following notions are defined in \cite[Def.~1.2]{BondalKapranov}. Let $\mathscr{T}$ be a triangulated category and $\mathscr{S}$ a strictly full subcategory of $\mathscr{T}$. We call $\mathscr{S}$ a \textbf{right} (resp. left) \textbf{admissible subcategory} of $\mathscr{T}$ if every object $X$ in $\mathscr{T}$ fits into a distinguished triangle $S\rightarrow X\rightarrow S'\rightarrow\Sigma S$ with $S\in \mathscr{S}$ and $S'\in\mathscr{S}^\perp$ (resp. $S'\rightarrow X\rightarrow S\rightarrow\Sigma S$ with $S\in \mathscr{S}$ and $S'\in{}^\perp\mathscr{S}$). We call $\mathscr{S}$ an \textbf{admissible subcategory} of $\mathscr{T}$ given that it is both left and right admissible.

\begin{prp}\label{prp:Ho-verdier-abstract}
  Let $\mathcal{A}$ be an additive category, $\exact{}_2 \subseteq \exact{}_1$ be two exact structures on $\mathcal{A}$, such that $(\mathcal{A},\exact{}_\ell)$, for $\ell=1,2$ have enough projectives. Suppose that, for each $\ell=1,2$, there is a hereditary Hovey triple $(\mathcal{C}_\ell,\mathcal{W}_\ell,\mathcal{A})$ in the exact category $(\mathcal{A},\exact{}_\ell)$ with $\mathcal{C}_\ell\cap\mathcal{W}_\ell=\Prj{(\mathcal{A},\exact{}_\ell)}$, and assume that $\mathcal{W}_1 \supseteq \mathcal{W}_2$. Then the following statements hold.
  \begin{prt}
  \item  $\mathcal{W}_1$, viewed as a full subcategory of $\operatorname{Ho}(\mathcal{A},\exact{}_2)$, is a thick triangulated subcategory; 
  \item The functor
  \(
      I_* \colon \operatorname{Ho}(\mathcal{A},\exact{}_2) \longrightarrow \operatorname{Ho}(\mathcal{A},\exact{}_1)
  \)
  induced by the identity functor $I \colon (\mathcal{A},\exact{}_2) \longrightarrow (\mathcal{A},\exact{}_1)$ (which is homotopical) can be identified with the Verdier localisation functor $V \colon \operatorname{Ho}(\mathcal{A},\exact{}_2) \longrightarrow \operatorname{Ho}(\mathcal{A},\exact{}_2)/\mathcal{W}_1$;
  \item There is a colocalisation sequence of triangulated categories as follows: %Assume that $\mathcal{C}_1\cap\mathcal{W}_1\subseteq\mathcal{C}_2\cap\mathcal{W}_2$. Then, one has that the diagram
  \[\begin{tikzcd}
	{\mathcal{W}_1} && {\operatorname{Ho}(\mathcal{A},\exact{}_2)} && {\operatorname{Ho}(\mathcal{A},\exact{}_1)}
	\arrow["J"{description}, from=1-1, to=1-3]
	%\arrow["{J_r}", bend left=30pt, from=1-3, to=1-1]
	\arrow["{J_\ell}"', bend right=30pt, from=1-3, to=1-1]
	\arrow["{I_*}"{description}, from=1-3, to=1-5]
	%\arrow["{I^r}", shift left=3, from=1-5, to=1-3]
	\arrow["{I^\ell}"', bend right=30pt, from=1-5, to=1-3]
\end{tikzcd}.\]
  \end{prt}
\end{prp}

\begin{proof}
 \proofoftag{a} An object $X$ of $\mathcal{A}$) is isomorphic to zero in $\operatorname{Ho}(\mathcal{A},\exact{}_1)$ if and only if $X$ lies in $\mathcal{W}_1$. Thus, $\Ker{I_*}$ is the full subcategory $\mathcal{W}_1$ of $\operatorname{Ho}(\mathcal{A},\exact{}_2)$. Since $I_*$ is a triangle functor, it follows from \cite[Rem.~2.1.7]{Nee} that $\mathcal{W}_1$ is thick.

  \proofoftag{b} By the universal property of the Verdier localisation, see \cite[Thm.~2.1.8]{Nee}, there is a unique functor $F$ that makes the following diagram commutative:
  \begin{equation}\label{eq:diag-chase1}
     \begin{gathered}
      \xymatrix{
        \operatorname{Ho}(\mathcal{A},\exact{}_2)
        \ar[r]^-{I_*} 
        \ar[d]_-{V}
        &
        \operatorname{Ho}(\mathcal{A},\exact{}_1)
        \\
        \operatorname{Ho}(\mathcal{A},\exact{}_2)/\mathcal{W}_1 
        \ar@{.>}[ur]_-{F}
        & {}
      }
      \end{gathered}
  \end{equation}
Note that, by the proof of (a), $\Ker{F}=\{0\}$. To construct a functor in the opposite direction, consider the composite
  \begin{equation*}
      \xymatrix{
        \mathcal{A}
        \ar[r]^-{q_2} 
        &
        \operatorname{Ho}(\mathcal{A},\exact{}_2)
        \ar[r]^-{V}         
        &
        \operatorname{Ho}(\mathcal{A},\exact{}_2)/\mathcal{W}_1 
      }
  \end{equation*}
  where $q_2$ is the canonical projection functor. We claim that $Vq_2$ maps weak equivalences in $(\mathcal{A},\exact{}_1)$ to isomorphisms in $\operatorname{Ho}(\mathcal{A},\exact{}_2)/\mathcal{W}_1$. To prove this, let $\varphi \colon X \to Y$ be a weak equivalence in $(\mathcal{A},\exact{}_1)$. Consider the morphism $q_2(\varphi) \colon X \to Y$ and complete it to a distinguished triangle 
  \begin{equation}
  \label{eq:q2-triangle}
   \xymatrix{
      q_2(X) \ar[r]^-{q_2(\varphi)} & q_2(Y) \ar[r] & C \ar[r] & \Sigma q_2(X)
   }   
  \end{equation}
  in $\operatorname{Ho}(\mathcal{A},\exact{}_2)$. Applying the triangle functor $I_*$, which is the identity on objects, to this triangle and using the fact that $I_*q_2=q_1$, where $q_1$ is the canonical projection functor, see \eqref{Ho-Ho}, we get a 
  distinguished triangle in $\operatorname{Ho}(\mathcal{A},\exact{}_1)$:
  \begin{equation*}
   \xymatrix{
      q_1(X) \ar[r]^-{q_1(\varphi)} & q_1(Y) \ar[r] & I_*C \ar[r] & \Sigma q_1(X)
   }   .
  \end{equation*}
By assumption, $\varphi$ is a weak equivalence in $(\mathcal{A},\exact{}_1)$, so the latter triangle shows that $I_*C\cong 0$. Thus $FV(C)$ is isomorphic to zero and $\Ker{F}=\{0\}$, we get that $V(C)\cong 0$ in $\operatorname{Ho}(\mathcal{A},\exact{}_2)/\mathcal{W}_1$ and, thus, the triangle \eqref{q2-triangle} shows that $Vq_2(\varphi)$ is an isomorphism, as claimed. By the universal property of the homotopy category $\operatorname{Ho}(\mathcal{A},\exact{}_1)$ we get a unique triangle functor $G$ that makes the following diagram commutative:
  \begin{equation}\label{eq:diag-chase2}
      \begin{gathered}
      \xymatrix@C=2.5pc{
        \mathcal{A} 
        \ar[d]_-{q_1}
        \ar[r]^-{q_2}
        &
        \operatorname{Ho}(\mathcal{A},\exact{}_2)
        \ar[d]^-{V}
        \\
        \operatorname{Ho}(\mathcal{A},\exact{}_1)
        \ar@{.>}[r]_-{G}
        &
        \operatorname{Ho}(\mathcal{A},\exact{}_2)/\mathcal{W}_1
      }
      \end{gathered}
.\end{equation}
  It now follows from the commutativity of diagrams \eqref{diag-chase1} and \eqref{diag-chase2} that
$FGq_1\,\cong\,FVq_2\,\cong\,I_*q_2\,\cong\,q_1$
  which by the universal property of $q_1$ one has $FG\,\cong\,\mathbf{1}$. Similarly, it holds that $GFVq_2\,\cong\,GI_*q_2\,\cong\,Gq_1\,\cong\,Vq_2,$
and by the universal property of $q_2$ and $V$ we conclude that $GF\,\cong\,\mathbf{1}$. In particular, $F$ is an equivalence, which by the commutative diagram \eqref{diag-chase1} identifies $I_*$ with the Verdier localisation $V$, as claimed.

\proofoftag{c} As reviewed in Subsection \ref{exact-model-structures}, the inclusions of $\mathcal{C}_\ell$ into $\mathcal{A}$ induce equivalences of triangulated categories $\theta_\ell\colon\underline{\mathcal{C}_\ell}\rightarrow \operatorname{Ho}(\mathcal{A},\exact{}_\ell)$, with a quasi-inverse $\eta_\ell$ defined by taking $\mathcal{C}_\ell$-cofibrant replacements. Recall that the projective-injective objects in $\mathcal{C}_\ell$ are precisely the trivial cofibrant objects $\mathcal{C}_\ell\cap \mathcal{W}_\ell$. By the first two assertions of \lemref{inclusions-1}, the inclusion functor $\mathcal{C}_1\rightarrow \mathcal{C}_2$ induces therefore a well-defined functor $F\colon \underline{\mathcal{C}_1}\rightarrow \underline{\mathcal{C}_2}$.
We claim that $\theta_2\circ F\circ \eta_1$ is a left adjoint to $I_*$ or, equivalently, that $F$ is left adjoint to $G:=\eta_1\circ I_*\circ \theta_2$. Let $C_1$ be an object in $\underline{\mathcal{C}_1}$ and $C_2$ an object in $\underline{C}_2$. We will show that there are isomorphisms
\[
  \xymatrix{
  \operatorname{Hom}_{\underline{\mathcal{C}_2}}(F(C_1),C_2)  
  \ar@<0.7ex>[r]^-{\varphi}
  &
  \operatorname{Hom}_{\underline{\mathcal{C}_1}}(C_1,G(C_2))
  \ar@<0.7ex>[l]^-{\psi}
  }.
\]

Note that $F(C_1) = C_1$, by definition of $F$. Recall that there is an $\exact{}_1$-conflation
\[
  \xi:=\xymatrix{
    W_1 \ar@{>->}[r]^-{\alpha} & 
    \widetilde{C_2} \ar@{->>}[r]^-{\beta} & 
    C_2
  }
\]
with $W_1$ in $\mathcal{W}_1$ and $\widetilde{C_2}$ in $\mathcal{C}_1$. As explained above $G(C_2)$ is nothing but the image of $\widetilde{C}_2$ in the stable category $\underline{\mathcal{C}_2}$. To define $\psi$, for any morphism $g \colon C_1 \to \widetilde{C_2}$ in $\mathcal{A}$, we set $\psi([g]_1)=[\beta g]_2$, where $[\alpha]_\ell$ denotes the equivalence class of the morphism $\alpha$ is the stable category $\underline{\mathcal{C}_\ell}$. Note that $\psi$ is well-defined: if $[g]_1 = [g']_1$ then $g-g'$ factors through an object $X$ in $\mathcal{C}_1\cap\mathcal{W}_1$, which is also an object in $\mathcal{C}_2\cap\mathcal{W}_2$ by Lemma \ref{lem:inclusions-1}. Now $\beta g -\beta g' = \beta(g-g')$ also factors through the same $X$, so $[\beta g]_2 = [\beta g']_2$.

On the other hand, we define $\varphi$ as follows. Let $f \colon C_1 \to C_2$ be a morphism in $\mathcal{A}$. Applying $\Hom{\mathcal{A}}{C_1}{-}$ to the $\exact{}_1$-conflation $\xi$ yields an exact sequence:
\[
  \Hom{\mathcal{A}}{C_1}{\widetilde{C_2}} 
  \longrightarrow
  \Hom{\mathcal{A}}{C_1}{C_2}
  \longrightarrow
  \Ext{\exact{}_1}{1}{C_1}{W_1}=0
\]
so there exists $g \colon C_1 \to \widetilde{C_2}$ with $\beta g = f$. If $g'$ is another such lift, since $\beta(g-g') = f-f = 0$, the difference $g-g'$ factors through (the kernel) $W_1$ in $\mathcal{W}_1$. Since $C_1$ is cofibrant and $\widetilde{C_2}$ is fibrant in the model category $(\mathcal{A},\exact{}_1)$, by \cite[Prop.~4.4(5)]{MR2811572} $g-g'$ must factor through an object in $\mathcal{C}_1 \cap \mathcal{W}_1$, which means that $[g]_1 = [g']_1$. We conclude that there is a well-defined map $\bar{\varphi} \colon \Hom{\mathcal{A}}{C_1}{C_2} \to   \operatorname{Hom}_{\underline{\mathcal{C}_1}}(C_1,G(C_2))
$ sending $f$ to $[g]_1$ such that $\beta g=f$. We want to show that $\bar{\varphi}$ factors via the natural projection $\Hom{\mathcal{A}}{C_1}{C_2}\rightarrow \Hom{\underline{\mathcal{C}_2}}{C_1}{C_2}$. For that, we need to argue that if a morphism $f \colon C_1 \to C_2$ factors through an object in $\mathcal{C}_2 \cap \mathcal{W}_2$ then $\bar{\varphi}(f)=[g]_1=0$, i.e. $g$ factors through an object in $\mathcal{C}_1 \cap \mathcal{W}_1$. If $f \colon C_1 \to C_2$ factors through an object in $\mathcal{C}_2 \cap \mathcal{W}_2$, since $\mathcal{C}_2 \cap \mathcal{W}_2 \subseteq \mathcal{W}_2 \subseteq \mathcal{W}_1$, $f$ factors by an object in $\mathcal{W}_1$, and another application of  \cite[Prop.~4.4(5)]{MR2811572}, guarantees that $f$ even factors through an object $Z$ in $\mathcal{C}_1 \cap \mathcal{W}_1$. We write $h\colon C_1\rightarrow Z$ and $k\colon Z\rightarrow C_2$ such that $f=kh$. Since $Z$ lies in $\mathcal{C}_1$, an application of $\Hom{\mathcal{A}}{Z}{-}$ to the $\exact{}_1$-conflation $\xi$ yields (as above) a morphism $t \colon Z \to \widetilde{C_2}$ such that $\beta t = k$. The map $g-t h\colon C_1\rightarrow \widetilde{C_2}$ satisfies $\beta(g-t h) = \beta g- \beta t h = f - kh = 0$, implying that $g-t h$ factors through $\alpha$ (the kernel of $\beta$) via some morphisms $\omega\colon C_1\rightarrow W_1$.
Therefore, $g$ admits a factorisation through the object $W_1 \oplus Z$ (which lies in $\mathcal{W}_1$) as follows:
\[
  \xymatrix{
    C_1 \ar[dr]_-{
      \left(
      \begin{smallmatrix}
        \omega \\ h
      \end{smallmatrix}
      \right)
    } \ar[rr]^-{g} & & GC_2
    \\
    {} & W_1\oplus Z \ar[ur]_-{\left(
      \begin{smallmatrix}
        \alpha & \ell
      \end{smallmatrix}
      \right)} & {} 
  }
\]
A third application of \cite[Prop.~4.4(5)]{MR2811572} ensures that $g$ factors through an object in $\mathcal{C}_1 \cap \mathcal{W}_1$, as wanted.

It follows easily from the construction of $\varphi$ and $\psi$ that these are inverse maps. It is a routine check that they are also natural in $C_1$ and $C_2$, thus proving the desired adjunction. Finally, by (the dual of) \cite[Prop.~9.1.18]{Nee} or (dual of) \cite[Prop.~4.9.1]{Krause_2010}, we also get a left adjoint to the inclusion $J\colon \mathcal{W}_1 \rightarrow \operatorname{Ho}(\mathcal{A},\exact{}_2)$, completing the colocalising sequence.
\end{proof}

By working with the fibrant objects, one can prove the existence of a localisation sequence using the dual assumptions and dual arguments.

\begin{prp}\label{prp:Ho-verdier-abstract-op}
  Let $\mathcal{A}$ be an additive category, $\exact{}_2 \subseteq \exact{}_1$ be two exact structures on $\mathcal{A}$, such that $(\mathcal{A},\exact{}_\ell)$, for $\ell=1,2$ have enough injectives. Suppose that, for each $\ell=1,2$, there is a hereditary Hovey triple $(\mathcal{A},\mathcal{W}_\ell,\mathcal{F}_\ell)$ in the exact category $(\mathcal{A},\exact{}_\ell)$ with $\mathcal{W}_\ell\cap\mathcal{F}_\ell=\Inj{(\mathcal{A},\exact{}_\ell)}$, and assume that $\mathcal{W}_1 \supseteq \mathcal{W}_2$. Then the following statements hold.
  \begin{prt}
  \item  $\mathcal{W}_1$, viewed as a full subcategory of $\operatorname{Ho}(\mathcal{A},\exact{}_2)$, is a thick triangulated subcategory; 
  \item The functor
  \(
      I_* \colon \operatorname{Ho}(\mathcal{A},\exact{}_2) \longrightarrow \operatorname{Ho}(\mathcal{A},\exact{}_1)
  \)
  induced by the identity functor $I \colon (\mathcal{A},\exact{}_2) \longrightarrow (\mathcal{A},\exact{}_1)$ (which is homotopical) can be identified with the Verdier localisation functor $V \colon \operatorname{Ho}(\mathcal{A},\exact{}_2) \longrightarrow \operatorname{Ho}(\mathcal{A},\exact{}_2)/\mathcal{W}_1$;
  \item  There is a localisation sequence of triangulated categories as follows:
  \[\begin{tikzcd}
	{\mathcal{W}_1} && {\operatorname{Ho}(\mathcal{A},\exact{}_2)} && {\operatorname{Ho}(\mathcal{A},\exact{}_1)}
	\arrow["J"{description}, from=1-1, to=1-3]
	\arrow["{J_r}", bend left=30pt, from=1-3, to=1-1]
	%\arrow["{J_\ell}"', shift right=3, from=1-3, to=1-1]
	\arrow["{I_*}"{description}, from=1-3, to=1-5]
	\arrow["{I^r}", bend left=30pt, from=1-5, to=1-3]
	%\arrow["{I^\ell}"', shift right=3, from=1-5, to=1-3]
\end{tikzcd}.\]
  \end{prt}
\end{prp}

\begin{rmk}\label{rmk:notes_on_recollement}
     In case that the conditions of Proposition \ref{prp:Ho-verdier-abstract} and Proposition \ref{prp:Ho-verdier-abstract-op} are both satisfied, we get a recollement of triangulated categories:
  \[\begin{tikzcd}
	{\mathcal{W}_1} && {\operatorname{Ho}(\mathcal{A},\exact{}_2)} && {\operatorname{Ho}(\mathcal{A},\exact{}_1)}
	\arrow["J"{description}, from=1-1, to=1-3]
	\arrow["{J_r}", shift left=3, from=1-3, to=1-1]
	\arrow["{J_\ell}"', shift right=3, from=1-3, to=1-1]
	\arrow["{I_*}"{description}, from=1-3, to=1-5]
	\arrow["{I^r}", shift left=3, from=1-5, to=1-3]
	\arrow["{I^\ell}"', shift right=3, from=1-5, to=1-3]
\end{tikzcd}.\]

    \item Note that in the proof of Proposition \ref{prp:Ho-verdier-abstract} we describe explicitly the left adjoint to $I_*$ and, therefore, we can describe its essential image as the subcategory formed by those objects which are isomorphic in $\operatorname{Ho} (\mathcal{A},\exact{}_2)$ to an object in $\mathcal{C}_1$.  Dually, the essential image of the right adjoint to $I_*$ contains precisely the objects in $\operatorname{Ho}(\mathcal{A},\exact{}_2)$ which are isomorphic to an object in $\mathcal{F}_1$. 
\end{rmk}

\begin{cor}\label{cor:cof-fib-equiv}
Suppose that we are in the setup of both Propositions \ref{prp:Ho-verdier-abstract} and \ref{prp:Ho-verdier-abstract-op}, and $\Phi_2$ denote the weak equivalences in the exact model structure of $(\mathcal{A},\exact{}_2)$. Then $I_*$ induces an equivalences $$\mathcal{C}_1[\Phi_2^{-1}]\rightarrow \operatorname{Ho}(\mathcal{A},\exact{}_1)\ \ {\text and}\ \ \mathcal{F}_1[\Phi_2^{-1}]\rightarrow \operatorname{Ho}(\mathcal{A},\exact{}_1)$$
%canonical functor $I_*\colon\operatorname{Ho}(\mathcal{A},\exact{}_2)\longrightarrow\operatorname{Ho}(\mathcal{A},\exact{}_1)$ induces equivalences of categories
 %   $\mathcal{C}_1[\mathcal{Weq}^{-1}$
   %\[
    %\operatorname{Ho}^{\mathcal{C}_1}(\mathcal{A},\exact{}_2)\,\cong\, \operatorname{Ho}(\mathcal{A},\exact{}_1)\quad\text{and}\quad \operatorname{Ho}^{\mathcal{F}_1}(\mathcal{A},\exact{}_2)\,\cong\, \operatorname{Ho}(\mathcal{A},\exact{}_1)
   %\]
   % where $\operatorname{Ho}^{\mathcal{C}_1}(\mathcal{A},\exact{}_2)$ is the full subcategory of $\operatorname{Ho}(\mathcal{A},\exact{}_2)$ consisting of all the cofibrant objects with respect to the model structure $(\mathcal{C}_1,\mathcal{W}_1,\mathcal{A})$ and $\operatorname{Ho}^{\mathcal{F}_1}(\mathcal{A},\exact{}_2)$ is the full subcategory of $\operatorname{Ho}(\mathcal{A},\exact{}_2)$ consisting of all the fibrant objects with respect to the model structure $(\mathcal{A},\mathcal{W}_1,\mathcal{F}_1)$.
\end{cor}

\begin{proof}
The categories $\mathcal{C}_1[\Phi_2^{-1}]$ and $\mathcal{F}_1[\Phi_2^{-1}]$ are (non-strict!) subcategories of $\operatorname{Ho}(\mathcal{A},\exact{}_1)$ in which $I_*$ acts fully faithfully (since they are in the essential image of $I^\ell$ and $I^r$, respectively, see the remark above). Moreover, every object in $\operatorname{Ho}(\mathcal{A},\exact{}_1)$ is obtained, up to isomorphism, by applying $I_*$ to these subcategories, and therefore $I_*$ indeed induces the desired equivalences.
    %It follows directly from \prpref{Ho-verdier-abstract}, \prpref{Ho-verdier-abstract-op}, \rmkref{notes_on_recollement}, as $\mathcal{C}_1$ and $\mathcal{F}_1$ inside $\operatorname{Ho}(\mathcal{A},\exact{}_2)$, are the (non-essential) images of the fully faithful functors $I^\ell$ and $I^r$, respectively.% and \cite[Prop.~4.10.1]{Krause_2010} (and its dual statement).
\end{proof}

%\emph{Warning:} The essential images of the functors $I^\ell$ and $I^r$; that is, the closures of the images under isomorphisms in $\operatorname{Ho}(\mathcal{A},\exact{}_2)$, are ${}^\perp\mathcal{W}_1$ and $\mathcal{W}_1^\perp$ inside $\operatorname{Ho}(\mathcal{A},\exact{}_2)$. Indeed, this follows from \cite[Prop.~4.10.1]{Krause_2010} (and its dual statement). Thus, one also has equivalences:
%\[
%    \operatorname{Ho}^{{}^\perp\mathcal{W}_1}(\mathcal{A},\exact{}_2)\,\cong\, \operatorname{Ho}(\mathcal{A},\exact{}_1)\quad\text{and}\quad \operatorname{Ho}^{\mathcal{W}_1^\perp}(\mathcal{A},\exact{}_2)\,\cong\, \operatorname{Ho}(\mathcal{A},\exact{}_1)\,.
 %   \]

\subsection{The \texorpdfstring{$Q$}{Q}-shaped homotopy category}
\label{Q-shaped-homotopy-category}
\phantom{.} \vspace*{1ex}

 In Section \ref{sec:Q-shaped derived category}, we have retrieved the $Q$-shaped derived category by placing a suitably chosen exact model structure on the exact category which is maximal among those under consideration, namely on the abelian category $\lMod{Q,A}$. Here we consider an exact model structure which is minimal among those under consideration and study the interaction between the homotopy categories induced by the maximal and minimal exact structure. For this purpose, we use \thmref{Q-shaped-relative-Prj} and assume the following setup.
\begin{stp}
\label{stp:HJ+theta}
Let $Q$, $\Bbbk$, and $A$ be as in \stpref{HJ}, and let $\theta$ be a class of $A$-modules satisfying the following properties:
\begin{rqm}    
\item $\theta$ is saturated and precovering;
\item the exact category $\lMode{A}{\theta}$ is efficient;
\item the class $\mathscr{G}_\theta:=\mathscr{G}(\theta)=\{\Sq{q}(T)\,|\, T\in\theta,\,q\in Q\}$ generates a complete cotorsion pair in $\lMode{Q,A}{\theta}$.
\end{rqm}
\end{stp}

Note that \stpref{HJ+theta} guarantees that  \thmref{Q-shaped-relative-Prj} can be applied to the exact categories $(\lMod{A},\exact{}) = \lMode{A}{\theta}$ and $(\lMod{Q,A},\exact{Q}) = \lMode{Q,A}{\theta}$ with $\mathscr{T}=\theta$; indeed: Part (1) in \stpref{HJ+theta} and \lemref{precovering} imply that $\lMode{A}{\theta}$ has enough projectives. As $\lMode{A}{\theta}$ is efficient by part (2) it has, in particular, exact coproducts (recall from Subsection \ref{efficient} that this follows from \cite[Lem.~1.4]{SaorinStovicek}). The completeness of the cotorsion pair $(\mathcal{C}(\theta),\mathcal{W}(\theta))$ that appears in \thmref{Q-shaped-relative-Prj} is guaranteed by part (3) in \stpref{HJ+theta}.

\begin{lem}\label{generation of cotorsion by a set}
Let $Q$ and $A$ be as in \stpref{HJ} and assume that $\theta$ is a saturated precovering class of $A$-modules for which $\lMode{Q,A}{\theta}$ is efficient. If $\theta$ is the saturation of a set $\rho$ of $A$-modules, then the class of objects $\mathscr{G}_\theta:=\{\Sq{q}(T)\,|\, T\in\theta,\,q\in Q\}$ generates a complete cotorsion pair in $\lMode{Q,A}{\theta}$.
\end{lem}
\begin{proof}
It follows from Lemma \ref{lem:saturation}(2) that $\theta=\mathsf{Add}(A\oplus \oplus_{T\in \rho}T)$. Since $\Sq{q}$ commutes with coproducts, the cotorsion pair generated by the class $\mathscr{G}_\theta$ coincides with the one generated by the set $\{\Sq{q}(M)\colon M\in \{A\}\cup \rho\}$. Since we are working in an efficient category by assumption, it follows by Proposition \ref{prp:CW-properties} that the cotorsion pair is complete.
%We show the cotorsion pair generated by $\mathscr{G}_\theta$ is in fact generated by the set of object $\mathscr{G}_\rho:=\{\Sq{q}(T)\,|\, T\in\rho,\,q\in Q\}$. This suffices to show that it is complete (see Proposition \ref{prp:CW-properties}). We do this simply by showing that $\mathscr{G}_\theta=\mathsf{Add}{\mathscr{G}_\rho}$. Indeed, note that since $\Sq{q}$ commutes with coproducts, and since coproducts are exact in the efficient exact category $\lMode{Q,A}{\theta}$, the assignment $\Sq{q}$ commutes with coproducts (up to projective summands). Since we have $\theta=\mathsf{Add}(A\oplus \bigoplus_{T\in\rho}T)$ by \lemref{saturation}, the commutation with coproducts just mentioned implies that $\mathsf{Add}(\mathscr{G}_\theta)=\mathsf{Add}(\mathscr{G}_\rho)$ and, therefore, $\mathscr{G}_\theta^\perp=\mathscr{G}_\rho^\perp$, as wanted.
\end{proof}

We can now list some examples of classes $\theta$ satisfying the properties in the setup above.

\begin{exa}\label{exa:theta satisfying setup}
The following classes $\theta$ satisfy the conditions of Setup \ref{stp:HJ+theta}. Note that conditions (1) and (2) are established in \exaref{Prj-PPrj-Mod}, so we only show that they satisfy condition (3) as well.
\begin{enumerate}
\item $\theta=\Prj{A}$. In this case $\lMode{Q,A}{\theta}$ is the abelian category $\lMod{Q,A}$. The cotorsion pair obtained in this case is proven to be complete in Proposition \ref{prp:assumptions-ok-Prj}.
\item $\theta=\lMod{A}$. The completeness of the cotorsion pair generated by $\mathscr{G}_\theta$ is proven in Proposition \ref{prp:injective cotorsion pair is complete}.
\item $\theta=\Prj{\lMode{A}{\rho}}$, where $\rho$ is a set of finitely presented $A$-modules (see Lemma \ref{generation of cotorsion by a set}). In particular, if we take $\rho=\lsmod{A}$, then we see that $\mathscr{G}_\rho$ generates a complete cotorsion pair in the (efficient) \emph{objectwise pure exact structure} $\lMode{Q,A}{\rho}$.
\end{enumerate}
\end{exa}

In item (2) of \exaref{theta satisfying setup} we show that the conditions of Theorem \ref{thm:Q-shaped-relative-Prj} are satisfied for $\theta=\lMod{A}$. We focus our attention on this particular case. %we give a distinct name in the homotopy category obtained by this exact model structure.

\begin{dfn}\label{dfn:q-shaped-homotopy}
    Let $\theta=\lMod{A}$ in which case the exact structure on $\lMod{A}^\theta$ is the split exact structure and consequently, the exact structure on $\lMod{Q,A}^\theta$ is the objectwise split exact structure. The homotopy category of the exact model structure associated to the set $\mathscr{G}_\theta$ by Theorem \ref{thm:Q-shaped-relative-Prj} is called the $Q$\textbf{-shaped homotopy category} of $A$, and we denote by $\QSH{Q}{A}$.
\end{dfn}
The following example justifies the name in a natural way.

\begin{exa}\label{some homotopy categories}
Consider the class $\theta=\lMod{A}$. 
If we take the quiver $Q=Q^{\text{cpx}}$ discussed in the introduction, for which $\lMod{Q,A}$ is equivalent to the category of cochain complexes $\mathbf{Ch}(A)$, the exact category $\lMode{Q,A}{\theta}$ is $\mathbf{Ch}(A)$ equipped with the degreewise split exact structure. It is well-known from \cite[Chap.~I.3]{Happel} that this is a Frobenius exact structure in $\mathbf{Ch}(A)$ whose stable category is precisely the (usual) homotopy category of complexes $\mathbf{K}(A)$. If instead, we consider the case that $Q=Q^{\text{$N$-cpx}}$, the quiver that underlines the category of $N$-complexes, a generalisation of Happel's construction of the homotopy category was introduced in \cite[Thm.~4.3]{gillespie2015homotopy} by Gillespie.
    
More generally, for any $Q$ and $A$ as in Setup \ref{stp:HJ}, we have seen in \exaref{Prj-PPrj-Mod}(c) that the category $\lMode{Q,A}{\theta}$ is always Frobenius. The exact model structure associated to $\mathscr{G}_\theta$ by Theorem \ref{thm:Q-shaped-relative-Prj}
is in this case, by Propositon \ref{prp:injective cotorsion pair is complete}, the one associated to the Hovey triple $(\lMode{Q,A}{\theta},\Inj({\lMode{Q,A}{\theta}}),\lMode{Q,A}{\theta})$. This means that the homotopy category is just the stable category. Hence, in both cases presented in this example, the $Q$-shaped homotopy category coincides with the classical homotopy category of complexes and $N$-complexes, respectively.
\end{exa}

     In particular, the abstract machinery developed in \prpref{Ho-verdier-abstract} and \prpref{Ho-verdier-abstract-op} applies also for $\mathcal{A}=\lMod{Q,A}$ endowed with exact structures satisfying Setup \ref{stp:HJ+theta}. In particular, considering the classes $\theta_1=\lPrj{A}$ and $\theta_2=\lMod{A}$ in $\lMod{A}$, one retrieves a $Q$-shaped version of the result described in \cite[Ex.~4.14]{Krause_2010} for categories of complexes.

\begin{thm}\label{thm:Q-shaped-verider-loc}
    Consider the classes $\theta_1=\lPrj{A}$ and $\theta_2=\lMod{A}$ in $\lMod{A}$. The following statements hold.
    \begin{prt}
        \item The functor $I\colon \lMod{Q,A}^{\theta_2}\longrightarrow\lMod{Q,A}^{\theta_1}$ is homotopical and the induced functor $I_*\colon \QSH{Q}{A}\longrightarrow\QSD{Q}{A}$ is a Verdier localisation;
        \item The functor $I_*\colon \QSH{Q}{A}\longrightarrow\QSD{Q}{A}$ admits a fully faithful left adjoint functor $\mathbf{c}\colon\QSD{Q}{A}\longrightarrow\QSH{Q}{A}$ and a fully faithful right adjoint functor $\mathbf{f}\colon\QSD{Q}{A}\longrightarrow\QSH{Q}{A}$, which map a $Q$-shaped module to its cofibrant replacement and fibrant replacement, respectively (in the abelian exact category $\lMode{Q,A}{\theta}$);
        \item Denote by $\mathbf{K}_Q^{\text{ac}}(A)$ the trivial objects $\mathcal{W}({\theta_1})$. There is a recollement of triangulated categories of the form
    \[
    \begin{tikzcd}[column sep=6em]
	{\mathbf{K}^{\text{ac}}_Q(A)} & {\QSH{Q}{A}} & {\QSD{Q}{A}}
	\arrow["j"{description}, from=1-1, to=1-2]
	\arrow["{\ell}"', shift right=3, from=1-2, to=1-1]
 \arrow["{r}", shift left=3, from=1-2, to=1-1]
	\arrow["I_*"{description}, from=1-2, to=1-3]
	\arrow["{\mathbf{c}}"', shift right=3, from=1-3, to=1-2]
 \arrow["{\mathbf{f}}", shift left=3, from=1-3, to=1-2]
\end{tikzcd}.
    \]
    \end{prt}
\end{thm}

\begin{proof}
We apply \prpref{Ho-verdier-abstract} and \prpref{Ho-verdier-abstract-op}. In the (abelian) exact category $\lMode{Q,A}{\theta_1}$ we have, as recalled in Subsection \ref{Prelim-on-Q-shaped} the (abelian) exact model structures $({}^\perp\mathbb{E},\mathbb{E},\lMod{Q,A})$ and $(\lMod{Q,A},\mathbb{E},\mathbb{E}^\perp)$. We proved in Examples \ref{exa:Q-shaped derived category revisited} \ref{exmp:injective abelian model} that these model structures are obtained via the general machinery of Theorems \ref{thm:Q-shaped-relative-Prj} and \ref{thm:Q-shaped-relative-Inj}, respectively. In $\lMode{Q,A}{\theta_2}$, we have the exact Hovey triple $(\lMode{Q,A}{\theta_2},\Prj({\lMode{Q,A}{\theta_2}})=\Inj({\lMode{Q,A}{\theta_2}}),\lMode{Q,A}{\theta_2})$ (see Example \ref{some homotopy categories}). To prove that the assumptions of \prpref{Ho-verdier-abstract} and \prpref{Ho-verdier-abstract-op} are satisfied, we just need to show that $\Prj({\lMode{Q,A}{\theta_2}})=\Inj({\lMode{Q,A}{\theta_2}})\subseteq \mathbb{E}$. Since $\theta_1$ and $\theta_2$ are selected in a way to satisfy Setup \ref{stp:HJ+theta}, we can make use of item (2) of Theorem \ref{thm:Q-shaped-relative-Prj} with $\mathscr{T}_\ell=\theta_\ell$, $\ell=1,2$, to identify trivial objects. More concretely, an object $X$ in $\lMode{Q,A}{\theta_\ell}$ is trivial if and only if one has
 \[
 \sHH{i}{\stalkco{q}}{n}(\Hom{A}{T}{X})=0,\, \text{ for all } T \in \mathscr{T}_\ell, \, q \in Q,\, n \in \mathbb{Z},\, \text{ and }\, i>0,
 \]
 where $\sHH{i}{\stalkco{q}}{n}$ is computed in $\lMod{Q}$, independently of the exact structure on $\lMode{Q,A}{\theta_\ell}$. Since  $\mathscr{T}_1\subseteq\mathscr{T}_2$, it follows that, if $X$ is trivial in \smash{$\lMode{Q,A}{\theta_2}$}, then it is also trivial in \smash{$\lMode{Q,A}{\theta_1}$}, as wanted.
\end{proof}

\begin{cor}\label{cor: derived is cof homotopy}
Consider the classes $\theta_1=\lPrj{A}$ and $\theta_2=\lMod{A}$ in $\lMod{A}$. Let $\mathcal{C}_1$ and $\mathcal{F}_1$ denote the cofibrant objects and fibrant objects with respect to the two  abelian model structures considered in $\lMode{Q,A}{\theta_1}$. Let $\Phi_2$ denote the weak equivalences for the exact model structure in $\lMode{Q,A}{\theta_2}$. The canonical functor $I_*\colon\QSH{Q}{A}\longrightarrow\QSD{Q}{A}$ induces triangle equivalences
$$\mathbf{K}^{\text{cof}}_Q(A):=\mathcal{C}_1[\Phi_2^{-1}]\stackrel{\sim}{\longrightarrow}\QSD{Q}{A}\quad\text{ and }\quad\mathbf{K}^{\text{fib}}_Q(A):=\mathcal{F}_1[\Phi_2^{-1}]\stackrel{\sim}{\longrightarrow}\QSD{Q}{A}.$$
\end{cor}
\begin{proof}
    Immediate by Theorem \ref{thm:Q-shaped-verider-loc} and \corref{cof-fib-equiv}.
\end{proof}

We denote by $\QSH{Q}{\lPrj{A}}$ the full subcategory of $\QSH{Q}{A}$ whose objects are $\Bbbk$-linear functors $Q\longrightarrow\lPrj{A}$. Dually we define $\QSH{Q}{\lInj{A}}$.

\begin{cor}\label{cor:finite-gldim-homotopy}
    Assume that $A$ has finite left global dimension. Then the functor $I_*\colon\QSH{Q}{A}\longrightarrow\QSD{Q}{A}$ induces triangle equivalences of categories 
    $$\QSH{Q}{\lPrj{A}}\,\cong\,\QSD{Q}{A}\quad\text{ and }\quad\ \QSH{Q}{\lInj{A}}\,\cong\,\QSD{Q}{A}.$$
\end{cor}

\begin{proof}
    This follows from Corollary \ref{cor: derived is cof homotopy} and \cite[Thm.~E]{HJ-TAMS}, where they prove that for a ring $A$ with finite left global dimension, the cofibrant objects in the (projective) abelian model structure are precisely the $\Bbbk$-linear functors $Q\longrightarrow\lPrj{A}$, while the fibrant objects in the (injective) abelian model structure are precisely the $\Bbbk$-linear functors $Q\longrightarrow\lInj{A}$.
\end{proof}

\begin{exa}
\begin{enumerate}
    \item In the case that $Q=Q^{\text{cpx}}$ and $\lMod{Q,A}=\mathbf{Ch}(A)$ is the category of chain complexes of $A$-modules, Theorem \ref{thm:Q-shaped-verider-loc} recovers that the derived category $\mathbf{D}(A)$ is the Verdier quotient of the homotopy category modulo the acyclic complexes $\mathbf{K}^{\text{ac}}(A)\subset\mathbf{K}(A)$, in symbols:
    \[
    \mathbf{D}(A)= \frac{\mathbf{K}(A)}{\mathbf{K}^{\text{ac}}(A)}
    \]
    which is a classical result of Verdier \cite[Chap.~III,~Def.~1.2.2]{JLV77}. Moreover the recollement diagram
\[
\begin{tikzcd}[column sep=6em]
	{\mathbf{K}^{\text{ac}}(A)} & {\mathbf{K}(A)} & {\mathbf{D}(A)}
	\arrow["j"{description}, from=1-1, to=1-2]
	\arrow["{j^*}"', shift right=3, from=1-2, to=1-1]
	\arrow["{j^!}", shift left=3, from=1-2, to=1-1]
	\arrow["q"{description}, from=1-2, to=1-3]
	\arrow["{\mathbf{i}}", shift left=3, from=1-3, to=1-2]
	\arrow["{\mathbf{p}}"', shift right=3, from=1-3, to=1-2]
\end{tikzcd},\]
where $\textbf{i},\textbf{p}$ are the DG-injective and DG-projective resolution functors, is described in \cite[Ex.~4.14]{Krause_2010}. In this context, \corref{finite-gldim-homotopy} recovers that for a ring $A$ with finite (left) global dimension, one has $\mathbf{D}(A)=\mathbf{K}(\lPrj{A})$ (see for example the introduction of \cite{PJr05c} or \cite[Prop.~3.4]{LLAHBF91}).
   \item In the case that $Q=Q^{N\text{-cpx}}$ and $\lMod{Q,A}=\mathbf{Ch}_N(A)$ is the category of  $N$-complexes of $A$-modules, Theorem \ref{thm:Q-shaped-verider-loc} and Corollary \ref{cor: derived is cof homotopy} recover the stable t-structures described in \cite[Thm.~3.17]{MR3742439}.
\end{enumerate}
\end{exa}

\bibliographystyle{alpha}
\bibliography{refs}

@article {HZZ,
    AUTHOR = {Hu, Jiangsheng and Zhang, Dondong and Zhou, Panyue},
     TITLE = {Proper resolutions and {G}orensteinness in extriangulated
              categories},
   JOURNAL = {Front. Math. China},
  FJOURNAL = {Frontiers of Mathematics in China},
    VOLUME = {16},
      YEAR = {2021},
    NUMBER = {1},
     PAGES = {95--117},
      ISSN = {1673-3452,1673-3576},
   MRCLASS = {18G80 (18E10 18G10 18G25)},
  MRNUMBER = {MR4227180},
MRREVIEWER = {Jing\ He},
       DOI = {10.1007/s11464-021-0887-8},
       URL = {https://doi.org/10.1007/s11464-021-0887-8},
}

@article {MR4658661,
    AUTHOR = {Holm, Henrik and Odaba\c{s}\i , Sinem},
     TITLE = {The tensor embedding for a {G}rothendieck cosmos},
   JOURNAL = {Sci. China Math.},
  FJOURNAL = {Science China. Mathematics},
    VOLUME = {66},
      YEAR = {2023},
    NUMBER = {11},
     PAGES = {2471--2494},
      ISSN = {1674-7283,1869-1862},
   MRCLASS = {18E20 (18D15 18D20 18E10 18G05)},
  MRNUMBER = {MR4658661},
       DOI = {10.1007/s11425-021-2046-9},
       URL = {https://doi.org/10.1007/s11425-021-2046-9},
}

@incollection {MR4697475,
    AUTHOR = {Mathieu, Martin and Rosbotham, Michael},
     TITLE = {Schanuel's {L}emma for {E}xact {C}ategories},
 BOOKTITLE = {Multivariable {O}perator {T}heory},
     PAGES = {531--542},
 PUBLISHER = {Birkh\"{a}user/Springer, Cham},
      YEAR = {2023},
      ISBN = {978-3-031-50534-8; 9783031505355},
   MRCLASS = {99-06},
  MRNUMBER = {MR4697475},
       DOI = {10.1007/978-3-031-50535-5\_21},
       URL = {https://doi.org/10.1007/978-3-031-50535-5_21},
}

@article {MR3459032,
    AUTHOR = {Gillespie, James},
     TITLE = {Gorenstein complexes and recollements from cotorsion pairs},
   JOURNAL = {Adv. Math.},
  FJOURNAL = {Advances in Mathematics},
    VOLUME = {291},
      YEAR = {2016},
     PAGES = {859--911},
      ISSN = {0001-8708},
   MRCLASS = {18E30 (16E05)},
  MRNUMBER = {MR3459032},
MRREVIEWER = {Yuehui Zhang},
       DOI = {10.1016/j.aim.2016.01.004},
       URL = {https://doi-org.ep.fjernadgang.kb.dk/10.1016/j.aim.2016.01.004},
}

@article {MR1080854,
    AUTHOR = {Neeman, Amnon},
     TITLE = {The derived category of an exact category},
   JOURNAL = {J. Algebra},
  FJOURNAL = {Journal of Algebra},
    VOLUME = {135},
      YEAR = {1990},
    NUMBER = {2},
     PAGES = {388--394},
      ISSN = {0021-8693},
   MRCLASS = {18E30},
  MRNUMBER = {MR1080854},
MRREVIEWER = {Jeremy Rickard},
       DOI = {10.1016/0021-8693(90)90296-Z},
       URL = {https://doi-org.ep.fjernadgang.kb.dk/10.1016/0021-8693(90)90296-Z},
}

@article {Buhler,
    AUTHOR = {B{\"u}hler, Theo},
     TITLE = {Exact categories},
   JOURNAL = {Expo. Math.},
  FJOURNAL = {Expositiones Mathematicae},
    VOLUME = {28},
      YEAR = {2010},
    NUMBER = {1},
     PAGES = {1--69},
      ISSN = {0723-0869},
   MRCLASS = {18E10 (18-02 18E30)},
  MRNUMBER = {MR2606234},
MRREVIEWER = {Sunil K. Chebolu},
       DOI = {10.1016/j.exmath.2009.04.004},
       URL = {https://doi-org.ep.fjernadgang.kb.dk/10.1016/j.exmath.2009.04.004},
}

@article {MR3742439,
    AUTHOR = {Iyama, Osamu and Kato, Kiriko and Miyachi, Jun-ichi},
     TITLE = {Derived categories of {$N$}-complexes},
   JOURNAL = {J. Lond. Math. Soc. (2)},
  FJOURNAL = {Journal of the London Mathematical Society. Second Series},
    VOLUME = {96},
      YEAR = {2017},
    NUMBER = {3},
     PAGES = {687--716},
      ISSN = {0024-6107},
   MRCLASS = {18E30 (16E35 16G70 18G35 18G60)},
  MRNUMBER = {MR3742439},
MRREVIEWER = {Dag Oskar Madsen},
       DOI = {10.1112/jlms.12084},
       URL = {https://doi.org/10.1112/jlms.12084},
}

@article {HJ-JLMS,
    AUTHOR = {Holm, Henrik and J{\o}rgensen, Peter},
     TITLE = {The {$Q$}-shaped derived category of a ring},
   JOURNAL = {J. Lond. Math. Soc. (2)},
  FJOURNAL = {Journal of the London Mathematical Society. Second Series},
    VOLUME = {106},
      YEAR = {2022},
    NUMBER = {4},
     PAGES = {3263--3316},
      ISSN = {0024-6107,1469-7750},
   MRCLASS = {18G80 (16E35 18E35 18N40)},
  MRNUMBER = {MR4524199},
MRREVIEWER = {Xinhong\ Chen},
}

@book {Qui67,
    AUTHOR = {Quillen, Daniel G.},
     TITLE = {Homotopical algebra},
    SERIES = {Lecture Notes in Math.},
    VOLUME = {43},    
 PUBLISHER = {Springer-Verlag, Berlin-New York},
      YEAR = {1967},
     PAGES = {iv+156 pp. (not consecutively paged)},
   MRCLASS = {18.20 (55.00)},
  MRNUMBER = {MR0223432},
MRREVIEWER = {A. K. Bousfield},
}

@article {MR2811572,
    AUTHOR = {Gillespie, James},
     TITLE = {Model structures on exact categories},
   JOURNAL = {J. Pure Appl. Algebra},
  FJOURNAL = {Journal of Pure and Applied Algebra},
    VOLUME = {215},
      YEAR = {2011},
    NUMBER = {12},
     PAGES = {2892--2902},
      ISSN = {0022-4049},
   MRCLASS = {18E10 (18G35 55U15 55U35)},
  MRNUMBER = {MR2811572},
MRREVIEWER = {Timothy Porter},
       DOI = {10.1016/j.jpaa.2011.04.010},
       URL = {https://doi.org/10.1016/j.jpaa.2011.04.010},
}

@article {MR3719530,
    AUTHOR = {Dell'Ambrogio, Ivo and Stevenson, Greg and \v{S}\v{t}ov\'{\i}\v{c}ek, Jan},
     TITLE = {Gorenstein homological algebra and universal coefficient
              theorems},
   JOURNAL = {Math. Z.},
  FJOURNAL = {Mathematische Zeitschrift},
    VOLUME = {287},
      YEAR = {2017},
    NUMBER = {3-4},
     PAGES = {1109--1155},
      ISSN = {0025-5874},
   MRCLASS = {16E65 (18E30 19K35 46L80)},
  MRNUMBER = {MR3719530},
MRREVIEWER = {Matthew Pressland},
       DOI = {10.1007/s00209-017-1862-7},
       URL = {https://doi.org/10.1007/s00209-017-1862-7},
}

@article {MR2404296,
    AUTHOR = {Enochs, Edgar E. and Estrada, Sergio and Garc\'{\i}a Rozas, Juan Ramon},
     TITLE = {Gorenstein categories and {T}ate cohomology on projective
              schemes},
   JOURNAL = {Math. Nachr.},
  FJOURNAL = {Mathematische Nachrichten},
    VOLUME = {281},
      YEAR = {2008},
    NUMBER = {4},
     PAGES = {525--540},
      ISSN = {0025-584X},
   MRCLASS = {14F05 (14F43 16E65 16G20 18E15 18F20 18G25)},
  MRNUMBER = {MR2404296},
MRREVIEWER = {Andrei D. Halanay},
       DOI = {10.1002/mana.200510622},
       URL = {https://doi.org/10.1002/mana.200510622},
}

@article {OberstRohrl,
    AUTHOR = {Oberst, Ulrich and R\"{o}hrl, Helmut},
     TITLE = {Flat and coherent functors},
   JOURNAL = {J. Algebra},
  FJOURNAL = {Journal of Algebra},
    VOLUME = {14},
      YEAR = {1970},
     PAGES = {91--105},
      ISSN = {0021-8693},
   MRCLASS = {18.20},
  MRNUMBER = {MR257181},
MRREVIEWER = {J. L. Palmquist},
       DOI = {10.1016/0021-8693(70)90136-5},
       URL = {https://doi.org/10.1016/0021-8693(70)90136-5},
}

@preamble{
   "\def\cprime{$'$} "
}

@book {Happel,
    AUTHOR = {Happel, Dieter},
     TITLE = {Triangulated categories in the representation theory of
              finite-dimensional algebras},
    SERIES = {London Math. Soc. Lecture Note Ser.},
    VOLUME = {119},
 PUBLISHER = {Cambridge University Press, Cambridge},
      YEAR = {1988},
     PAGES = {x+208},
      ISBN = {0-521-33922-7},
   MRCLASS = {16A46 (16A48 16A62 16A64 18E30)},
  MRNUMBER = {MR935124},
MRREVIEWER = {Alfred G. Wiedemann},
       DOI = {10.1017/CBO9780511629228},
       URL = {http://dx.doi.org/10.1017/CBO9780511629228},
}

@preamble{
   "\def\soft#1{\leavevmode\setbox0=\hbox{h}\dimen7=\ht0\advance
    \dimen7 by-1ex\relax\if t#1\relax\rlap{\raise.6\dimen7
    \hbox{\kern.3ex\char'47}}#1\relax\else\if T#1\relax
    \rlap{\raise.5\dimen7\hbox{\kern1.3ex\char'47}}#1\relax
    \else\if d#1\relax\rlap{\raise.5\dimen7\hbox{\kern.9ex
    \char'47}}#1\relax\else\if D#1\relax\rlap{\raise.5\dimen7
    \hbox{\kern1.4ex\char'47}}#1\relax\else\if l#1\relax
    \rlap{\raise.5\dimen7\hbox{\kern.4ex\char'47}}#1\relax
    \else\if L#1\relax\rlap{\raise.5\dimen7\hbox{\kern.7ex
    \char'47}}#1\relax\else\message{accent \string\soft
    \space #1 not defined!}#1\relax\fi\fi\fi\fi\fi\fi} "
}

@book {Mac,
    AUTHOR = {Mac Lane, Saunders},
     TITLE = {Categories for the working mathematician},
    SERIES = {Grad. Texts in Math.},
    VOLUME = {5},
   EDITION = {Second},
 PUBLISHER = {Springer-Verlag, New York},
      YEAR = {1998},
     PAGES = {xii+314},
      ISBN = {0-387-98403-8},
   MRCLASS = {18-02},
  MRNUMBER = {MR1712872},
}

@incollection {Salce,
    AUTHOR = {Salce, Luigi},
     TITLE = {Cotorsion theories for abelian groups},
 BOOKTITLE = {Symposia {M}athematica, {V}ol. {XXIII} ({C}onf. {A}belian
              {G}roups and their {R}elationship to the {T}heory of
              {M}odules, {INDAM}, {R}ome, 1977)},
     PAGES = {11--32},
 PUBLISHER = {Academic Press, London-New York},
      YEAR = {1979},
   MRCLASS = {20K40 (18E40)},
  MRNUMBER = {MR565595},
MRREVIEWER = {P. L. Sperry},
}

@article {SaorinStovicek,
    AUTHOR = {Saor{\'{\i}}n, Manuel and {\v{S}}{\v{t}}ov{\'{\i}}{\v{c}}ek,
              Jan},
     TITLE = {On exact categories and applications to triangulated adjoints
              and model structures},
   JOURNAL = {Adv. Math.},
  FJOURNAL = {Advances in Mathematics},
    VOLUME = {228},
      YEAR = {2011},
    NUMBER = {2},
     PAGES = {968--1007},
      ISSN = {0001-8708},
   MRCLASS = {18E10},
  MRNUMBER = {MR2822215},
MRREVIEWER = {Wolfgang Rump},
       DOI = {10.1016/j.aim.2011.05.025},
       URL = {http://dx.doi.org/10.1016/j.aim.2011.05.025},
}

@incollection {Stovicek2013,
    AUTHOR = {{\v{S}}{\v{t}}ov{\'{\i}}{\v{c}}ek, Jan},
     TITLE = {Exact model categories, approximation theory, and cohomology
              of quasi-coherent sheaves},
 BOOKTITLE = {Advances in representation theory of algebras},
    SERIES = {EMS Ser. Congr. Rep.},
     PAGES = {297--367},
 PUBLISHER = {Eur. Math. Soc., Z\"urich},
      YEAR = {2013},
   MRCLASS = {18E10 (18E30 18F20)},
  MRNUMBER = {MR3220541},
MRREVIEWER = {R. H. Street},
}

@book {GobelTrlifaj,
    AUTHOR = {G{\"o}bel, R{\"u}diger and Trlifaj, Jan},
     TITLE = {Approximations and endomorphism algebras of modules},
    SERIES = {de Gruyter Exp. Math.},
    VOLUME = {41},
 PUBLISHER = {Walter de Gruyter GmbH \& Co. KG, Berlin},
      YEAR = {2006},
     PAGES = {xxiv+640},
      ISBN = {978-3-11-011079-1; 3-11-011079-2},
   MRCLASS = {16D90 (03E75 16D10 16E30 16S50 16S90)},
  MRNUMBER = {MR2251271},
MRREVIEWER = {Sverre O. Smal{\o}},
       DOI = {10.1515/9783110199727},
       URL = {http://dx.doi.org/10.1515/9783110199727},
}

@incollection {Freyd66,
    AUTHOR = {Freyd, Peter},
     TITLE = {Splitting homotopy idempotents},
 BOOKTITLE = {Proc. {C}onf. {C}ategorical {A}lgebra ({L}a {J}olla, {C}alif.,
              1965)},
     PAGES = {173--176},
 PUBLISHER = {Springer-Verlag New York, Inc., New York},
      YEAR = {1966},
   MRCLASS = {18.15 (55.40)},
  MRNUMBER = {MR206069},
MRREVIEWER = {Ronald\ Brown},
}

@article {Hovey02,
    AUTHOR = {Hovey, Mark},
     TITLE = {Cotorsion pairs, model category structures, and representation
              theory},
   JOURNAL = {Math. Z.},
  FJOURNAL = {Mathematische Zeitschrift},
    VOLUME = {241},
      YEAR = {2002},
    NUMBER = {3},
     PAGES = {553--592},
      ISSN = {0025-5874},
     CODEN = {MAZEAX},
   MRCLASS = {55U35 (18E30 18G55)},
  MRNUMBER = {MR1938704},
       DOI = {10.1007/s00209-002-0431-9},
       URL = {http://dx.doi.org/10.1007/s00209-002-0431-9},
}

@article {EEnOJn95b,
    AUTHOR = {Enochs, Edgar E. and Jenda, Overtoun M. G.},
     TITLE = {Gorenstein injective and projective modules},
   JOURNAL = {Math. Z.},
  FJOURNAL = {Mathematische Zeitschrift},
    VOLUME = {220},
      YEAR = {1995},
    NUMBER = {4},
     PAGES = {611--633},
      ISSN = {0025-5874},
     CODEN = {MAZEAX},
   MRCLASS = {16E30 (16D50)},
  MRNUMBER = {MR1363858},
MRREVIEWER = {T. W. Hungerford},
}

@article {HHl04a,
  AUTHOR =	 {Holm, Henrik},
  TITLE =	 {Gorenstein homological dimensions},
  JOURNAL =	 {J. Pure Appl. Algebra},
  FJOURNAL =	 {Journal of Pure and Applied Algebra},
  VOLUME =	 {189},
  YEAR =	 {2004},
  NUMBER =	 {1-3},
  PAGES =	 {167--193},
  ISSN =	 {0022-4049},
  CODEN =	 {JPAAA2},
  MRCLASS =	 {16E10 (16E05 16E30)},
  MRNUMBER =	 {MR2038564},
  MRREVIEWER =	 {Zhaoyong Huang},
}

@incollection {JLV77,
  AUTHOR =	 {Verdier, Jean-Louis},
  TITLE =	 {Cat\'{e}gories d\'{e}riv\'{e}es. Quelques
                  r\'{e}sultats {\upshape (}\'{e}tat 0{\upshape )}},
  BOOKTITLE =	 {SGA~$4\tfrac{1}{2}$},
  PAGES =	 {262--311},
  PUBLISHER =	 {Springer},
  ADDRESS =	 {Berlin Heidelberg New York},
  YEAR =	 {1977},
}

@article {LLAHBF91,
  AUTHOR =	 {Avramov, Luchezar L. and Foxby, Hans-Bj{\o}rn},
  TITLE =	 {Homological dimensions of unbounded complexes},
  JOURNAL =	 {J. Pure Appl. Algebra},
  FJOURNAL =	 {Journal of Pure and Applied Algebra},
  VOLUME =	 {71},
  YEAR =	 {1991},
  NUMBER =	 {2-3},
  PAGES =	 {129--155},
  ISSN =	 {0022-4049},
  CODEN =	 {JPAAA2},
  MRCLASS =	 {18G20 (13D05 18E25 18G15 55U25)},
  MRNUMBER =	 {MR1117631},
  MRREVIEWER =	 {M. H. Bijan-Zadeh},
}

@book {Nee,
  AUTHOR =	 {Neeman, Amnon},
  TITLE =	 {Triangulated categories},
  SERIES =	 {Annals of Mathematics Studies},
  VOLUME =	 {148},
  PUBLISHER =	 {Princeton University Press},
  ADDRESS =	 {Princeton, NJ},
  YEAR =	 {2001},
  PAGES =	 {viii+449},
  ISBN =	 {0-691-08685-0; 0-691-08686-9},
  MRCLASS =	 {18E30 (55-02 55N20 55U35)},
  MRNUMBER =	 {MR1812507},
  MRREVIEWER =	 {Stanis{\l}aw Betley},
}

@article {PJr05c,
  AUTHOR =	 {J{\o}rgensen, Peter},
  TITLE =	 {The homotopy category of complexes of projective
                  modules},
  JOURNAL =	 {Adv. Math.},
  FJOURNAL =	 {Advances in Mathematics},
  VOLUME =	 {193},
  YEAR =	 {2005},
  NUMBER =	 {1},
  PAGES =	 {223--232},
  ISSN =	 {0001-8708},
  CODEN =	 {ADMTA4},
  MRCLASS =	 {18E30 (16E05)},
  MRNUMBER =	 {MR2132765},
  MRREVIEWER =	 {Jun-ichi Miyachi},
}

@book {modcat,
    AUTHOR = {Hovey, Mark},
     TITLE = {Model categories},
    SERIES = {Math. Surveys Monogr.},
    VOLUME = {63},
 PUBLISHER = {American Mathematical Society},
   ADDRESS = {Providence, RI},
      YEAR = {1999},
     PAGES = {xii+209},
      ISBN = {0-8218-1359-5},
   MRCLASS = {55U35 (18D15 18G30 18G55)},
  MRNUMBER = {MR1650134},
MRREVIEWER = {Teimuraz Pirashvili},
}

@book {mta,
    AUTHOR = {Jensen, Christian U. and Lenzing, Helmut},
     TITLE = {Model-theoretic algebra with particular emphasis on fields,
              rings, modules},
    SERIES = {Algebra, Logic and Applications},
    VOLUME = {2},
 PUBLISHER = {Gordon and Breach Science Publishers},
   ADDRESS = {New York},
      YEAR = {1989},
     PAGES = {xiv+443},
      ISBN = {2-88124-717-2},
   MRCLASS = {03C60 (00A05 03-02 12L12 13L05 16-02)},
  MRNUMBER = {MR1057608},
MRREVIEWER = {Anand Pillay},
}

@book {rha,
  AUTHOR =	 {Enochs, Edgar E. and Jenda, Overtoun M. G.},
  TITLE =	 {Relative homological algebra},
  SERIES =	 {de Gruyter Exp. Math.},
  VOLUME =	 {30},
  PUBLISHER =	 {Walter de Gruyter \& Co.},
  ADDRESS =	 {Berlin},
  YEAR =	 {2000},
  PAGES =	 {xii+339},
  ISBN =	 {3-11-016633-X},
  MRCLASS =	 {16E65 (13H10 16-02 16E10 18G25)},
  MRNUMBER =	 {MR1753146},
  MRREVIEWER =	 {J. Kuzmanovich},
}

@preamble{
   "\providecommand{\arxiv}[2][AC]{\mbox{\href{http://arxiv.org/abs/#2}{\tt
    arXiv:#2 [math.#1]}}}
\providecommand{\oldarxiv}[2][AC]{\mbox{\href{http://arxiv.org/abs/math/#2}{\sf arXiv:math/#2 [math.#1]}}}"
}

@preamble{
"\providecommand{\MR}[1]{\mbox{\href{http://www.ams.org/mathscinet-getitem?mr=#1}{#1}}}
  \renewcommand{\MR}[1]{\mbox{\href{http://www.ams.org/mathscinet-getitem?mr=#1}{#1}}}"
}

@article {gillespie2015homotopy,
    AUTHOR = {Gillespie, James},
     TITLE = {The homotopy category of {$N$}-complexes is a homotopy
              category},
   JOURNAL = {J. Homotopy Relat. Struct.},
  FJOURNAL = {Journal of Homotopy and Related Structures},
    VOLUME = {10},
      YEAR = {2015},
    NUMBER = {1},
     PAGES = {93--106},
      ISSN = {2193-8407,1512-2891},
   MRCLASS = {18G35 (55U15 55U35)},
  MRNUMBER = {3313637},
MRREVIEWER = {Philippe\ Gaucher},
       DOI = {10.1007/s40062-013-0043-6},
       URL = {https://doi.org/10.1007/s40062-013-0043-6},
}

@incollection {Krause_2010,
    AUTHOR = {Krause, Henning},
     TITLE = {Localization theory for triangulated categories},
 BOOKTITLE = {Triangulated categories},
    SERIES = {London Math. Soc. Lecture Note Ser.},
    VOLUME = {375},
     PAGES = {161--235},
 PUBLISHER = {Cambridge Univ. Press, Cambridge},
      YEAR = {2010},
      ISBN = {978-0-521-74431-7},
   MRCLASS = {18E35 (18E30)},
  MRNUMBER = {2681709},
MRREVIEWER = {Jue\ Le},
       DOI = {10.1017/CBO9781139107075.005},
       URL = {https://doi.org/10.1017/CBO9781139107075.005},
}

@article {HJ-TAMS,
    AUTHOR = {Holm, Henrik and J{\o}rgensen, Peter},
     TITLE = {The {$Q$}-shaped derived category of a ring -- compact and perfect objects},
   JOURNAL = {Trans. Amer. Math. Soc.},
  FJOURNAL = {Transactions of the American Mathematical Society},
    VOLUME = {377},
      YEAR = {2024},
    NUMBER = {5},
     PAGES = {3095–3128}
}

\end{document}